\documentclass[english]{article}
\usepackage[T1]{fontenc}
\usepackage[latin9]{inputenc}
\usepackage{color}
\usepackage{amsmath}
\usepackage{amsthm}
\usepackage{amssymb}
\usepackage{bbm}
\PassOptionsToPackage{normalem}{ulem}
\usepackage{ulem}

\usepackage{hyperref}
\hypersetup{
     colorlinks   = true,
     citecolor    = blue,
     linkcolor    = blue
}

\usepackage{graphicx}
\usepackage{float}
\usepackage{comment}

\usepackage{appendix}

\makeatletter
\numberwithin{equation}{section}
\numberwithin{figure}{section}
\theoremstyle{plain}
\newtheorem{thm}{\protect\theoremname}[section]
\theoremstyle{remark}
\newtheorem{rem}[thm]{\protect\remarkname}
\theoremstyle{plain}
\newtheorem{lem}[thm]{\protect\lemmaname}
\ifx\proof\undefined
\newenvironment{proof}[1][\protect\proofname]{\par
	\normalfont\topsep6\p@\@plus6\p@\relax
	\trivlist
	\itemindent\parindent
	\item[\hskip\labelsep\scshape #1]\ignorespaces
}{%
	\endtrivlist\@endpefalse
}
\providecommand{\proofname}{Proof}
\fi
\theoremstyle{plain}
\newtheorem{cor}[thm]{\protect\corollaryname}
\theoremstyle{plain}
\newtheorem{prop}[thm]{\protect\propositionname}

\theoremstyle{definition}
\newtheorem{defn}[thm]{\protect\definitionname}

\theoremstyle{definition}
\newtheorem{example}[thm]{\protect\examplename}

\theoremstyle{plain}
\newtheorem*{theoremA}{Theorem A}
\theoremstyle{plain}
\newtheorem*{theoremB}{Theorem B}
\theoremstyle{plain}
\newtheorem*{theoremC}{Theorem C}
\theoremstyle{plain}
\newtheorem*{theoremD}{Theorem D}
\theoremstyle{plain}
\newtheorem*{theoremE}{Theorem E}

\date{}

\makeatother

\usepackage{babel}
\providecommand{\corollaryname}{Corollary}
\providecommand{\lemmaname}{Lemma}
\providecommand{\propositionname}{Proposition}
\providecommand{\remarkname}{Remark}
\providecommand{\theoremname}{Theorem}
\providecommand{\definitionname}{Definition}
\providecommand{\examplename}{Example}

\title{Optimal Tail Estimates for  Differential Equations Driven by Fractional Brownian Motion}
\author{Horatio Boedihardjo\thanks{Department of Statistics, University of Warwick, Coventry CV4 7AL,
United Kingdom. Email:horatio.boedihardjo@warwick.ac.uk.}$,\ \  $Xi Geng\thanks{School of Mathematics and Statistics, University of Melbourne,  Parkville VIC 3010, Australia.
Email: xi.geng@unimelb.edu.au.}$\ \ $and  Sheng Wang\thanks{\color{black}{International School for Advanced Studies (SISSA), Via Bonomea, 265, 34136 Trieste TS, Italy. Email: swang@sissa.it.}}}
\date{}

\begin{document}

\maketitle
\begin{abstract}
The goal of the present paper is to investigate the exact decay rate  of the tail probability $\mathbb{P}(|X_1-x_0|>R)$ for large $R$, where $X_t$ is the solution to a multidimensional stochastic differential equation driven by a fractional Brownian motion with initial condition $x_0$.  

In the first place, under the assumption of uniform ellipticity, we establish a general $(2H+1)$-Weibull lower tail estimate in Young's regime of $H\in(1/2,1)$ and a general Gaussian lower tail estimate in the rough regime of $H\in(1/4,1/2)$. These two estimates are seen to be sharp in their respective regimes. 

In the second place, we prove a striking fact by constructing explicit examples that a uniformly elliptic system with $C_b^\infty$-coefficients could \textit{fail} to have Gaussian lower tail in Young's regime. As a consequence, in our modest opinion, the multidimensional Gaussian lower estimate of \cite{BKT16} might not hold in its current form of generality without further assumptions on the vector fields.

In the third place, we provide a simple nondegeneracy condition on the vector fields, under which a Gaussian lower tail can be established in Young's regime. In addition, given the existence of the aforementioned counterexamples, we  prove another striking fact that any $2\times 2$ periodic, uniformly elliptic system always has Gaussian lower tail in Young's regime. 

Lastly, in the rough regime we establish a general $(2H+1)$-Weibull lower tail for a rich class of systems that satisfy a noncommutativity condition on the vector fields. This result provides a partially affirmative answer to a conjecture raised in \cite{BG24} on the genericity of the well-known Cass-Litterer-Lyons estimate for noncommutative systems.

\end{abstract}

\newpage

\tableofcontents{}

\newpage
\section{Introduction and summary of main results}

We consider a multidimensional stochastic differential equation (SDE)
\begin{equation}
dX_{t}=V_{0}(X_{t})dt+\sum_{i=1}^{d}V_{i}(X_{t})dB_{t}^{i},\quad X_{0}=x_{0}\in\mathbb{R}^{N}\label{eq:SDEIntro}
\end{equation}
on the unit interval $[0,1]$. Here $B_{t}=(B_{t}^{1},\cdots,B_{t}^{d})$
is a $d$-dimensional fractional Brownian motion (fBM) with Hurst
parameter $H\in(1/4,1)$ and the vector fields $V_{0},\cdots,V_{d}:\mathbb{R}^{N}\rightarrow\mathbb{R}^{N}$
are assumed to be of class $C_{b}^{\infty}$ (smooth, bounded and
with bounded derivatives of all orders). Under the framework of Lyons'
rough path theory, the SDE (\ref{eq:SDEIntro}) has a unique pathwise
solution (see \cite{CQ02,Lyo98}).

Our study in the present paper is deeply driven by the following basic
question.

\vspace{2mm}\noindent \textit{Question}. \textit{How fast does the
tail probability 
\begin{equation}
\mathbb{P}(|X_{1}-x_{0}|>R)\label{eq:TailProb}
\end{equation}
decay as $R$ goes to infinity?}

\vspace{2mm}\noindent We are mainly interested in \textit{lower estimates}
of the tail probability (\ref{eq:TailProb}) under certain nondegeneracy condition on the vector fields
(for instance,
uniform ellipticity). We first survey a few related existing results from the
literature and then discuss the main findings of the present paper.

\subsection{Existing results in the literature}

In the Brownian motion case (i.e. $H=1/2$), it is classical that
\begin{equation}
\mathbb{P}(|X_{1}-x_{0}|>R)\leqslant C_{1}e^{-C_{2}R^{2}}\label{eq:GUp}
\end{equation}
for all large $R$ with some positive constants $C_{1},C_{2}$ depending
only $d,N$ and the vector fields. This Gaussian upper estimate can
be established easily e.g. by using a martingale approach. On the other hand, suppose that $V_{1},\cdots,V_{d}$ are \textit{uniformly elliptic},
in the sense that there are positive constants $\lambda,\Lambda$, such
that 
\begin{equation}
\lambda\cdot{\rm id}_{\mathbb{R}^{N}}\leqslant V(x)V(x)^{T}\leqslant\Lambda\cdot{\rm id}_{\mathbb{R}^{N}}\quad\text{for all }x\in\mathbb{R}^{N}.\label{eq:UE}
\end{equation}
Then it can be shown that the tail probability (\ref{eq:TailProb}) admits
a matching Gaussian lower bound:
\begin{equation}
\mathbb{P}(|X_{1}-x_{0}|>R)\geqslant C_{3}e^{-C_{4}R^{2}}\label{eq:GLow}
\end{equation}
for all large $R.$ We refer the reader to \cite{KS85,KS87} for an
in-depth quantitative study of these and related questions. 

The situation becomes quite different when
$H\neq1/2$. In this case, one can no longer apply semimartingale
or Markovian techniques as in the Brownian case. We divide our discussion
into two cases: \textit{Young's regime} $H>1/2$ and the \textit{rough regime} $H<1/2$.

\subsubsection*{Young's regime }

Suppose that $H\in(1/2,1)$. By using Gaussian concentration method,
it was proved by Baudoin, Ouyang and Tindel \cite{BOT14} that the
solution $X_{1}$ also has a Gaussian upper tail, namely, the upper
estimate (\ref{eq:GUp}) continues to hold. 

On the other hand, under the assumption of uniform ellipticity, an argument 
was proposed by Besal\'u, Kohatsu-Higa and Tindel \cite{BKT16} to show that
the tail probability of $X_{1}$ admits a matching \textit{Gaussian}
lower bound (\ref{eq:GLow}).

\subsubsection*{Rough regime}

Suppose that $H\in(1/4,1/2)$. It was proved by Cass, Litterer and
Lyons \cite{CLL13} that 
\begin{equation}\label{eq:CCL}
\mathbb{P}(|X_{1}-x_{0}|>R)\leqslant C_{1}e^{-C_{2}R^{\gamma}}
\end{equation}
for all large $R$, where $\gamma$ is any exponent that is smaller than $2H+1$. As pointed out in \cite{BNOT16}, a more careful
application of Cameron-Martin embedding allows one to take $\gamma=2H+1$.
In particular, one has the following $(2H+1)$-Weibull upper estimate:
\begin{equation}\label{eq:BNOT}
\mathbb{P}(|X_{1}-x_{0}|>R)\leqslant C_{1}e^{-C_{2}R^{2H+1}}
\end{equation}
for all large $R.$ 

In the more recent work of Boedihardjo and Geng \cite{BG24}, the authors
constructed a specific SDE example whose solution admits essentially
a matching lower tail estimate. In fact, the equation considered in
\cite{BG24} was a rather special stochastic line integral example. It was shown
that 

\begin{equation}
\mathbb{P}\Big(\Big|\int_{0}^{1}\phi(B_{t}^{1})dB_{t}^{2}\Big|\Big)\geqslant C_{1}e^{-C_{2}R^{\gamma}}\label{eq:BGEx}
\end{equation}
with $\gamma$ being any exponent strictly larger than $2H+1$. Here
$\phi:\mathbb{R}\rightarrow\mathbb{R}$ is a given function whose
derivative satisfies a mild nondegeneracy condition. The result of (\ref{eq:BGEx}) shows that the Cass-Litterer-Lyons estimate is essentially sharp.

To our best knowledge, no lower tail estimates are available
for general SDEs in the rough regime. 
\begin{rem}
The results of \cite{BNOT16,BOT14,BKT16} are obtained in the stronger
form of density estimates by using the Malliavin calculus. The corresponding
tail probability estimates follow easily by integrating the density
estimates.
\end{rem}

\subsection{Main results}

We now summarise our main results in the present paper. Some of these
findings are quite striking and counterintuitive. 

\subsubsection*{General lower estimates for elliptic systems}

Our first main result is a general lower estimate for uniformly elliptic
systems (see Theorem \ref{thm:GenLowYoung} and Theorem \ref{thm:GenLowRough} respectively). 

\vspace{2mm}\noindent 
\begin{theoremA} Suppose that the vector
fields are of class $C_{b}^{\infty}$ and satisfy the uniform ellipticity
condition (\ref{eq:UE}).

\vspace{2mm}\noindent (i) Let $H\in(1/2,1).$ There exist positive
constants $C_{1},C_{2}$ such that 

\[
\mathbb{P}(|X_{1}-x_{0}|>R)\geqslant C_{1}e^{-C_{2}R^{2H+1}}
\]
for all large $R.$ 

\vspace{2mm}\noindent (ii) Let $H\in(1/4,1/2)$. There exist positive
constants $C_{1},C_{2}$ such that \begin{equation}
\mathbb{P}(|X_{1}-x_{0}|>R)\geqslant C_{1}e^{-C_{2}R^{2}}\label{eq:GenLowRoughIntro}
\end{equation}
for all large $R$. 
\end{theoremA}

In particular, the estimate (\ref{eq:GenLowRoughIntro}) generalises the main result of \cite{BKT16} to the rough regime of $H\in(1/4,1/2)$.

To summarise, one now has the following two-sided estimates
for uniformly elliptic systems:
\begin{equation}
C_{1}e^{-C_{2}R^{2H+1}}\leqslant\mathbb{P}(|X_{1}-x_{0}|>R)\leqslant C_{3}e^{-C_{4}R^{2}}\label{eq:TwoSideYoung}
\end{equation}
in Young's regime and 
\begin{equation}
C_{1}e^{-C_{2}R^{2}}\leqslant\mathbb{P}(|X_{1}-x_{0}|>R)\leqslant C_{3}e^{-C_{4}R^{2H+1}}\label{eq:TwoSideRough}
\end{equation}
in the rough regime. The estimate (\ref{eq:TwoSideRough}) is clearly
optimal, which can seen by the trivial example $X_{t}=B_{t}$ (i.e.
$d=N,$ $V_{0}=0,$ $(V_{1},\cdots,V_{d})={\rm id}$) and the line integral
example (\ref{eq:BGEx}). The upper estimate of (\ref{eq:TwoSideYoung})
is also optimal by the same example $X_{t}=B_{t}.$

\subsubsection*{Optimality of the lower estimate in Young's regime}

The first striking result of the present paper is that the lower estimate
of (\ref{eq:TwoSideYoung}) in Young's regime is also optimal.

\begin{theoremB} There are examples of uniformly
elliptic systems $X_{t}$ with $C_{b}^{\infty}$-coefficients driven
by fBM with Hurst parameter $H\in(1/2,1)$, such that 
\[
\mathbb{P}\big(|X_{1}-x_{0}|>R\big)\leqslant C_{1}e^{-C_{2}R^{2H+1}}
\]
for all large $R$.
\end{theoremB}

 As a consequence of Theorem B, a uniformly elliptic system with $C_b^\infty$-coefficients could fail to have a Gaussian lower tail. In contrast, note from Theorem A (ii) that the Gaussian lower tail is always valid in the rough regime. 

In the present paper, we construct two such counterexamples
(see Theorem \ref{thm:YoungCounterEx} and Theorem \ref{thm:2DCounter} respectively).
In our modest opinion, these two results demonstrate that the Gaussian lower estimate of \cite[Theorem 1.2]{BKT16} in the
multidimensional setting might not hold in its current form of generality without further
assumptions on the vector fields. It will be clear that what governs a Gaussian lower tail is a delicate dynamical mechanism in which uniform ellipticity does not play an essential role. Nonetheless, the main theorem of \cite{BKT16} remains
valid when $d=N=1$. 

Our third main result provides a simple sufficient condition on the
vector fields, under which a Gaussian lower estimate indeed holds (see Theorem
\ref{thm:GLowH}).

\begin{theoremC} Consider the SDE (\ref{eq:SDEIntro})
with $H\in(1/2,1)$. Suppose that Condition (H) in Section \ref{sec:LowGauCondYoung}
holds. Then there exist positive constants $C_{1},C_{2},$ such that
\[
\mathbb{P}\big(|X_{1}-x_{0}|>R\big)\geqslant C_{1}e^{-C_{2}R^{2}}
\]
for all large $R$.
\end{theoremC}

The counterexamples we constructed in Theorem B have
a radially periodic nature: the underlying vector fields are periodic
along radial directions. Given the existence of such exotic examples,
our second striking result is that the Gaussian lower estimate is
indeed always valid for \textit{coordinate-periodic}, $2\times 2$ elliptic systems (see Theorem \ref{thm:2DPeriodic}).

\begin{theoremD} Let $\{V_{1},V_{2}\}$
be smooth, uniformly elliptic vector fields on the torus $\mathbb{T}^{2}=\mathbb{R}^{2}/\mathbb{Z}^{2}$.
Let $X_{t}$ be the solution to the SDE 
\[
dX_{t}=\tilde{V}_{1}(X_{t})dB_{t}^{1}+\tilde{V}_{2}(X_{t})dB_{t}^{2},\quad X_{0}=x_{0}\in\mathbb{R}^{2},
\]
where $(B_{t}^{1},B_{t}^{2})$ is a 2D fBM with Hurst parameter $H\in(1/2,1)$
and $\tilde{V}_{i}$ is the periodic extension of $V_{i}$ to $\mathbb{R}^{2}$
($i=1,2$). Then there exist positive constants $C_{1},C_{2},$ such
that 
\[
\mathbb{P}\big(|X_{1}-x_{0}|>R\big)\geqslant C_{1}e^{-C_{2}R^{2}}
\]
for all large $R$.
\end{theoremD}

Due to the existence of the aforementioned
counterexamples, there is no hope for proving a Gaussian lower estimate
by using a general and robust argument that is solely based
on ellipticity. In fact, the underlying mechanism for Theorem C and
Theorem D has a strong flavour of dynamical systems, while ellipticity
only plays a nonessential role in the argument. It is plausible that
Theorem D fails in dimension three and higher. 

\subsubsection*{The $(2H+1)$-Weibull lower tail for noncommutative systems in the rough Regime}

It was conjectured in \cite{BG24} that the $(2H+1)$-Weibull lower
tail (\ref{eq:BGEx}) should be a generic phenomenon for noncommutative
systems (noncommutativity means nonvanishing Lie brackets in a certain
sense). In our last main result, we give a partially affirmative solution
to this conjecture (see Theorem \ref{thm:RoughLowNC}).

\begin{theoremE} Consider the SDE (\ref{eq:SDEIntro})
with $H\in(1/4,1/2)$. Suppose that Condition (NC) in Section \ref{sec:LowGenRough}
holds. Then there exist positive constants $C_{1},C_{2},$ such that
\[
\mathbb{P}\big(|X_{1}-x_{0}|>R\big)\geqslant C_{1}e^{-C_{2}R^{2H+1}}
\]
for all large $R$.
\end{theoremE}

Condition (NC) is exactly a type of noncommutativity condition on
the vector fields (while ellipticity is irrelevant here). It is robust
and is expected to be generic. As a result, Theorem E provides a partially
affirmative answer to the main question raised in \cite{BG23}. As an application of the new method developed in Section \ref{sec:LowGenRough}, we are also able to strengthen the result of \cite{BG24} to an exact $(2H+1)$-Weibull lower tail under a weaker nondegeneracy assumption (see Proposition \ref{prop:BGNew}). 

\subsection{The main dynamical mechanism}

We briefly summarise the main dynamical mechanism that governs 
the underlying tail estimates. Stated in  vague terms,
what determines the precise tail estimate is the large-scale behaviour
of the \textit{control distance function} 
\[
d_{H}(x,y):=\inf\big\{\|h\|_{\bar{{\cal H}}}:h\in\bar{{\cal H}},x\stackrel{h}{\rightsquigarrow}y\big\},
\]
where the notation $x\stackrel{h}{\rightsquigarrow}y$ means that
the ODE controlled by $h$ drives the starting location $x$ to the
terminal location $y$ at time one. 

Let $x_{0}$ be the initial condition of the SDE (\ref{eq:SDEIntro})
and let $B_{R}(x_0):=\{x\in\mathbb{R}^{N}:|x-x_0|\leqslant R\}$. The
moral is that 
\begin{equation}
d_{H}(x_{0},B_R(x_0)^c)\gtrsim R^{\alpha}\ ``\implies"\ \text{upper }(2\alpha)\text{-Weibull tail},\label{eq:LargeCMUpTail}
\end{equation}
while 
\begin{equation}
d_{H}(x_{0},B_R(x_0)^c)\lesssim R^{\alpha}\ ``\implies"\ \text{lower }(2\alpha)\text{-Weibull tail}.\label{eq:SmallCMLowTail}
\end{equation}
Of course, none of the above two statements are mathematically valid
nor are they proved in any precise form in the current paper. In fact,
the control distance estimates are deterministic, while one must establish
robust \textit{perturbative} estimates to turn the deterministic properties
into probabilistic tail estimates. 

\subsubsection*{General lower estimates in both regimes}

For a uniformly elliptic system, a typical control that drives the
dynamics to a distance of order $R$ should have Cameron-Martin norm
$\lesssim R^{H+1/2}$ in Young's regime and $\lesssim R$ in the rough
regime. This naturally leads to a lower $(2H+1)$-Weibull tail for
the SDE solution in Young's regime and a lower Gaussian tail in the
rough regime (Theorem A). This is essentially the mechanism of (\ref{eq:SmallCMLowTail}).

\subsubsection*{Young's regime: Gaussian lower tail vs $(2H+1)$-Weibull upper tail}

The mechanism behind the Gaussian lower tail (Theorem C) is simple
to describe. Geometrically, Condition (H) creates a ballistic (i.e.
nearly linear) tunnel around the integral curve controlled by one
particular direction (say $u$). The dynamics is able to propagate
a distance of order $R$ under any driving path that is close enough
to the reference control $h_{t}^{R}:=Rtu$ under $p$-variation. A
crucial point here is that the reference control $h^{R}$ has an exceptionally
small Cameron-Martin norm ($\|h^{R}\|_{\bar{{\cal H}}}=O(R)\ll R^{H+1/2}$).
The existence of a ballistic tunnel enforced by a control with small
Cameron-Martin norm leads to a heavy tail (Gaussian tail in this case).
This is again the mechanism of (\ref{eq:SmallCMLowTail}). In the
periodic setting of Theorem D, the identification of such a ballistic
tunnel turns out to be much subtler. One has to carefully examine
the deterministic dynamics of torus flows associated with the vector
fields. See Section \ref{sec:TorusStrategy} for a more detailed discussion
on the proof strategy for Theorem D. 

On the other hand, in the counterexamples given by Theorem B, the
above mechanism fails completely. What we will prove is that any control
that drives the dynamics to a distance of order $R$ must have Cameron-Martin
norm at least of order $R^{H+1/2}$. This phenomenon is caused by
the radially periodic nature of the underlying vector fields. Such
a property forces the SDE solution to have an exceptionally light
tail ($(2H+1)$-Weibull in this case). This is essentially the mechanism
of (\ref{eq:LargeCMUpTail}). 

\subsubsection*{Rough regime: $(2H+1)$-Weibull lower tail}

Under the noncommutativity condition (NC), one can construct a control
leading to $O(R)$-dynamical displacement, whose Cameron-Martin norm
is only of order $R^{H+1/2}$ (note that $R^{H+1/2}\ll R$ in this
case). The existence of such a control with small Cameron-Martin norm
results in the heavier $(2H+1)$-Weibull tail. This is again the mechanism
of (\ref{eq:SmallCMLowTail}).

\subsection{Organisation of the paper}

In Section \ref{sec:Prelim}, we recall the basic definition of fBM and its Cameron-Martin space. 

In Section \ref{sec:LowGenRough}, we prove the optimal $(2H+1)$-Weibull lower tail in the rough regime under a natural noncommutativity condition on the vector fields (Theorem \ref{thm:RoughLowNC}). As a  straightforward byproduct of the technique developed in that section, we also prove a general Gaussian lower tail under the assumption of uniform ellipticity (Theorem \ref{thm:GenLowRough}).

In Section \ref{sec:CounterEx}, we discuss two explicit examples of uniformly elliptic equations which possess $(2H)+1$-Weibull upper tail in Young's regime. The main interesting result is the genuine elliptic system provided by Theorem \ref{thm:2DCounter}. Nonetheless, the simpler example of Theorem \ref{thm:YoungCounterEx} is helpful for understanding the essential mechanism behind Theorem \ref{thm:2DCounter}.

In Section \ref{sec:LowGauCondYoung}, we establish the optimal Gaussian lower tail in Young's regime under a ballistic-type nondegeneracy condition (Theorem \ref{thm:GLowH}). We also prove a general $(2H)+1$-Weibull lower tail for uniformly elliptic systems (Theorem \ref{thm:GenLowYoung}). 

In Section \ref{sec:2DPeriodic}, we prove that every periodic, uniformly elliptic $2\times 2$ system has Gaussian lower tail in Young's regime (Theorem \ref{thm:2DPeriodic}). This is of course only striking in Young's regime, in view of the aforementioned counterexamples as well as the general Theorem \ref{thm:GenLowRough} in the rough regime. 

\newpage
\section{Fractional Brownian motion and its Cameron-Martin space}\label{sec:Prelim}

In this section, we recall the definition of a fractional Brownian
motion and several basic properties of its Cameron-Martin space that
are needed in our analysis. The main reference for this part is \cite{DU99,KMS93}.
\begin{defn}
A $d$-dimensional \textit{fractional Brownian motion} (fBM) with
Hurst parameter $H\in(0,1)$ is an $\mathbb{R}^{d}$-valued stochastic
process $B_{t}=(B_{t}^{1},\cdots,B_{t}^{d})$ ($t\geqslant0$) which
satisfies the following properties.

\vspace{2mm}\noindent (i) $B_{0}=0$ a.s.

\vspace{2mm}\noindent (ii) $B_{t}$ is a centered Gaussian process
whose covariance structure is given by 
\[
\mathbb{E}\big[B_{s}^{i}B_{t}^{j}\big]=R(s,t)\delta_{ij},\quad s,t\in[0,1],\ i,j=1,\cdots,d,
\]
where 
\[
R(s,t)=\frac{1}{2}\big(s^{2H}+t^{2H}-|s-t|^{2H}\big).
\]
(iii) $B_{t}$ has continuous sample paths a.s. 
\end{defn}
Throughout the paper, the time horizon is fixed to be $[0,1].$ Unless
otherwise stated, the fBM is always realised on the canonical path
space $({\cal W},{\cal B}({\cal W}),\mu_{H})$. Here ${\cal W}=C_{0}([0,1];\mathbb{R}^{d})$
denotes the space of continuous paths on $[0,1]$ that vanish at $t=0$.
The space ${\cal W}$ is equipped with the uniform topology and ${\cal B}({\cal W})$
denotes its Borel $\sigma$-algebra. The probability measure $\mu_{H}$
is the law of the fBM. Under $\mu_{H},$ the coordinate process 
\[
B_{t}(w):=w_{t},\quad w\in{\cal W}
\]
is an fBM. When $H=1/2,$ the process $B_{t}$ is a Brownian motion
and $\mu_{1/2}$ is the Wiener measure. 

Next, we define the Cameron-Martin space of an fBM. Let ${\cal C}_{1}$
denote the $L^{2}$-closure of 
\[
{\rm Span}\big\{ B_{t}^{i}:0\leqslant t\leqslant1,\ i=1,\cdots,d\big\}
\]
in $L^{2}({\cal W},\mu)$. 
\begin{defn}
Let $\bar{{\cal H}}$ be the space of elements $h\in{\cal W}$ which
can be expressed as 
\begin{equation}
h_{t}=\mathbb{E}[B_{t}Z],\quad0\leqslant t\leqslant1\label{eq:CMRep}
\end{equation}
for some $Z\in{\cal C}_{1}$. The space $\bar{{\cal H}}$ is equipped
with an inner product structure which is defined by 
\[
\langle h_{1},h_{2}\rangle_{\bar{{\cal H}}}:=\mathbb{E}[Z_{1}Z_{2}],\quad h_{1},h_{2}\in\bar{{\cal H}},
\]
where $Z_{i}$ is the element of ${\cal C}_{1}$ representing $h_{i}$
through (\ref{eq:CMRep}) ($i=1,2$). The Hilbert space $(\bar{{\cal H}},\langle\cdot,\cdot\rangle_{\bar{{\cal H}}})$
is called the (intrinsic) \textit{Cameron-Martin space} of the fBM. 
\end{defn}
A useful analytic tool for the computation of Cameron-Martin norm
is fractional calculus which we now recall.
\begin{defn}
Let $f$ be a function defined on $[0,1]$ and let $\alpha>0.$ Its
\textit{left fractional integral} of order $\alpha$ is defined by
\[
(I_{0^{+}}^{\alpha}f)(t):=\frac{1}{\Gamma(\alpha)}\int_{0}^{t}f(s)(t-s)^{\alpha-1}ds,\quad0\leqslant t\leqslant1.
\]
Its \textit{left fractional derivative} of order $\alpha$ is defined
by 
\[
(D_{0^{+}}^{\alpha}f)(t):=\Big(\frac{d}{dt}\Big)^{[\alpha]+1}\big(I_{0^{+}}^{1-\{\alpha\}}f\big)(t),\quad0\leqslant t\leqslant1,
\]
where $[\alpha]$ is the integer part of $\alpha$ and $\{\alpha\}:=\alpha-[\alpha]$. 
\end{defn}
\begin{rem}
If $\alpha\in(0,1)$ and $f$ is absolutely continuous, the left fractional
derivative can also be computed as
\begin{equation}\label{eq:FracDer}
(D_{0^{+}}^{\alpha}f)(t):=\frac{1}{\Gamma(1-\alpha)}\Big[\frac{f(0)}{t^{\alpha}}+\int_{0}^{t}(t-s)^{-\alpha}\dot{f}(s)ds\Big].
\end{equation}
\end{rem}

The following result of \cite[Theorem 2.1 and Theorem 3.3]{DU99}
is of fundamental importance. 
\begin{thm}
\label{thm:CMFracRep}Let $H\in(0,1).$ There is a linear isomorphism
\begin{equation}
{\cal K}_{H}:L^{2}([0,1])\rightarrow I_{0^{+}}^{H+1/2}\big(L^{2}([0,1])\big)\label{eq:KHIso}
\end{equation}
which is defined by 
\[
{\cal K}_{H}\varphi:=\begin{cases}
c_{H}\cdot I_{0^{+}}^{1}\big(t^{H-1/2}I_{0^{+}}^{H-1/2}\big(s^{1/2-H}\varphi(s)\big)(t)\big), & H\in(1/2,1);\\
c_{H}\cdot I_{0^{+}}^{2H}\big(t^{1/2-H}I_{0^{+}}^{1/2-H}\big(s^{H-1/2}\varphi(s)\big)(t)\big), & H\in(0,1/2].
\end{cases}
\]
Here $c_{H}$ is a universal constant depending only on $H$. In addition,
one has $\bar{{\cal H}}=I_{0^{+}}^{H+1/2}(L^{2}([0,1];\mathbb{R}^{d}))$
and 
\[
\|h\|_{\bar{{\cal H}}}=\big\|{\cal K}_{H}^{-1}h\big\|_{L^{2}([0,1];\mathbb{R}^{d})}
\]
for all $h\in\bar{{\cal H}}$.
\end{thm}
By using the above explicit formula for ${\cal K}_{H}$, it is not
hard to write down the kernel $K_{H}(t,s)$ ($s<t$) such
that 
\begin{equation}
({\cal K}_{H}\varphi)(t)=\int_{0}^{t}K_{H}(t,s)\varphi(s)ds\label{eq:KHRep}
\end{equation}
for any $\varphi\in L^{2}([0,1])$ (see \cite[Section 1.1]{BHOZ08}).
We will not give the expression of $K_{H}(t,s)$ since it is not important
to us. A useful observation from Theorem \ref{thm:CMFracRep} and
(\ref{eq:KHRep}) is that an fBM with Hurst parameter $H$ can be
represented as 
\begin{equation}
B_{t}=\int_{0}^{t}K_{H}(t,s)dW_{s},\label{eq:fBMRep}
\end{equation}
where $(W_{t})_{0\leqslant t\leqslant1}$ is a $d$-dimensional Brownian
motion. 

We present a lemma on the equivalence of the Cameron-Martin norm in the $H>1/2$ case. This lemma will be very useful to us and is also of independent interest for other purposes. Its proof relies on various existing results from different references and it is helpful to include the argument here for completeness. 

\begin{lem}\label{lem:EquivNorm}

Let $H\in (1/2,1).$ Then one has 
\begin{equation}
\|h\|_{\bar{{\cal H}}}^{2}\asymp_{H}\|D_{0^+}^{H-1/2}\dot{h}\|^2_{L^2([0,1])}\asymp_{H}\|\dot{h}\|_{L^{2}([0,1])}^{2}+\int_{0}^{1}\int_{0}^{1}\frac{\big|\dot{h}_{t}-\dot{h}_{s}\big|^{2}}{|t-s|^{2H}}dsdt\label{eq:EquivFracNorm}
\end{equation}
for all $h\in \bar{\mathcal{H}}$. Here $a \asymp_{H} b$ means $c_H b \leqslant a \leqslant C_H b$ for some positive constants $c_H,C_H$ depending only on $H$. 
\end{lem}

\begin{proof}
(i) The equivalence 
\[
\|h\|_{\bar{{\cal H}}}\asymp_{H}\|D_{0^{+}}^{H-1/2}\dot{h}\|_{L^{2}([0,1])}
\]
is contained in \cite[Theorem 19]{Pic10} (with obvious restriction
from $\mathbb{R}_{+}$ to $[0,1]$). The equivalence 
\begin{equation}
\|D_{0^{+}}^{H-1/2}u\|_{L^{2}([0,1])}^{2}\asymp_{H}\|u\|_{L^{2}([0,1])}^{2}+\int_{0}^{1}\int_{0}^{1}\frac{|u_{t}-u_{s}|^{2}}{|t-s|^{2H}}dsdt\label{eq:Equiv2}
\end{equation}
on $I_{0^{+}}^{H-1/2}(L^{2}([0,1]))$ requires some effort which we
now address. We set $\mu:=H-1/2$. It is important to note that $\mu\in(0,1/2)$.

\vspace{2mm}\noindent (ii) We claim that $I_{0+}^{\mu}:L^{2}([0,1])\rightarrow L^{2}([0,1])$
is a bounded linear operator, i.e. 
\begin{equation}
\|I_{0+}^{\mu}f\|_{2}\leqslant C_{\mu}\|f\|_{2}.\label{eq:CtyImu}
\end{equation}
Indeed, one can write $I_{0+}^{\mu}f=k_{\mu}*f$ where 
\[
k_{\mu}(t):=\frac{t^{\mu-1}}{\Gamma(\mu)}{\bf 1}_{(0,1)}(t).
\]
The estimate (\ref{eq:CtyImu}) follows from Young's convolution inequality
with $C_{\mu}:=\|k_{\mu}\|_{1}.$ Equivalently, one has 
\[
\|f\|_{2}\leqslant C_{\mu}\|D_{0^{+}}^{\mu}f\|_{2}
\]
for all $f\in I_{0^{+}}^{\mu}(L^{2}([0,1]))$. 

\vspace{2mm}\noindent (iii) Define the following norms (cf. \cite{ER06}):
\[
\|f\|_{J^{\mu}(0,1)}:=\|f\|_{L^{2}([0,1])}+\|D_{0^{+}}^{\mu}f\|_{L^{2}([0,1])},\quad f\in C_{c}^{\infty}((0,1));
\]
\[
\|g\|_{H^{\mu}(\mathbb{R})}:=\|g\|_{L^{2}(\mathbb{R})}+\|[\omega\mapsto|\omega|^{\mu}\hat{g}(\omega)]\|_{L^{2}(\mathbb{R})},\quad g\in C_{c}^{\infty}(\mathbb{R}),
\]
where $\hat{g}$ denotes the Fourier transform of $g$. Let $J_{0}^{\mu}(0,1)$ (respectively, $H^{\mu}(0,1)$)
be the closure of $C_{c}^{\infty}((0,1))$ under $\|\cdot\|_{J^{\mu}(0,1)}$ (respectively, $\|\cdot\|_{H^{\mu}(\mathbb{R})}$; here functions on $(0,1)$ are
extended by zero to the whole $\mathbb{R}$). According to \cite[Theorem 2.13]{ER06},
one has $J_{0}^{\mu}(0,1)=H^{\mu}(0,1)$ and the two norms $\|\cdot\|_{J^{\mu}(0,1)}$,
$\|\cdot\|_{H^{\mu}(\mathbb{R})}$ are equivalent. 

\vspace{2mm}\noindent (iv) We claim that 
\begin{equation}
J_{0}^{\mu}(0,1)=I_{0^{+}}^{\mu}(L^{2}([0,1])),\label{eq:J0JEq}
\end{equation}
where the latter space is equipped with the same norm. One only needs
to show that $I_{0^{+}}^{\mu}(L^{2}([0,1]))\subseteq J_{0}^{\mu}(0,1)$
as the other direction is trivial. To this end, let $u=I_{0^{+}}^{\mu}f$
with some $f\in L^{2}([0,1])$. Pick a sequence $f_{n}\in C_{c}^{\infty}((0,1))$
with $f_{n}\rightarrow f$ in $L^{2}$ and set $v_{n}:=I_{0^{+}}^{\mu}f_{n}$.
By (\ref{eq:CtyImu}), one has $v_{n}\rightarrow u$ under $\|\cdot\|_{J^{\mu}(0,1)}$. 

It remains to show that $v_{n}\in J_{0}^{\mu}(0,1)$. Indeed, given
small $\varepsilon>0$, let $\chi_{\varepsilon}\in C^{\infty}([0,1])$
be such that 
\[
\chi_{\varepsilon}(t)=1\quad\text{for }t\leqslant1-2\varepsilon,\quad\chi_{\varepsilon}(t)=0\quad\text{for }t\geqslant1-\varepsilon.
\]
\[0\leqslant \chi_\varepsilon\leqslant 1,\quad |\dot{\chi}_\varepsilon|\leqslant C\varepsilon^{-1}.\]
Define $v_{n,\varepsilon}:=\chi_{\varepsilon}v_{n}\in C_{c}^{\infty}((0,1))$
and set 
\[
w_{n,\varepsilon}:=v_{n}-v_{n,\varepsilon}=(1-\chi_{\varepsilon})v_{n}.
\]
It is apparent that $\|w_{n,\varepsilon}\|_{2}\rightarrow0$ as $\varepsilon\rightarrow0^{+}$
($n$ is fixed). In addition, since 
\[
\|\dot{w}_{n,\varepsilon}\|_{\infty}\leqslant\frac{C_{v_{n}}}{\varepsilon}\quad\text{and}\quad w_{n,\varepsilon}(t)=0\ \text{for }t\leqslant1-2\varepsilon,
\]
explicit calculation based on the formula (\ref{eq:FracDer}) shows
that 
\[
\|D_{0^{+}}^{\mu}w_{n,\varepsilon}\|_{2}\leqslant C_{\mu,v_{n}}\varepsilon^{1/2-\mu}.
\]
Since $\mu\in(0,1/2),$ it follows that $\|D_{0^{+}}^{\mu}w_{n,\varepsilon}\|_{2}\rightarrow0$
as $\varepsilon\rightarrow0^{+}$. Consequently, $v_{n,\varepsilon}\rightarrow v_{n}$
under $\|\cdot\|_{J^{\mu}(0,1)}$ and thus $v_{n}\in J_{0}^{\mu}(0,1)$.
One therefore concludes that $J_{0}^{\mu}(0,1)=I_{0^{+}}^{\mu}(L^{2}([0,1]))$. 

As a consequence of Parts (ii), (iii) and the relation (\ref{eq:J0JEq}), one arrives at the
following equivalence of norms 
\[
\|D_{0^{+}}^{\mu}u\|_{L^{2}([0,1])}\asymp_{H}\|\tilde{u}\|_{H^{\mu}(\mathbb{R})}
\]
for all $u\in I_{0^{+}}^{\mu}(L^{2}([0,1]))$. Here $\tilde{u}$ denotes
the extension of $u$ by zero to the whole $\mathbb{R}$. 

\vspace{2mm}\noindent (v) Define the fractional Sobolev norm (cf.
\cite{DPV12})
\[
\|g\|_{W^{\mu,2}(\mathbb{R})}^{2}:=\|g\|^2_{L^{2}(\mathbb{R})}+\int_{\mathbb{R}}\int_{\mathbb{R}}\frac{|g_{t}-g_{s}|^{2}}{|t-s|^{1+2\mu}}dsdt,\quad g\in C_{c}^{\infty}(\mathbb{R}).
\]
According to \cite[Proposition 3.4]{DPV12}, 
\[
\int_{\mathbb{R}}\int_{\mathbb{R}}\frac{|g_{t}-g_{s}|^{2}}{|t-s|^{1+2\mu}}dsdt=c_{\mu}\|[\omega\mapsto|\omega|^{\mu}\hat{g}(\omega)]\|_{L^{2}(\mathbb{R})}^{2}
\]
with some explicit constant $c_{\mu}.$ As a result, the $H^{\mu}$-
and $W^{\mu,2}$-norms are also equivalent. To complete the proof
of (\ref{eq:Equiv2}), it remains to show that 
\begin{equation}
\int_{\mathbb{R}}\int_{\mathbb{R}}\frac{|\tilde{u}_{t}-\tilde{u}_{s}|^{2}}{|t-s|^{1+2\mu}}dsdt\leqslant C_{\mu}\Big[\|u\|_{L^{2}([0,1])}^{2}+\int_{0}^{1}\int_{0}^{1}\frac{|u_{t}-u_{s}|^{2}}{|t-s|^{1+2\mu}}dsdt\Big]\label{eq:FracSobEst}
\end{equation}
for all $u\in I_{0^{+}}^{\mu}(L^{2}([0,1]))$.

To this end, one first writes 
\begin{align}
 & \int_{\mathbb{R}}\int_{\mathbb{R}}\frac{|\tilde{u}_{t}-\tilde{u}_{s}|^{2}}{|t-s|^{1+2\mu}}dsdt\nonumber \\
 & ={\cal I}(u)+2\int_{0}^{1}u_{t}^{2}\Big(\int_{-\infty}^{0}\frac{ds}{(t-s)^{1+2\mu}}+\int_{1}^{\infty}\frac{ds}{(s-t)^{1+2\mu}}\Big)dt\Big)\nonumber \\
 & ={\cal I}(u)+\frac{1}{\mu}(A(u)+B(u)),\label{eq:FracSobDecomp}
\end{align}
where 
\[
{\cal I}(u):=\int_{0}^{1}\int_{0}^{1}\frac{|u_{t}-u_{s}|^{2}}{|t-s|^{1+2\mu}}dsdt,
\]
\[
A(u):=\int_{0}^{1}t^{-2\mu}u_{t}^{2}dt,\quad B(u):=\int_{0}^{1}(1-t)^{-2\mu}u_{t}^{2}dt.
\]
We only estimate $A(u)$. The integral $B(u)$ can be handled in the
same way by a change of variables $s=1-t.$

Let $\delta\in(0,1/3)$ and $\eta>0$. Define 
\[
A_{\delta}:=\int_{\delta}^{1}t^{-2\mu}u_{t}^{2}dt.
\]
It is obvious that 
\begin{equation}
\int_{1/3}^{1}t^{-2\mu}u_{t}^{2}dt\leqslant3^{2\mu}\|u\|_{2}^{2}.\label{eq:Third1Int}
\end{equation}
We now estimate the $\int_{\delta}^{1/3}$-integral. First of all,
one has the following elementary inequality 
\begin{equation}
u_{t}^{2}\leqslant(1+\eta)u_{s}^{2}+(1+\eta^{-1})(u_{t}-u_{s})^{2}\label{eq:EleIneq}
\end{equation}
for all $s,t$. By averaging (\ref{eq:EleIneq}) over $s\in(2t,3t)$
(i.e. applying $t^{-1}\int_{2t}^{3t}ds$ to both sides) and
then integrating against $t^{-2\mu}dt$ over $(\delta,1/3)$, one
finds that 
\begin{equation}
\int_{\delta}^{1/3}t^{-2\mu}u_{t}^{2}dt\leqslant(1+\eta)I_{\delta}+(1+\eta^{-1})J_{\delta},\label{eq:DelThirdInt}
\end{equation}
where 
\begin{align*}
I_{\delta} & :=\int_{\delta}^{1/3}t^{-1-2\mu}\int_{2t}^{3t}u_{s}^{2}dsdt,\\
J_{\delta} & :=\int_{\delta}^{1/3}t^{-1-2\mu}\int_{2t}^{3t}(u_{t}-u_{s})^{2}dsdt.
\end{align*}
By interchanging the iterated integrals, one has 
\begin{equation}
I_{\delta}  \leqslant\int_{2\delta}^{1}u_{s}^{2}\Big(\int_{s/3}^{s/2}t^{-1-2\mu}dt\Big)ds =a_{\mu}\int_{2\delta}^{1}s^{-2\mu}u_{s}^{2}ds\leqslant a_{\mu}A_{\delta},\label{eq:IDelEst}
\end{equation}
where 
\begin{equation}
a_{\mu}:=\int_{2}^{3}r^{2\mu-1}dr<1\label{eq:amu}
\end{equation}
and the last inequality follows from the fact that $(2\delta,1)\subseteq(\delta,1)$.
In addition, by using the inequality 
\[
t^{-1-2\mu}\leqslant\frac{2^{1+2\mu}}{(s-t)^{1+2\mu}}\quad\text{for all }s\in(2t,3t),
\]
one easily finds that 
\begin{equation}
J_{\delta}\leqslant2^{1+2\mu}{\cal I}(u).\label{eq:JDelEst}
\end{equation}
By substituting (\ref{eq:IDelEst}, \ref{eq:JDelEst}) into (\ref{eq:DelThirdInt})
and noting (\ref{eq:Third1Int}), it follows that 
\begin{equation}
A_{\delta}\leqslant(1+\eta)a_{\mu}A_{\delta}+(1+\eta^{-1})2^{1+2\mu}{\cal I}(u)+3^{2\mu}\|u\|_{2}^{2}.\label{eq:ADelIneq}
\end{equation}

Since $a_{\mu}<1$ (see (\ref{eq:amu})), one can fix a small $\eta$
so that $(1+\eta)a_{\mu}<1$. The estimate (\ref{eq:ADelIneq}) becomes
\[
A_{\delta}\leqslant\frac{1}{1-(1+\eta)a_{\mu}}\big((1+\eta^{-1})2^{1+2\mu}{\cal I}(u)+3^{2\mu}\|u\|_{2}^{2}\big).
\]
Note the the right hand side does not depend on $\delta.$ By taking
$\delta\downarrow0$, one concludes that 
\[
A(u)\leqslant C_{\mu}\big(\|u\|_{2}^{2}+{\cal I}(u)\big).
\]
The same estimate holds for $B(u)$. By applying these inequalities to
(\ref{eq:FracSobDecomp}), one obtains the desired estimate (\ref{eq:FracSobEst}).

Now the proof of Lemma \ref{lem:EquivNorm} is complete. 
\end{proof}

We also recall a formula from \cite[Equation (13)]{NO02} that will be useful for a later estimate of the Cameron-Martin norm in the $H<1/2$ case. 

\begin{lem} Let $H\in (0,1/2)$. Then one has \begin{equation}
\big(K_{H}^{-1}h\big)(t)=\int_{0}^{t}Q(t,s)\dot{h}_{s}ds\label{eq:InvKForm}
\end{equation}
for any absolutely continuous $h$, where 
\[
Q(t,s):=c_{H}t^{H-1/2}s^{1/2-H}(t-s)^{-(1/2+H)}
\]
and $c_{H}$ is an explicit constant depending only on $H$.
\end{lem}

Finally, we discuss the SDE (\ref{eq:SDEIntro}) driven by an fBM with Hurst parameter $H\in(1/4,1)$. We always assume that the vector fields $V_{0},V_{1},\cdots,V_{d}$ are of class $C_b^\infty.$ There are two basic regularity regimes to be considered: \textit{Young's
regime} $H\in(1/2,1)$ and \textit{the rough regime }$H\in(1/4,1/2]$.
It is classical that with probability one, sample paths of fBM are
$\alpha$-H\"older continuous for any $\alpha<H$ and have finite
$p$-variation for any $p>1/H$. We use $C^{p\text{-var}}([0,1];\mathbb{R}^{d})$
to denote the space of paths $w\in{\cal W}$ with finite $p$-variation. 

\textit{Young's regime}. If $H\in(1/2,1)$, by using Young integration
theory one can define pathwise integrals of the type $\int_{s}^{t}f(x_{s})dy_{s}$
for any smooth function $f$ and any $x,y\in C^{p\text{-var}}([0,1];\mathbb{R})$.
The SDE (\ref{eq:SDEIntro}) is well-posed with the stochastic integral
being interpreted in the sense of Young (see \cite{Lyo94}).

\textit{Rough regime}. If $H\in(1/4,1/2]$, one needs rough path theory
to make sense of the SDE (\ref{eq:SDEIntro}). The moral is that with
probability one, $B_{t}$ can be canonically regarded (more precisely, enhanced) as a geometric
rough path with roughness $p\in(1/H,4)$. Under the framework of rough
path theory, the SDE (\ref{eq:SDEIntro}) is well-posed in the space
of geometric rough paths provided that the stochastic integral is
interpreted as a rough path integral (see \cite{CQ02,Lyo98}). If $H\leqslant 1/4$, one can no longer regard $B_t$ as a geometric rough path in any canonical way. We therefore restrict our study to the setting of $H > 1/4$.

\newpage
\section{General lower tail estimates in the rough regime }\label{sec:LowGenRough}

In this section, we consider the SDE (\ref{eq:SDEIntro}) in the rough regime of $H\in(1/4,1/2)$. The vector fields $V_{i}$ are assumed to be of class $C_b^\infty$. Recall from the   Cass-Litterer-Lyons estimate (\ref{eq:CCL}) (see also (\ref{eq:BNOT})) that the solution to \textit{any} system with $C_b^\infty$-coefficients always possesses a $(2H+1)$-Weibull upper tail in the rough regime of $H\in (1/4,1/2)$. In addition, the example of \cite{BG24} demonstrates the sharpness of this estimate (see (\ref{eq:BGEx})). It was conjectured in the same paper 
that the matching lower estimate (\ref{eq:BGEx})
should be a generic phenomenon for uniformly elliptic, noncommutative
systems in the rough regime.

Our aim in this section is to confirm the aforementioned conjecture
for a broad class of vector fields. As we will see, ellipticity is
not a relevant condition while noncommutativity plays an essential
role. We first introduce the following quantitative noncommutativity
condition (NC). We define the vector field
\[
V_{I}:=[V_{i_{1}},[V_{i_{2}},\cdots[V_{i_{k-1}},V_{i_{k}}]\cdots]]
\]
for any word $I=(i_{1},\cdots,i_{k})$. For $l\leqslant r,$ we use
${\cal V}_{l}^{r}$ to denote the collection of vector fields of the
form $V_{I}$, where $I=(i_{1},\cdots,i_{k})$ with $l\leqslant k\leqslant r$
and $i_{j}=1,\cdots,d.$

\vspace{2mm}\noindent (\textbf{NC}) There exist an integer $r\geqslant2$,
a positive constant $\kappa$ and a smooth function $F:\mathbb{R}^{N}\rightarrow\mathbb{R}$
with uniformly bounded derivatives $DF,D^{2}F,\cdots,D^{r+1}F$, such
that 
\begin{equation}
\inf_{x\in\mathbb{R}^{N}}\max_{I\in{\cal V}_{2}^{r}}\big|(V_{I}F)(x)\big|\geqslant\kappa.\label{eq:NC}
\end{equation}
The main result of this section is stated as follows.
\begin{thm}
\label{thm:RoughLowNC}Let $X_{t}$ be the solution to the SDE (\ref{eq:SDEIntro}) with $H\in(1/4,1/2)$
and $V_{i}\in C_{b}^{\infty}$. Suppose that the vector fields $\{V_{1},\cdots,V_{d}\}$
satisfy Condition (NC). Then there exist positive constants $C_{1},C_{2}$
depending only on $H,V$ and dimensions, such that 
\[
\mathbb{P}(|X_{1}-x_{0}|>R)\geqslant C_{1}e^{-C_{2}R^{2H+1}}
\]
for all large $R$. 
\end{thm}
\begin{rem}
For the above theorem to hold, it is critical that $|I|\geqslant2$
in Condition (NC). The condition therefore captures the noncommutativity
of the vector fields $V_{1},\cdots,V_{d}$. If one allows $|I|=1$
(i.e., including the original vector fields $V_{i}$), one can at
best obtain a Gaussian lower bound 
\[
\mathbb{P}(|X_{1}-x_{0}|>R)\geqslant C_{1}e^{-C_{2}R^{2}},
\]
which is consistent with the trivial example $(V_{1},\cdots,V_{d})={\rm id}$
($X_{t}=B_{t}$) (cf. Theorem \ref{thm:GenLowYoung}).
\end{rem}
\begin{rem}
A sufficient condition for (NC) is the following uniform nondegeneracy
condition on $V_{I}$.

\vspace{2mm}\noindent (NC') There exist an integer $r\geqslant2$,
a positive constant $\kappa$ and a unit vector $\eta\in\mathbb{R}^{N}$,
such that 
\[
\inf_{x\in\mathbb{R}^{N}}\max_{I\in{\cal V}_{2}^{r}}\big|\langle V_{I}(x),\eta\rangle_{\mathbb{R}^{N}}\big|\geqslant\kappa.
\]
In fact, Condition (NC) follows by taking $F(x)=\langle\eta,x\rangle_{\mathbb{R}^{N}}.$ 
\end{rem}

\subsection{A general criterion}\label{sec:CondG}

In this subsection, we first prove an abstract theorem under the following
more general structural Condition (G) (see Theorem \ref{thm:RoughLowG}
below). In the next subsection, we will prove that Condition (NC)
implies Condition (G). 

\vspace{2mm}\noindent (\textbf{G}) There exist a globally Lipschitz
function $F:\mathbb{R}^{N}\rightarrow\mathbb{R}$, a measurable map $q:\mathbb{R}^{N}\rightarrow W^{1,\infty}([0,1];\mathbb{R}^{d})$
(absolutely continuous paths with a.e. bounded derivatives) and positive
constants $K,a$, such that the following properties hold true:

\vspace{2mm}\noindent (G1) $q_{x}(0)=q_{x}(1)$ and $\|\dot{q}_{x}\|_{L^{\infty}}\leqslant K$;

\vspace{2mm}\noindent (G2) $F(\Psi_{q_{x}}(x))\geqslant F(x)+a$
for all $x\in\mathbb{R}^{N}$. Here $\Psi_{q}:\mathbb{R}^{N}\rightarrow\mathbb{R}^{N}$
denotes the time-one solution map for the skeleton ODE 
\[
dz_{t}=\sum_{i=1}^{d}V_{i}(z_{t})dq_{t}^{i},\quad z_{0}=x
\]
given any BV path $q:[0,1]\rightarrow\mathbb{R}^{d}$, namely, $\Psi_{q}(x):=z_{1}$. 
\begin{thm}
\label{thm:RoughLowG}Let $X_{t}$ be the solution to the SDE (\ref{eq:SDEIntro}) with  $H\in(1/4,1/2)$
and $V_{i}\in C_{b}^{\infty}$. Suppose that Condition (G) holds.
Then there exist positive constants $C_{1},C_{2}$ depending only
on $H,V,F,K,a$, such that 
\begin{equation}
\mathbb{P}(|X_{1}-x_{0}|>R)\geqslant C_{1}e^{-C_{2}R^{2H+1}}\label{eq:RoughLowG}
\end{equation}
for all large $R$. 
\end{thm}
In what follows, we develop the main ingredients towards the proof
of Theorem \ref{thm:RoughLowG}.

\subsubsection{Local continuity estimate}

This part is just a collection of known estimates from rough path
theory. To set up the lemma properly, let $I=[s,s+\delta]$ be a given
interval. Let $({\bf b}_{t})_{t\in I}$ be a given weakly geometric
$p$-rough path over $\mathbb{R}^{d}$. Recall from \cite{LQ02} that
the $p$-variation of ${\bf b}$ is defined by 
\[
\|{\bf b}\|_{p\text{-var}}:=\Big(\sum_{i=1}^{\lfloor p\rfloor}\sup_{{\cal P}}\sum_{t_{j}\in{\cal P}}\big|{\bf b}_{t_{j-1},t_{j}}^{i}\big|^{p/i}\Big)^{1/p},
\]
where ${\cal P}$ ranges over all finite partitions of $I$. Let $q^{I}:I\rightarrow\mathbb{R}^{d}$
be a BV path. For each $x\in\mathbb{R}^{N}$, we define $(Y_{t}^{x})_{t\in I}$
to be the solution to the RDE 
\begin{equation}
dY_{t}^{x}=\sum_{i=1}^{d}V_{i}(Y_{t}^{x})d(T_{q^{I}}{\bf b})^{i}+V_{0}(Y_{t}^{x})dt,\quad Y_{s}^{x}=x.\label{eq:YxIEqn}
\end{equation}
Here $T_{q^{I}}{\bf b}$ denotes the translation of ${\bf b}$ along
direction $q^{I}$ as a weakly geometric $p$-rough path (see \cite{CLL13}). 
\begin{lem}
\label{lem:LocCty}Let $K>0$ be a given number. There exists a positive
constant $C_{p,V,K}$ depending only on $p,V,K$, such that the following
estimate
\[
\big|Y_{s+\delta}^{x}-\Psi_{q^{I}}(x)\big|\leqslant C_{p,V,K}\big(\|{\bf b}\|_{p\text{-var}}+\delta\big)
\]
holds uniformly for all $s\geqslant0,\delta>0$, $x\in\mathbb{R}^{N}$,
all weakly geometric $p$-rough paths ${\bf b}$ and BV paths $q^{I}$
satisfying $\|{\bf b}\|_{p\text{-var}}\leqslant1$ and $\|q^{I}\|_{1\text{-var}}\leqslant K$.
\end{lem}
\begin{proof}
First of all, recall from \cite[Theorem 9.33]{FV10} that 
\begin{equation}
d_{p\text{-var}}\big(T_{q^{I}}{\bf b},T_{q^{I}}{\bf 1}\big)\leqslant C_{p,K}\|{\bf b}\|_{p\text{-var}}\label{eq:TranCty}
\end{equation}
uniformly for all ${\bf b}$ and $q^{I}$ satisfying $\|{\bf b}\|_{p\text{-var}}\leqslant1$
and $\|q^{I}\|_{1\text{-var}}\leqslant K$. Let $(Y_{t}^{0})_{t\in I}$
be the solution to the RDE 
\begin{equation}
d\bar{Y}_{t}^{x}=\sum_{i=1}^{d}V_{i}(\bar{Y}_{t}^{x})d(T_{q^{I}}{\bf b})^{i},\quad\bar{Y}_{s}^{x}=x.\label{eq:BarYxEqn}
\end{equation}
Respectively, let $z_{t}$ be the solution to the ODE
\[
dz_{t}^{x}=\sum_{i=1}^{d}V_{i}(z_{t}^{x})d(q_{t}^{I})^{i},\quad z_{s}^{x}=x.
\]
Note that $z_{s+\delta}^{x}=\Psi_{q^{I}}(x).$ According to the standard
RDE continuity estimate (see \cite[Theorem 10.38]{FV10}) and (\ref{eq:TranCty}),
one has 
\begin{equation}
\big|\bar{Y}_{s+\delta}^{x}-\Psi_{q^{I}}(x)\big|\leqslant C_{p,V,K}\|{\bf b}\|_{p\text{-var}}.\label{eq:BarYPsi}
\end{equation}
On the other hand, the equation (\ref{eq:YxIEqn}) of $Y_{t}^{x}$
can be written as
\begin{equation}
dY_{t}^{x}=\sum_{i=1}^{d}V_{i}(Y_{t}^{x})d(T_{q^{I}}{\bf b})^{i}+\delta V_{0}(Y_{t}^{x})d\tau_{t}^{I},\label{eq:YxEqnDel}
\end{equation}
where $\tau_{t}^{I}\triangleq\frac{t-s}{\delta}$ satisfies $\|\tau^{I}\|_{1\text{-var}}=1$.
By applying the RDE continuity estimate with drift (see \cite[Theorem 12.10]{FV10})
to the equations (\ref{eq:BarYxEqn}, \ref{eq:YxEqnDel}), one finds
that 
\begin{equation}
\big|Y_{s+\delta}^{x}-\bar{Y}_{s+\delta}^{x}\big|\leqslant C_{p,V,K}\delta.\label{eq:YBarY}
\end{equation}
The desired estimate follows by combining (\ref{eq:BarYPsi}) and
\ref{eq:YBarY}.
\end{proof}

\subsubsection{\label{subsec:LinGrowth}Linear growth under iterated loop controls}

We use $\Phi(t;s,x,{\bf b})$ ($t\geqslant s$) to denote the time-$t$
solution to the RDE 
\[
dy_{t}=\sum_{i=1}^{d}V_{i}(y_{t})d{\bf b}_{t}^{i}+V_{0}(y_{t})dt
\]
with starting time $s$ and starting location $x$.

Let $n\geqslant1$. Under Condition (G), we construct two (random)
sequences $\{X_{t_{k}}^{(n)}\}_{k=1}^{n}$ and $\{q_{k}^{(n)}\}_{k=1}^{n}$
in the following inductive way. Let $I_{k}:=[t_{k-1},t_{k}]$ where
$t_{k}:=k/n$ ($0\leqslant k\leqslant n$). Define
\[
X_{t_{1}}^{(n)}:=\Phi(t_{1};t_{0},x_{0},T_{q_{1}^{(n)}}{\bf B}|_{I_{1}}),
\]
where $q_{1}^{(n)}$ is the loop $q_{x_{0}}$ arising from Condition
(G) but reparametrised on the interval $I_{1}$, namely, 
\[
q_{1}^{(n)}(t):=q_{x_{0}}(nt),\quad t\in I_{1}.
\]
Inductively, suppose that $X_{t_{k-1}}^{(n)}$ and $q_{k-1}^{(n)}$
have been constructed. We then define 
\[
X_{t_{k}}^{(n)}:=\Phi(t_{k};t_{k-1},X_{t_{k-1}}^{(n)},T_{q_{k}^{(n)}}{\bf B}|_{I_{k}}),
\]
where $q_{k}^{(n)}$ is the loop $q_{X_{t_{k-1}}^{(n)}}$ from Condition
(G) but reparametrised on $I_{k}$, namely, 
\[
q_{k}^{(n)}(t):=q_{X_{t_{k-1}}^{(n)}}\big(nt-(k-1)\big),\quad t\in I_{k}.
\]
It is clear that 
\[
X_{t_{k}}^{(n)}=\Phi(t_{k};0,x_{0},T_{h^{(n)}}{\bf B})
\]
for all $k=1,\cdots,n,$ where 
\begin{equation}
h^{(n)}:=q_{1}^{(n)}\sqcup\cdots\sqcup q_{n}^{(n)}\label{eq:CMhn}
\end{equation}
with $\sqcup$ denoting path concatenation. 

For each $\varepsilon>0,$ we also define the event 
\begin{equation}
A_{n}(\varepsilon):=\Big\{\max_{1\leqslant k\leqslant n}\|{\bf B}\|_{p\text{-var};I_{k}}\leqslant\varepsilon\Big|.\label{eq:AnEps}
\end{equation}
Here $p\in(2,4)$ is a fixed number such that $pH>1.$
\begin{lem}
\label{lem:LinDisplace}Under Condition (G), there exist constants
$\Lambda,\varepsilon_{0}$ $n_{0}$ depending only on $p,V,F,K,a$,
such that the following estimate 
\[
\big|X_{1}^{(n)}-x_{0}\big|\geqslant\Lambda n
\]
holds on the event $A_{n}(\varepsilon)$, uniformly for all $n>n_{0}$
and $\varepsilon<\varepsilon_{0}.$ 
\end{lem}
\begin{proof}
By the previous construction and Property (G2), one has 
\[
F\Big(\Psi_{q_{k}^{(n)}}(X_{t_{k-1}}^{(n)})\Big)\geqslant F\big(X_{t_{k-1}}^{(n)}\big)+a
\]
for every $k$. On the other hand, since $F$ is globally Lipschitz
(say, with Lipschitz constant $C_{F}$) and $\|q_{k}^{(n)}\|_{1\text{-var}}\leqslant K$
by assumption, one knows from Lemma \ref{lem:LocCty} that 
\[
\Big|F\big(X_{t_{k}}^{(n)}\big)-F\Big(\Psi_{q_{k}^{(n)}}(X_{t_{k-1}}^{(n)})\Big)\Big|\leqslant C_{F}C_{p,V,K}\Big(\|{\bf B}\|_{p\text{-var};I_{k}}+\frac{1}{n}\Big),
\]
provided that $\|{\bf B}\|_{p\text{-var};I_{k}}\leqslant1.$ It follows
that 
\begin{align*}
F\big(X_{t_{k}}^{(n)}\big) & \geqslant F\big(X_{t_{k-1}}^{(n)}\big)+a-C_{p,V,K,F}\Big(\|{\bf B}\|_{p\text{-var};I_{k}}+\frac{1}{n}\Big),
\end{align*}
provided that $\|{\bf B}\|_{p\text{-var};I_{k}}\leqslant1.$ Now one
chooses
\[
\varepsilon_{0}\triangleq\frac{a}{4C_{p,V,K,F}}\wedge1,\quad n_{0}\triangleq\frac{4C_{p,V,K,F}}{a}.
\]
Then 
\[
F\big(X_{t_{k}}^{(n)}\big)\geqslant F\big(X_{t_{k-1}}^{(n)}\big)+\frac{a}{2}
\]
holds on the event $A_{n}(\varepsilon)$ uniformly for all $n>n_{0},$
$\varepsilon<\varepsilon_{0}$ and $k=1,\cdots,n$. The conclusion
of the lemma thus follows by summing over $k$ and applying the global
Lipschitz property of $F$. 
\end{proof}

\subsubsection{A Cameron-Martin norm estimate}

In this part, we regard the previously constructed $h^{(n)}$ (see
(\ref{eq:CMhn})) as a Cameron-Martin path in $\bar{{\cal H}}$ and
estimate its Cameron-Martin norm. We first state a general lemma. 
\begin{lem}
\label{lem:CMEst}Let $K>0$ be a given number. There exists a positive
constant $C_{H,K}$, such that the following estimate 
\begin{equation}\label{eq:CmEstRoughLow}
\|h^{(n)}\|_{\bar{{\cal H}}}\leqslant C_{H}Kn^{H+1/2}
\end{equation}
holds uniformly for all $n\geqslant1$ and all absolutely continuous
paths $h^{(n)}$ satisfying $\|\dot{h}^{(n)}\|_{L^{\infty}}\leqslant Kn$
and $h_{t_{k-1}}^{^{(n)}}=h_{t_{k}}^{(n)}$ for every $k=1,\cdots,n.$
\end{lem}
\begin{proof}

Now let $h^{(n)}$ be a given absolutely continuous path such that
$\|\dot{h}^{(n)}\|_{L^{\infty}}\leqslant Kn$ and $h_{t_{k-1}}^{^{(n)}}=h_{t_{k}}^{(n)}$
for every $k=1,\cdots,n.$ To prove the lemma, it suffices to show
that 
\begin{equation}
\big|\big(K_{H}^{-1}h^{(n)}\big)(t)\big|\leqslant C_{H}Kn^{H+1/2}\label{eq:InvKEst}
\end{equation}
for all $t\in[0,1]$. To this end, let $t\in I_{j}$ for some $j$.
We decompose the $s$-integral on the RHS of (\ref{eq:InvKForm})
into four parts:
\[
\int_{I_{1}}+\sum_{k=2}^{j-2}\int_{I_{k}}+\int_{I_{j-1}}+\int_{t_{j-1}}^{t}
\]
We estimate each part separately. 

Firstly, one has by assumption that
\begin{align*}
\Big|\int_{t_{j-1}}^{t}Q(t,s)\dot{h}_{s}^{(n)}ds\Big| & \leqslant C_{H}Kn\int_{t_{j-1}}^{t}(t-s)^{-(1/2+H)}(s/t)^{1/2-H}ds\\
 & \leqslant C_{H}Kn\int_{t_{j-1}}^{t}(t-s)^{-(1/2+H)}ds\\
 & \leqslant C_{H}Kn(t-t_{j-1})^{1/2-H}\leqslant C_{H}Kn^{H+1/2}.
\end{align*}
The integrals $\int_{I_{1}}$ and $\int_{I_{j-1}}$ are estimated
in the same way with the same $n^{H+1/2}$-growth. 

Next, consider the integral $\int_{I_{k}}$ with $2\leqslant k\leqslant j-2$.
Since $h_{t_{k-1}}^{(n)}=h_{t_{k}}^{(n)}$, one has 
\begin{align*}
&\Big|\int_{I_{k}}Q(t,s)\dot{h}_{s}^{(n)}ds\Big|\\
 & = \Big|\int_{I_{k}}\big(Q(t,s)-Q(t,t_{k})\big)\dot{h}_{s}^{(n)}ds\Big|\\
& \leqslant  C_{H}Kn\cdot\Big(\frac{j-1}{n}\Big)^{H-1/2}\cdot\Big[\Big(\frac{k-1}{n}\Big)^{-(1/2+H)}\Big(\frac{j-1-k}{n}\Big)^{-(1/2+H)}\\
 & \ \ \ +\Big(\frac{k}{n}\Big)^{1/2-H}\Big(\frac{j-1-k}{n}\Big)^{-(3/2+H)}\Big]\cdot\frac{1}{n^{2}},
\end{align*}
where the last inequality follows by estimating $\partial_{s}Q(t,s)$
in the obvious way. The exponent of $n$ on the RHS is precisely $H+1/2$.
It remains to see that 
\begin{equation}
(j-1)^{H-1/2}\sum_{k=2}^{j-2}(k-1)^{-(1/2+H)}(j-1-k)^{-(1/2+H)}\leqslant C_{H},\label{eq:kSumBound1}
\end{equation}
\begin{equation}
(j-1)^{H-1/2}\sum_{k=2}^{j-2}k^{1/2-H}(j-1-k)^{-(3/2+H)}\leqslant C_{H}\label{eq:kSumBound2}
\end{equation}
uniformly for all $j$. The inequality (\ref{eq:kSumBound1}) follows
by observing that the LHS is a Riemann sum approximation of the Beta
integral 
\[
\int_{0}^{1}x^{-(1/2+H)}(1-x)^{-(1/2+H)}dx.
\]
The inequality (\ref{eq:kSumBound1}) follows by observing that the
LHS is bounded above by 
\[
\sum_{k=2}^{j-2}(j-1-k)^{-(3/2+H)}\leqslant\sum_{l=1}^{\infty}l^{-(3/2+H)}=:C_{H}<\infty.
\]
Therefore, one concludes that 
\[
\Big|\sum_{k=2}^{j-2}\int_{I_{k}}Q(t,s)\dot{h}_{s}^{(n)}ds\Big|\leqslant C_{H}Kn^{H+1/2}.
\]

The desired claim (\ref{eq:InvKEst}) is easily obtained by combining
the above estimates. This completes the proof of the lemma.
\end{proof}
\begin{cor}
\label{cor:CMEsthn}Let $h^{(n)}$ be the path constructed by (\ref{eq:CMhn})
in Section \ref{subsec:LinGrowth}. Then one has 
\[
\|h^{(n)}\|_{\bar{{\cal H}}}\leqslant C_{H}Kn^{H+1/2}
\]
for all $n$.
\end{cor}
\begin{proof}
One simply observes that $\|\dot{h}^{(n)}\|_{L^{\infty}}\leqslant Kn$
(because $\|\dot{q}_{k}^{(n)}\|_{L^{\infty}}\leqslant Kn$ by Condition
(G) and the definition of $q_{k}^{(n)}$) and $h_{t_{k-1}}^{(n)}=h_{t_{k}}^{(n)}$
for every $k$ (because all the $q_{k}^{(n)}$'s are loops by assumption).
\end{proof}

\subsubsection{Completing the proof of Theorem \ref{thm:RoughLowG}: Girsanov's transform}

Recall that $A_{n}(\varepsilon)$ is the event defined by (\ref{eq:AnEps}). 
\begin{lem}
\label{lem:Fernique}There exist positive constants $C_{1},C_{2}$
depending only on $H,$ such that 
\[
\mathbb{P}(A_{n}(\varepsilon)^{c})\leqslant C_{1}ne^{-C_{2}\varepsilon^{2}n^{2H}}
\]
for all $n\geqslant1$ and $\varepsilon\in(0,1)$.
\end{lem}
\begin{proof}
By definition, 
\[
\mathbb{P}(A_{n}(\varepsilon)^{c})=\mathbb{P}\big(\|{\bf B}\|_{p\text{-var};I_{k}}>\varepsilon\ \text{for some }k\big)\leqslant n\mathbb{P}\big(\|{\bf B}\|_{p\text{-var};[0,1]}>n^{H}\varepsilon\big),
\]
where we used a simple union bound and fBM scaling to reach the last
inequality. The result then follows from Fernique's estimate for Gaussian
rough paths (cf. \cite[Theorem 11.9]{FH14}).
\end{proof}
We are now in a position to prove Theorem \ref{thm:RoughLowG}. The
argument is a simple application of Girsanov's theorem together with
the exponential martingale inequality.

\begin{proof}[Proof of Theorem \ref{thm:RoughLowG}]

Let $h^{(n)}\in\bar{{\cal H}}$ be constructed as before and set $\theta^{(n)}:=K_{H}^{-1}h^{(n)}$.
It is clear from construction that $\theta^{(n)}$ is ${\cal F}^{W}$-adapted,
where $W$ is the BM in the fBM representation. Consider the martingale
\[
M_{t}^{(n)}:=\int_{0}^{t}\langle\theta_{s}^{(n)},dW_{s}\rangle_{\mathbb{R}^{d}}.
\]
According to Corollary \ref{cor:CMEsthn}, one has 
\[
\langle M^{(n)}\rangle_{1}=\|\theta^{(n)}\|_{L^{2}([0,1])}^{2}\leqslant C_{H}^{2}K^{2}n^{2H+1}=:\Gamma_{n}
\]
for all $n$. 

Define the martingale 
\[
Z_{t}^{(n)}\triangleq\exp\Big(-\int_{0}^{t}\langle\theta_{s}^{(n)},dW_{s}\rangle_{\mathbb{R}^{d}}-\frac{1}{2}\int_{0}^{t}\big|\theta_{s}^{(n)}\big|^{2}ds\Big),\quad0\leqslant t\leqslant1.
\]
According to Girsanov's theorem, under the new probability measure
$d\mathbb{Q}^{(n)}:=Z_{1}^{(n)}d\mathbb{P}$, the process 
\[
\tilde{W}_{t}^{(n)}:=W_{t}+\int_{0}^{t}\theta_{s}^{(n)}ds
\]
is a BM and hence $\tilde{B}_{t}^{(n)}:=B_{t}+h_{t}^{(n)}$ is an
fBM. Let $X_{t}$ be the solution to the SDE (\ref{eq:SDEIntro})
driven by ${\bf B}$ under $\mathbb{P}.$ Then one has
\begin{align*}
  \mathbb{P}\big(|X_{1}-x_{0}|\geqslant\Lambda n\big)
 & =\mathbb{Q}^{(n)}\big(\big|\Phi(1;0,x_{0},\tilde{B}^{(n)})-x_{0}\big|\geqslant\Lambda n\big)\\
 & =\mathbb{E}^{\mathbb{P}}\big[Z_{1}^{(n)}{\bf 1}_{\{|\Phi(1;0,x_{0},\tilde{B}^{(n)})-x_{0}|\geqslant\Lambda n\}}\big]\\
 & \geqslant\mathbb{E}^{\mathbb{P}}\big[Z_{1}^{(n)}{\bf 1}_{A_{n}(\varepsilon)}\big]
\end{align*}
for all $n>n_{0},$ $\varepsilon<\varepsilon_{0}$, where we used
Lemma \ref{lem:LinDisplace} to reach the last inequality with $\Lambda$
being the constant in that Lemma.

Now let 
\[
G_{n}\triangleq\big\{ M_{1}^{(n)}\leqslant\Gamma_{n}\big\}=\big\{ M_{1}^{(n)}\leqslant\Gamma_{n},\ \langle M^{(n)}\rangle_{1}\leqslant\Gamma_{n}\big\}.
\]
On the event $G_{n},$ one knows that $Z_{1}^{(n)}\geqslant e^{-3\Gamma_{n}/2}.$
In addition, by the exponential martingale inequality, one has 
\[
\mathbb{P}(G_{n}^{c})=\mathbb{P}\big(M_{1}^{(n)}>\Gamma_{n},\ \langle M^{(n)}\rangle_{1}\leqslant\Gamma_{n}\big)\leqslant2e^{-\Gamma_{n}/2}.
\]
It follows from Lemma \ref{lem:Fernique} that 
\begin{align*}
\mathbb{P}\big(A_{n}(\varepsilon)\cap G_{n}\big) & =\mathbb{P}\big(A_{n}(\varepsilon)\big)-\mathbb{P}\big(A_{n}(\varepsilon)\cap G_{n}^{c}\big)\\
 & \geqslant1-C_{1}ne^{-C_{2}\varepsilon^{2}n^{2H}}-2e^{-\Gamma_{n}/2}.
\end{align*}
We fix $\varepsilon<\varepsilon_{0}.$ There exists $n_{1}>n_{0},$
such that 
\[
\mathbb{P}\big(A_{n}(\varepsilon)\cap G_{n}\big)\geqslant\frac{1}{2}
\]
for all $n>n_{1}.$ As a consequence, one concludes that 
\[
\mathbb{P}\big(|X_{1}-x_{0}|\geqslant\Lambda n\big)\geqslant\mathbb{E}^{\mathbb{P}}\big[Z_{1}^{(n)}{\bf 1}_{A_{n}(\varepsilon)\cap G_{n}}\big]\geqslant\frac{1}{2}e^{-3\Gamma_{n}/2}=\frac{1}{2}e^{-3C_{H}^{2}K^{2}n^{2H+1}/2}
\]
for all $n>n_{1}$. 

The desired estimate (\ref{eq:RoughLowG}) follows by renaming $R:=\Lambda n$.

\end{proof}

\subsection{Condition (NC) implies Condition (G)}

In this section, we prove Theorem \ref{thm:RoughLowNC} by showing
that Condition (NC) implies Condition (G). The main result is stated
as follows. 
\begin{prop}
\label{prop:NCImpliesG}Let $\{V_{1},\cdots,V_{d}\}$ be a family
of $C_{b}^{\infty}$-vector fields that satisfy Condition (NC). Then
Condition (G) holds.
\end{prop}
Before discussing its proof, we first recall some basic notions from
\cite{FV10} on path signatures. The algebra of formal tensor series
over $\mathbb{R}^{d}$ (respectively, the step-$m$ truncated one
for any given $m\geqslant1$) is denoted as $T((\mathbb{R}^{d}))$
(respectively, $T^{(m)}(\mathbb{R}^{d})$). The map $\pi^{(m)}:T((\mathbb{R}^{d}))\rightarrow T^{(m)}(\mathbb{R}^{d})$
denotes the canonical projection. The exponential and logarithmic
transforms over $T((\mathbb{R}^{d}))$ are denoted as $\exp$ and
$\log$ respectively. The \textit{signature} of a BV path $\gamma:[a,b]\rightarrow\mathbb{R}^{d}$
is the element of $T((\mathbb{R}^{d}))$ defined by 
\[
S(\gamma)=\Big(1,\gamma_{b}-\gamma_{a},\int_{a<s<t<b}d\gamma_{s}\otimes d\gamma_{t},\cdots,\int_{a<t_{1}<\cdots<t_{n}<b}d\gamma_{t_{1}}\otimes\cdots\otimes d\gamma_{t_{n}},\cdots\Big).
\]
The signature satisfies the so-called \textit{Chen's identity}, i.e.
\[
S(\alpha\sqcup\beta)=S(\alpha)\otimes S(\beta)
\]
where $\sqcup$ means path concatenation. If $\gamma$ is the line
segment represented by a vector $v,$ one has $S(\gamma)=\exp(v).$
Given a word $I=(i_{1},\cdots,i_{k})$ with $1\leqslant i_{j}\leqslant d$,
we define 
\[
e_{I}:=[e_{i_{1}},[e_{i_{2}},\cdots[e_{i_{k-1}},e_{i_{k}}]\cdots]],
\]
where $\{e_{1},\cdots,e_{d}\}$ is the canonical basis of $\mathbb{R}^{d}$
and $[a,b]:=a\otimes b-b\otimes a.$ The set of all non-empty words
over $\{1,\cdots,d\}$ is denoted as ${\cal W}.$

The proof of Proposition \ref{prop:NCImpliesG} relies crucially on
the following construction, which is standard from rough path
theory.
\begin{lem}
\label{lem:CanPath}There exists a family $\{\gamma^{I}:I\in{\cal W}\}$
of piecewise linear paths such that the following properties hold
true:

\vspace{2mm}\noindent (i) $\pi^{(|I|)}S(\gamma^{I})={\bf 1}+e_{I}$
for every $I\in{\cal W}$;

\vspace{2mm}\noindent (ii) for each $m\geqslant1$, there exists
a positive constant $L_{m}$ depending only on $m$, such that 
\[
\|\gamma^{I}\|_{1\text{-var}}\leqslant L_{m}
\]
for all $I\in{\cal W}$ with $|I|\leqslant m$.
\end{lem}
\begin{proof}
We prove the lemma by induction on the length of the word. If $I=(i)$,
one simply defines $\gamma_{t}^{I}=te_{i}$ ($0\leqslant t\leqslant1$).
Suppose that $\gamma^{J}$ has been constructed to satisfy the required
properties for every word $J$ with $|J|<m$. Let $I$ be a word with
length $m$ and write $I=(i,J)$. One then defines 
\[
\gamma^{I}:=\gamma^{(i)}\sqcup\gamma^{J}\sqcup\overleftarrow{\gamma^{(i)}}\sqcup\overleftarrow{\gamma^{J}},
\]
where $\overleftarrow{\cdot}$ denotes the reversal of a path. By
the induction hypothesis, one can write 
\[
\pi^{(m)}\log S(\gamma^{J})=e_{J}+\xi
\]
where $\xi$ is some homogeneous Lie polynomial $\xi$ of degree $m$.
According to Chen's identity, one has
\begin{align*}
\pi^{(m)}\big[S(\gamma^{I})\big]= & \pi^{(m)}\big[\exp(e_{i})\otimes\exp(e_{J}+\xi)\otimes\exp(-e_{i})\otimes\exp(-e_{J}-\xi)\big]\\
= & \pi^{(m)}\Big[\Big(\sum_{p=0}^{\infty}\frac{1}{p!}e_{i}^{\otimes p}\Big)\otimes\Big(\sum_{q=0}^{\infty}\frac{1}{q!}(e_{J}+\xi)^{\otimes q}\Big)\otimes\Big(\sum_{r=0}^{\infty}\frac{1}{r!}(-e_{i})^{\otimes r}\Big)\\
 & \ \ \ \otimes\Big(\sum_{s=0}^{\infty}\frac{1}{s!}(-e_{J}-\xi)^{\otimes s}\Big)\Big].
\end{align*}
After expanding the above tensor products, it is plain algebra to
check that 
\[
\pi^{(m)}S(\gamma^{I})={\bf 1}+[e_{i},e_{J}]={\bf 1}+e_{I}.
\]
This proves Property (i) for $\gamma^{I}$. Property (ii) is obvious
because there are at most finitely many words of length $\leqslant m$.
\end{proof}
We also recall the following stochastic Taylor expansion for ODEs
(cf. \cite[P. 214]{FV10}). 
\begin{lem}
\label{lem:StocTaylor}Consider the following ODE:
\[
dy_{t}=\sum_{i=1}^{d}V_{i}(y_{t})dq_{t}^{i},\quad y_{0}=x\in\mathbb{R}^{N},
\]
where $V_{1},\cdots,V_{d}$ are $C_{b}^{\infty}$-vector fields and
$q:[0,1]\rightarrow\mathbb{R}^{d}$ is a BV path. Then one has 
\begin{align*}
f(y_{1})= & f(x)+\sum_{k=1}^{r}\sum_{i_{1},\cdots,i_{k}=1}^{d}(V_{i_{1}}\cdots V_{i_{k}}f)(x)\int_{0<t_{1}<\cdots<t_{k}<1}dq_{t_{1}}^{i_{1}}\cdots dq_{t_{k}}^{i_{k}}\\
 & \ \ \ +\sum_{i_{1},\cdots,i_{r+1}=1}^{d}\int_{0<t_{1}<\cdots<t_{r+1}<1}(V_{i_{1}}\cdots V_{i_{r+1}}f)(y_{t_{1}})dq_{t_{1}}^{i_{1}}\cdots dq_{t_{r+1}}^{i_{r+1}}
\end{align*}
for any $r\geqslant1$ and any function $f:\mathbb{R}^{N}\rightarrow\mathbb{R}$
that is continuously differentiable up to order $r+1.$
\end{lem}
\begin{proof}
This is a simple consequence of iterative applications of the chain
rule. 
\end{proof}
Now we are in a position to prove Proposition \ref{prop:NCImpliesG}.

\begin{proof}[Proof of Proposition \ref{prop:NCImpliesG}]

Let $r$ and $F$ be given by Condition (NC). According to Lemma \ref{lem:StocTaylor},
one has 
\begin{align}
F(\Psi_{q}(x))= & F(x)+\sum_{k=1}^{l}\sum_{i_{1},\cdots,i_{k}=1}^{d}(V_{i_{1}}\cdots V_{i_{k}}f)(x)\int_{0<t_{1}<\cdots<t_{k}<1}dq_{t_{1}}^{i_{1}}\cdots dq_{t_{k}}^{i_{k}}\nonumber \\
 & \ \ \ +R_{l+1}(x,q)\label{eq:FTay}
\end{align}
for all $l\leqslant r$, $x\in\mathbb{R}^{N}$ and BV paths $q:[0,1]\rightarrow\mathbb{R}^{d},$
where the remainder $R_{l+1}(x,q)$ satisfies
\begin{equation}
\sup_{x\in\mathbb{R}^{N}}|R_{l+1}(x,q)|\leqslant C_{V,F,r}\|q\|_{1\text{-var }}^{l+1}\label{eq:TayRem}
\end{equation}
with some positive constant $C_{V,F,r}$ depending only on $r$, the
vector fields and $\|D^{l}F\|_{\infty}$ for $1\leqslant l\leqslant r+1$. 

Let ${\cal D}$ denote the space of smooth differential operators
over $\mathbb{R}^{d}$. Let $\rho:T(\mathbb{R}^{d})\rightarrow{\cal D}$
be the algebra homomorphism induced by $\rho(e_{i}):=V_{i}$. The
expansion (\ref{eq:FTay}) can be rewritten as 
\begin{equation}
F(\Psi_{q}(x))=\big(\rho(\pi^{(l)}S(q))F\big)(x)+R_{l+1}(x,q).\label{eq:FTayAlg}
\end{equation}
In addition, one also notes that $\rho(e_{I})=V_{I}$ for any word
$I$. 

According to the assumption (\ref{eq:NC}), for each $x\in\mathbb{R}^{N}$,
there exists a word $I(x)$ such that $2\leqslant|I(x)|\leqslant r$
and $|(V_{I(x)}F)(x)|\geqslant\kappa.$ Let 
\[
\sigma(x):={\rm sgn}\big[(V_{I(x)}F)(x)\big]
\]
so that 
\[
|(V_{I(x)}F)(x)|=\sigma(x)(V_{I(x)}F)(x)\geqslant\kappa.
\]
Given $\delta\in(0,1)$, we define for each $x\in\mathbb{R}^{N}$
the absolutely continuous path 
\[
q_{x}^{\delta}:=(\delta\gamma^{I(x)})^{\sigma(x)}
\]
Here $\alpha{}^{-1}$ means the reversal of the path $\alpha$. One
knows from Lemma \ref{lem:CanPath} that 
\begin{equation}
\pi^{(|I(x)|)}S(q_{x}^{\delta})={\bf 1}+\delta^{|I(x)|}\sigma(x)e_{I(x)},\quad\|q_{x}^{\delta}\|_{1\text{-var}}\leqslant C_{r}\delta\label{eq:qxBd}
\end{equation}
for all $x$. 

It follows from the relation (\ref{eq:FTayAlg}) along with the properties
(\ref{eq:TayRem}, \ref{eq:qxBd}) that 
\[
\big|F\big(\Psi_{q_{x}^{\delta}}(x)\big)-F(x)-\delta^{|I(x)|}\sigma(x)(V_{I(x)}F)(x)\big|\leqslant C'_{V,F,r}\delta^{|I(x)|+1}.
\]
Therefore, 
\begin{align*}
F\big(\Psi_{q_{x}^{\delta}}(x)\big) & \geqslant F(x)+\delta^{|I(x)|}\sigma(x)(V_{I(x)}F)(x)-C'_{V,F,r}\delta^{|I(x)|+1}\\
 & \geqslant F(x)+\delta^{r}\big(\kappa-C'_{V,F,r}\delta\big).
\end{align*}
We choose $\delta:=\kappa/(2C'_{V,F,r})$ and define $a:=\kappa\delta^{r}/2$.
Then one has 
\[
F\big(\Psi_{q_{x}^{\delta}}(x)\big)\geqslant F(x)+a.
\]

To complete the proof, one simply chooses the same $F$ and takes
$q_{x}:=q_{x}^{\delta}$ (assuming that they are all parametrised
on $[0,1]$). Since $DF$ is bounded, $F$ is globally Lipschitz.
In addition, since $I(x)\geqslant2$ , one knows from the first relation
of (\ref{eq:qxBd}) that $q_{x}$ is a loop. It is also clear that
\[
K\triangleq\sup_{x}\|\dot{q}_{x}\|_{L^{\infty}}<\infty,
\]
since the family $\{q_{x}\}$ only involves paths of the form $(\delta\gamma^{I})^{\pm1}$
($|I|\leqslant r$) whose totality is finite.

\end{proof}

\subsection{Recovering the line integral estimate of Boedihardjo-Geng}

In this subsection, we reproduce the main result of \cite{BG24} by using
the current method. In fact, we will obtain a stronger result under
a weaker assumption.
\begin{prop}\label{prop:BGNew}
Let $\phi:\mathbb{R}\rightarrow\mathbb{R}$ be a $C_{b}^{\infty}$-function
such that $\phi'(w_{0})\neq0$ at some $w_{0}\in\mathbb{R}$. Let
$B_{t}=(X_{t},Y_{t})$ be a two-dimensional fBM with Hurst parameter
$H\in(1/4,1/2).$ Then there exist positive constants $C_{1},C_{2}$,
such that the following estimate
\[
\mathbb{P}\Big(\Big|\int_{0}^{1}\phi(X_{t})dY_{t}\Big|\Big)>R\big)\geqslant C_{1}e^{-C_{2}R^{2H+1}}
\]
holds for all large $R$.
\end{prop}
\begin{rem}
In \cite{BG24}, the assumption on the nondegeneracy of $\phi'$ is
slightly stronger and the authors only obtained a weaker estimate
with Weibull exponent being any $\gamma>2H+1$. 
\end{rem}
\begin{proof}
Consider the RDE 
\begin{equation}
dW_{t}=dX_{t},\ dZ_{t}=\phi(W_{t})dY_{t}\label{eq:RDELineInt}
\end{equation}
with $(W_{0},Z_{0})=(0,0).$ Then one has $Z_{1}=\int_{0}^{1}\phi(X_{t})dY_{t}$.
The underlying vector fields are given by $V_{1}=\partial_{w},$ $V_{2}=\phi(w)\partial_{z}$
and they satisfy $[V_{1},V_{2}]=\phi'(w)\partial_{z}$. Note that
the proposition is not an immediate consequence of Theorem \ref{thm:RoughLowNC},
since we do not impose a uniform nondegeneracy property on $\phi'$.
Instead, due to the rather special structure of the underlying RDE
(\ref{eq:RDELineInt}), the previous method can be easily adapted
to the current setting. We will only point out the essential changes. Without loss of generality, assume that $\phi'(w_0)>0$.

We consider the function $F(w,z):=z.$ Let $\kappa,\theta>0$ be such
that $\phi'(x)\geqslant\kappa$ for all $x\in[w_{0}-2\theta,w_{0}+2\theta]$.
Given $\delta>0$, consider the loop 
\[
q^{\delta}:=(\delta e_{w})\sqcup(\delta e_{z})\sqcup(-\delta e_{w})\sqcup(-\delta e_{z}),
\]
where $\{e_{w},e_{z}\}$ is the canonical basis of of the $(w,z)$-plane.
It is readily checked that 
\[
\Psi_{q^{\delta}}(w,z)=\big(w,z+\delta(\phi(w+\delta)-\phi(w))\big).
\]
In particular, 
\begin{align}
F\big(\Psi_{q^{\delta}}(w,z)\big)-F(w,z) & =\delta(\phi(w+\delta)-\phi(w))\nonumber \\
 & =\delta\int_{w}^{w+\delta}\phi'(u)du\geqslant\kappa\delta^{2}=:a\label{eq:Line1SDisp}
\end{align}
for any $w\in[w_{0}-\theta,w_{0}+\theta]$ and $\delta<\theta.$ This
essentially gives Condition (G) with the main difference being that
Condition (G2) only holds on the strip $\{(w,z):w\in[w_{0}-\theta,w_{0}+\theta]\}$. 

Let $n\geqslant1$. Construct the Cameron-Martin path 
\[
h^{(n)}:=\bar{q}\sqcup q_{1}^{(n)}\sqcup\cdots\sqcup q_{n}^{(n)}.
\]
Here $\bar{q}$ is the horizontal path from the origin to the point
$(w_{0},0)$ parametrised on $[0,1/2]$. The path $q_{k}^{(n)}$ is
$q^{\delta}$ reparametrised on $I_{k}:=[\frac{1}{2}+\frac{k-1}{2n},\frac{1}{2}+\frac{k}{2n}]$.
Let $U_{t}^{(n)}=(W_{t}^{(n)},Z_{t}^{(n)})$ ($0\leqslant t\leqslant1$)
be the solution to the RDE 
\[
\begin{cases}
dW_{t}^{(n)}=dX_{t}+dh_{t}^{(n),1}\\
dZ_{t}^{(n)}=\phi(W_{t}^{(n)})d\big(Y_{t}+h_{t}^{(n),2}\big)
\end{cases}
\]
with initial condition $(0,0)$. 

The key observation is that on the event $\{\|{\bf B}\|_{p\text{-var};[0,1]}\leqslant\theta\}$,
one has $W_{t_{k}}^{(n)}\in[w_{0}-\theta,w_{0}+\theta]$ for all $k=0,1,\cdots,n$.
Indeed, the construction of $h^{(n)}$ yields that 
\[
W_{t_{k}}^{(n)}=X_{t_{k}}+w_{0}
\]
for every $k$. As a result, 
\[
\big|W_{t_{k}}^{(n)}-w_{0}\big|=|X_{t_{k}}-0|\leqslant\|{\bf B}\|_{p\text{-var};[0,1]}\leqslant\theta.
\]
It follows from (\ref{eq:Line1SDisp}) that 
\[
F\big(\Psi_{q_{k}^{(n)}}(U_{t_{k-1}}^{(n)})\big)\geqslant F\big(U_{t_{k-1}}^{(n)}\big)+a
\]
for every $k$. 

One can now apply the same argument as in Section \ref{subsec:LinGrowth}
to conclude that 
\[
|U_{1}^{(n)}-U_{1/2}^{(n)}|\gtrsim n
\]
on the event 
\[
A_{n}'(\varepsilon):=\{\|{\bf B}\|_{p\text{-var};[0,1]}\}\bigcap\Big\{\max_{1\leqslant k\leqslant n}\|{\bf B}\|_{p\text{-var};I_{k}}\leqslant\varepsilon\Big\}.
\]
In addition, since each $q_{k}^{(n)}$ is a loop, one can argue in
the same way as in the proof of Lemma \ref{lem:CMEst} to show that
\[
\|h^{(n)}\|_{\bar{{\cal H}}}\lesssim n^{H+1/2}
\]
for all $n$ (the additional part $\bar{q}$ only produces a bounded
contribution). The desired lower tail estimate is then established
by applying the same Girsanov-type argument as before.
\end{proof}

\subsection{Gaussian lower tail under uniform ellipticity}\label{sec:GenLowRough}

We now prove a general Gaussian lower tail estimate under the uniform ellipticity assumption. As we will see, this is a rather trivial corollary of the previous argument for Theorem \ref{thm:RoughLowG}.

\begin{thm}\label{thm:GenLowRough}
Consider the SDE (\ref{eq:SDEIntro}) with $H\in (1/4,1/2)$. Suppose that the vector fields are $C_b^\infty$ and uniformly elliptic in the sense of (\ref{eq:UE}). Then there exist positive constants depending only on $H$ and the vector fields, such that the estimate
\begin{equation}\label{eq:GenLowRough}
\mathbb{P}\big(|X_{1}-x_{0}|>R\big)\geqslant C_{1}e^{-C_{2}R^{2}}
\end{equation}
holds for all large $R$.
\end{thm}

\begin{proof}
The argument is identical to the proof of Theorem \ref{thm:RoughLowG}. The main observation is that Condition (G) in Section \ref{sec:CondG} remains valid but without requiring $q_x$ to be a loop in (G1). Indeed, let $\eta$ be a fixed unit vector in $\mathbb{R}^N$ and define $F(x):=\langle x,\eta\rangle_{\mathbb{R}^N}$. Let 
\begin{equation}\label{eq:RightInv}
A(x):=V(x)^T\big(V(x)V(x)^T\big)^{-1}
\end{equation}be the right inverse of $V=(V_1,\cdots,V_d)$, which is well-defined due to ellipticity. Given $x\in\mathbb{R}^N$, the control  
\[
q_x(t):=\int_0^t A(x+s\eta)\eta ds
\]satisfies $\|\dot{q_x}\|_{\infty}\leqslant \|A\|_\infty$ and $F(\Psi_{q_x}(x)) = F(x) + 1$. Therefore, Condition (G) without the loop requirement in (G1) holds with $K = \|A\|_\infty$ and $a=1$.

The only change to the proof of Theorem \ref{thm:RoughLowG} is that the Cameron-Martin estimate (\ref{eq:CmEstRoughLow}) becomes 
\begin{equation}\label{eq:CMEstUE}
\|h^{(n)}\|_{\bar{{\cal H}}}\leqslant C_{H}Kn.
\end{equation}The estimate (\ref{eq:CMEstUE}) follows directly from the elementary inequality $\|h\|_{\bar{\cal H}}\leqslant C_H \|\dot{h}\|_\infty$, which can be seen by using the same representation (\ref{eq:InvKForm}). With the change of exponent from $H+1/2$ to $1$, one can now carry through exactly the same argument as in the proof of Theorem \ref{thm:RoughLowG} to deduce the desired Gaussian lower bound (\ref{eq:GenLowRough}). 

\end{proof}

\begin{rem}
It is clear from the above argument that the upgrade of a Gaussian lower tail to $(2H+1)$-Weibull is precisely due to the additional loop property in Condition (G1). This loop property has led to nice cancellations in the Cameron-Martin estimate, resulting in a smaller Cameron-Martin norm of order $O(n^{H+1/2})$ vs $O(n)$. From a geometric viewpoint, the possibility of moving the dynamics around by using \textit{loop} controls is a consequence of the noncommutativity of the vector fields (Condition (NC)).

\end{rem}

\newpage
\section{Examples with  $(2H+1)$-Weibull upper tail in Young's regime}\label{sec:CounterEx}

In this section, we construct an example of a $2\times 2$ uniformly elliptic system with $C_b^\infty$-coefficients in Young's regime, whose solution admits a $(2H+1)$-Weibull upper tail. Since $2H+1>2$, this example demonstrates a striking fact that a uniformly elliptic system with $C_b^\infty$-coefficients could fail to have a Gaussian lower tail in Young's regime.

\subsection{A simple example in dimension one with 2D noise}

Before presenting the aforementioned $2\times 2$ example, it is helpful to first discuss a simpler example with dimensions $N=1$ and $d=2$. Such an example already encodes the essential mechanism why a Gaussian lower bound could fail in general. 

\begin{thm}
\label{thm:YoungCounterEx}Let $\eta:\mathbb{R}\rightarrow\mathbb{R}$
be a smooth function such that $\eta(x+1)=\eta(x)+1$ and $\eta(x)=k$
for all $x\in[k+1/4,k+3/4]$ and $k\in\mathbb{Z}$. Consider the solution
$\{X_{t}\}$ to the ODE
\[
dX_{t}=\cos(\pi\eta(X_{t}))dB_{t}^{1}+\sin(\pi\eta(X_{t}))dB_{t}^{2},\quad X_{0}=0,
\]
where $(B_{t}^{1},B_{t}^{2})$ is a two-dimensional fBM with Hurst
parameter $H>1/2$. Then there exist positive constants $C_{1},C_{2}$
depending only on $H,$ such that 
\[
\mathbb{P}\big(|X_{1}|>R\big)\leqslant C_{1}e^{-C_{2}R^{2H+1}}
\]
for all large $R$. 
\end{thm}
The rest of this section is devoted to the proof of Theorem \ref{thm:YoungCounterEx}.

\subsubsection{Reduction to an oscillation event }

We begin with the following simple observation.
\begin{lem}
\label{lem:RedOsEv}The inclusion 
\begin{align}
\big\{|X_{1}|>n\big\}\subseteq & \Big\{\text{there exist }0<s_{1}<t_{1}<\cdots<s_{n}<t_{n}<1,\nonumber \\
 & \ \ \ \text{such that }(-1)^{k-1}(B_{t_{k}}^{1}-B_{s_{k}}^{1})=\frac{1}{2}\ \text{for every }k\Big\}\label{eq:OscilEv}
\end{align}
holds for all $n\geqslant1.$
\end{lem}
\begin{proof}
Suppose that $X_{1}>n$. Define 
\begin{align*}
t_{k} & :=\inf\{t:X_{t}=k-1/4\},\\
s_{k} & :=\sup\{t<t_{k}:X_{t}=k-3/4\}
\end{align*}
for each $1\leqslant k\leqslant n$. By continuity, one must have
\[
0<s_{1}<t_{1}<\cdots<s_{n}<t_{n}<1.
\]
Let $I_{k}:=[k-3/4,k-1/4]$. Since $X_{t}\in I_{k}$ for $t\in[s_{k},t_{k}]$,
one knows from the definition that 
\end{proof}
\[
dX_{t}=(-1)^{k-1}dB_{t}^{1}
\]
on $I_{k}$. Therefore, 
\[
\frac{1}{2}=X_{t_{k}}-X_{s_{k}}=(-1)^{k-1}(B_{t_{k}}^{1}-B_{s_{k}}^{1}).
\]
This proves that the RHS of (\ref{eq:OscilEv}) occurs. The same argument
applies to the other case $X_{1}<-n$.

\subsubsection{Cameron-Martin energy of oscillatory paths}

Since the problem has been reduced to the analysis of a one-dimensional
fBM by Lemma \ref{lem:CMLowFluct}, from now on we will assume that
$B_{t}$ is a 1D fBM with Hurst parameter $H>1/2$ and (intrinsic)
Cameron-Martin space $\bar{{\cal H}}.$ The core of the argument lies
in the observation that a sufficiently oscillatory Cameron-Martin
path cannot have a small Cameron-Martin norm if $H>1/2$. The main
lemma is stated as follows.
\begin{lem}
\label{lem:CMLowFluct}Let $a>0$ be given fixed. Let $n\geqslant1$
and let
\[
0<s_{1}<t_{1}<\cdots<s_{n}<t_{n}<1
\]
be a given sequence. There exists a positive constant $C_{H}$ depending
only on $H$ (but not on $n$ or the partition points $s_{i},t_{j}$),
such that the estimate
\[
\|h\|_{\bar{{\cal H}}}\geqslant C_{H}an^{H+1/2}
\]
holds uniformly for all those paths $h\in\bar{{\cal H}}$ that satisfy
\[
(-1)^{k-1}(h_{t_{k}}-h_{s_{k}})\geqslant a
\]
for every $k=1,\cdots,n.$
\end{lem}
The rest of this subsection is devoted to the proof of Lemma \ref{lem:CMLowFluct}.
We will always assume that $n$ is an even number without loss of
generality (if $n$ is odd, one simply ignores the interval $[s_{n},t_{n}]$
in the lemma).

As the first step, one has the following observation.
\begin{lem}
\label{lem:CMLowVar}The lower estimate 
\begin{equation}
\|h\|_{\bar{{\cal H}}}^{2}\geqslant\frac{a^{2}\big(\sum_{k=1}^{n}l_{k}^{-1}\big)^{2}}{{\rm Var}\big[\sum_{k=1}^{n}\frac{\varepsilon_{k}}{l_{k}}(B_{t_{k}}-B_{s_{k}})\big]}\label{eq:CMLowVar}
\end{equation}
holds uniformly for all $h\in\bar{{\cal H}}$ satisfying
\[
\varepsilon_{k}(h_{t_{k}}-h_{s_{k}})\geqslant a
\]
for every $k=1,\cdots,n,$ where $\varepsilon_{k}:=(-1)^{k-1}$ and
$l_{k}:=t_{k}-s_{k}$. 
\end{lem}
\begin{proof}
Consider the linear functional 
\[
\Lambda:\bar{{\cal H}}\rightarrow\mathbb{R},\quad\Lambda(h):=\sum_{k=1}^{n}\frac{\varepsilon_{k}}{l_{k}}\big(h_{t_{k}}-h_{s_{k}}\big).
\]
We claim that 
\begin{equation}
\|\Lambda\|_{\bar{{\cal H}}^{*}}^{2}={\rm Var}\Big[\sum_{k=1}^{n}\frac{\varepsilon_{k}}{l_{k}}(B_{t_{k}}-B_{s_{k}})\Big].\label{eq:LamNorm}
\end{equation}
Indeed, under the canonical isometric isomorphisms $\bar{{\cal H}}^{*}\cong\bar{{\cal H}}\cong\mathcal{C}_{1}^{B}$
(the last denotes the first Wiener chaos of $B$), the chaos element
corresponding to the functional $[h\mapsto h_{t}]$ is $B_{t}$ for
every given $t$. By linearity, the chaos element corresponding to
$\Lambda$ is 
\[
\sum_{k=1}^{n}\frac{\varepsilon_{k}}{l_{k}}(B_{t_{k}}-B_{s_{k}}).
\]
This proves the claim (\ref{eq:LamNorm}). On the other hand, for
any $h\in\bar{{\cal H}}$ satisfying the assumption of the lemma,
one has 
\[
\Lambda(h)\geqslant a\sum_{k=1}^{n}l_{k}^{-1}.
\]
The desired estimate (\ref{eq:CMLowVar}) follows from the Cauchy-Schwarz
inequality and the relation (\ref{eq:LamNorm}). 
\end{proof}
Now the problem is reduced to an upper estimate of the variance appearing
in (\ref{eq:CMLowVar}). To this end, one first notes from \cite[Section 2.1.1]{BHOZ08}
in the $H>1/2$ regime that 
\begin{equation}
{\rm Var}\Big[\int_{0}^{1}f(r)dB_{r}\Big]=H(2H-1)\int_{\mathbb{R}^{2}}f(u)f(v)|u-v|^{2H-2}dudv\label{eq:VarRep}
\end{equation}
for any deterministic and regular $f:[0,1]\rightarrow\mathbb{R}$
(say, piecewise smooth). Here one extends $f$ to the real line by
setting it to be zero outside $[0,1]$. In our context, 
\begin{equation}
f(x)=\sum_{k=1}^{n}\frac{\varepsilon_{k}}{l_{k}}{\bf 1}_{I_{k}}(x),\label{eq:f}
\end{equation}
where $I_{k}:=[s_{k},t_{k}]$.

Next, we derive a basic identity which allows one to analyse the RHS
of (\ref{eq:VarRep}) effectively.
\begin{lem}
\label{lem:fFId} Let $\beta\in(0,1)$. Then one has 
\[
\int_{\mathbb{R}^{2}}f(u)f(v)|u-v|^{-\beta}dudv=\beta(1+\beta)\int_{0}^{\infty}\frac{\|\Delta_{r}F\|_{L^{2}(\mathbb{R})}^{2}}{r^{2+\beta}}dr
\]
for all bounded and compactly supported $f:\mathbb{R}\rightarrow\mathbb{R}$,
where $F(x):=\int_{-\infty}^{x}f(u)du$ and $(\Delta_{r}F)(x):=F(x+r)-F(x)$.
\end{lem}
\begin{proof}
By definition, one has 
\begin{align*}
\|\Delta_{r}F\|_{L^{2}(\mathbb{R})}^{2} & =\int_{\mathbb{R}^{2}}\Big(\int_{x}^{x+r}f(u)du\Big)\Big(\int_{x}^{x+r}f(v)dv\Big)dx\\
 & =\int_{\mathbb{R}^{2}}f(u)f(v)\Big[\int_{\mathbb{R}}{\bf 1}_{[u-r,u]\cap[v-r,v]}(x)dx\Big]dudv\\
 & =\int_{\mathbb{R}^{2}}f(u)f(v)\big(r-|u-v|\big)_{+}dudv.
\end{align*}
It follows that 
\[
\int_{0}^{\infty}\frac{\|\Delta_{r}F\|_{L^{2}(\mathbb{R})}^{2}}{r^{2+\beta}}dr=\int_{\mathbb{R}^{2}}f(u)f(v)\Big[\int_{0}^{\infty}\frac{\big(r-|u-v|\big)_{+}}{r^{2+\beta}}dr\Big]dudv.
\]
The $dr$-integral is easily evaluated as 
\[
\int_{0}^{\infty}\frac{\big(r-|u-v|\big)_{+}}{r^{2+\beta}}dr=\frac{|u-v|^{-\beta}}{\beta(1+\beta)}.
\]
The desired relation thus follows. 
\end{proof}
The main upper estimate of the variance appearing in (\ref{eq:CMLowVar})
is now stated as follows. 
\begin{lem}
\label{lem:VarEst}There exists a positive constant $C_{H}$ depending
only on $H$, such that 
\[
{\rm Var}\Big[\sum_{k=1}^{n}\frac{\varepsilon_{k}}{l_{k}}(B_{t_{k}}-B_{s_{k}})\Big]\leqslant C_{H}\sum_{k=1}^{n}l_{k}^{2H-2}.
\]
\end{lem}
\begin{proof}
According to (\ref{eq:VarRep}) and Lemma \ref{lem:fFId}, 
\begin{equation}
{\rm Var}\Big[\sum_{k=1}^{n}\frac{\varepsilon_{k}}{l_{k}}(B_{t_{k}}-B_{s_{k}})\Big]=c_{H}\int_{0}^{\infty}\frac{\|\Delta_{r}F\|_{L^{2}(\mathbb{R})}^{2}}{r^{4-2H}}dr,\label{eq:VarRepF}
\end{equation}
where 
\[
c_{H}:=H(2H-1)(2-2H)(3-2H)
\]
and the function $F(x):=\int_{-\infty}^{x}f(u)du$ with $f$ being
defined by (\ref{eq:f}).

To estimate the RHS of (\ref{eq:VarRepF}), one further writes 
\[
F=\sum_{k=1}^{n/2}F_{k};\quad F_{k}:=\int_{-\infty}^{x}\Big(\frac{{\bf 1}_{I_{2k-1}}(u)}{l_{2k-1}}-\frac{{\bf 1}_{I_{2k}}(u)}{l_{2k}}\Big)du.
\]
It is easy to see that 
\begin{equation}
0\leqslant F_{k}\leqslant1,\quad\int_{\mathbb{R}}|F_{k}'(x)|dx\leqslant2.\label{eq:FkProp}
\end{equation}
In addition, the supports of the $F_{k}$'s are disjoint. We claim
that 
\begin{equation}
\|\Delta_{r}F_{k}\|_{L^{2}(\mathbb{R})}^{2}\leqslant\min\{L_{k}r^{2},2r\}\label{eq:DelFkEst}
\end{equation}
for all $r>0,$ where $L_{k}:=l_{2k-1}^{-1}+l_{2k}^{-1}$. Indeed,
one has 
\[
\big|(\Delta_{r}F_{k})(x)\big|^{2}=\Big|\int_{0}^{r}F_{k}'(x+u)du\Big|^{2}\leqslant r\int_{0}^{r}\big|F_{k}'(x+u)\big|^{2}du.
\]
It follows that 
\begin{equation}
\|\Delta_{r}F_{k}\|_{L^{2}(\mathbb{R})}^{2}\leqslant r^{2}\|F_{k}'\|_{L^{2}(\mathbb{R})}^{2}=L_{k}r^{2}.\label{eq:DelFk1}
\end{equation}
On the other hand, according to (\ref{eq:FkProp}), one also has 
\begin{align}
\|\Delta_{r}F_{k}\|_{L^{2}(\mathbb{R})}^{2} & \leqslant\int_{\mathbb{R}}\big|(\Delta_{r}F_{k})(x)\big|dx\nonumber \\
 & \leqslant\int_{\mathbb{R}}\int_{0}^{r}\big|F_{k}'(x+u)\big|dudx\leqslant2r.\label{eq:DelFk2}
\end{align}
The estimate (\ref{eq:DelFkEst}) thus follows from (\ref{eq:DelFk1})
and (\ref{eq:DelFk2}).

Now consider the decomposition 
\[
(\Delta_{r}F)(x)=\sum_{k=1}^{n/2}(\Delta_{r}F_{k})(x).
\]
Since the $F_{k}$'s have disjoint supports, for every $x\in\mathbb{R}$
there are at most two nonzero terms in the above summation. As a result,
one has 
\[
\big|(\Delta_{r}F)(x)\big|^{2}\leqslant2\sum_{k=1}^{n/2}\big|(\Delta_{r}F_{k})(x)\big|^{2}.
\]
It follows from (\ref{eq:DelFkEst}) that 
\begin{align*}
\int_{0}^{\infty}\frac{\|\Delta_{r}F\|_{L^{2}(\mathbb{R})}^{2}}{r^{4-2H}}dr\leqslant & 2\sum_{k=1}^{n/2}\int_{0}^{\infty}\frac{\|\Delta_{r}F_{k}\|_{L^{2}(\mathbb{R})}^{2}}{r^{4-2H}}dr\\
\leqslant & 2\sum_{k=1}^{n/2}\int_{0}^{\infty}\frac{\min\{L_{k}r^{2},2r\}}{r^{4-2H}}dr\\
= & \frac{2^{2H-1}}{(2H-1)(1-H)}\sum_{k=1}^{n/2}L_{k}^{2-2H}\\
\leqslant & \frac{2^{2H-1}}{(2H-1)(1-H)}\sum_{k=1}^{n/2}\big(l_{2k-1}^{2H-2}+l_{2k}^{2H-2}\big).
\end{align*}
In view of the relation (\ref{eq:VarRepF}), the result thus follows. 
\end{proof}
\begin{proof}[Proof of Lemma \ref{lem:CMLowFluct}]

According to Lemma \ref{lem:CMLowVar} and Lemma \ref{lem:VarEst},
one has 
\[
\|h\|_{\bar{{\cal H}}}^{2}\geqslant C_{H}a^{2}\cdot\frac{\big(\sum_{k=1}^{n}l_{k}^{-1}\big)^{2}}{\big(\sum_{k=1}^{n}l_{k}^{-(2-2H)}\big).}
\]
Denote $x_{k}:=l_{k}^{-1}$ and $\beta:=2-2H$. Since $\beta\in(0,1)$,
concavity of $x\mapsto x^{\beta}$ implies that 
\[
\frac{\big(\sum_{k=1}^{n}x_{k}\big)^{2}}{\sum_{k=1}^{n}x_{k}^{\beta}}=\Big(\sum_{k=1}^{n}x_{k}\Big)^{2-\beta}\frac{\big(\sum_{k=1}^{n}x_{k}\big)^{\beta}}{\sum_{k=1}^{n}x_{k}^{\beta}}\geqslant n^{\beta-1}\Big(\sum_{k=1}^{n}x_{k}\Big)^{2-\beta}.
\]
On the other hand, since 
\[
\sum_{k=1}^{n}x_{k}^{-1}=\sum_{k=1}^{n}l_{k}\leqslant1,
\]
the Cauchy-Schwarz inequality shows that 
\[
n^{2}\leqslant\Big(\sum_{k=1}^{n}x_{k}^{-1}\Big)\Big(\sum_{k=1}^{n}x_{k}\Big)\leqslant\sum_{k=1}^{n}x_{k}.
\]
As a consequence, 
\[
\|h\|_{\bar{{\cal H}}}^{2}\geqslant C_{H}a^{2}\cdot n^{\beta-1}\cdot n^{2(2-\beta)}=C_{H}a^{2}n^{2H+1}.
\]
This completes the proof of Lemma \ref{lem:CMLowFluct}.

\end{proof}

\subsubsection{The fixed-grid probabilistic estimate}

Our goal of this subsection is to prove the following estimate. We
continue to use the previous notation.
\begin{lem}
\label{lem:FixGdEst}There exist positive constants $C_{1},C_{2}$
depending only on $H$, such that the estimate 
\begin{equation}
\mathbb{P}\big(\varepsilon_{k}(B_{t_{k}}-B_{s_{k}})\geqslant a\ \text{for every }k=1,\cdots n\big)\leqslant C_{1}e^{-C_{2}a^{2}n^{2H+1}}\label{eq:FixGdEst}
\end{equation}
holds uniformly for all $a>0$, $n\geqslant1$ and $0<s_{1}<t_{1}<\cdots<s_{n}<t_{n}<1.$ 
\end{lem}
The proof of Lemma \ref{lem:FixGdEst} relies on the following standard
result from Gaussian analysis. Recall that the Cameron-Martin subspace
of a centered, nondegenerate Gaussian vector $Z=(Z_{1},\cdots,Z_{n})$
is the image ${\cal H}_{Z}$ of the map 
\begin{equation}
{\rm Span}\{Z_{1},\cdots,Z_{n}\}\ni\eta\mapsto z:=\mathbb{E}[\eta Z]\in\mathbb{R}^{n},\label{eq:FDChao}
\end{equation}
whose inner product structure is defined by 
\[
\langle z,w\rangle_{{\cal H}_{Z}}:=\mathbb{E}[\eta\zeta],
\]
where $\eta$ (respectively, $\zeta$) is the preimage of $z$ (respectively,
$w$) under (\ref{eq:FDChao}). Of course, ${\cal H}_{Z}=\mathbb{R}^{n}$
since $Z$ is nondegenerate but the underlying inner product is induced
from the covariance structure of $Z$. More explicitly, one has 
\[
\langle z,w\rangle_{{\cal H}_{Z}}=z^{T}Q^{-1}w
\]
where $Q$ is the covariance matrix of $Z$, i.e. $Z\sim N(0,Q).$
\begin{prop}
\label{prop:GauIso}Let $Z$ be a centered, nondegenerate Gaussian
vector. Let $K$ be a closed, convex subset of $\mathbb{R}^{n}$.
Suppose that 
\[
r:=d_{{\cal H}_{Z}}(o,K):=\inf_{z\in K}\|z\|_{{\cal H}_{Z}}>0.
\]
Then 
\[
\mathbb{P}(Z\in K)\leqslant1-\Phi(r),
\]
where $\Phi(t):=\frac{1}{\sqrt{2\pi}}\int_{-\infty}^{t}e^{-u^{2}/2}du$
is the standard normal distribution function.
\end{prop}
\begin{proof}
This is essentially the Gaussian isoperimetric inequality (cf. \cite{Led94}).
For completeness, let us provide the short and simple proof in the
current finite dimensional setting. Let $Q:={\rm Cov}[Z]$ and define
$X:=Q^{-1/2}Z\sim N(0,{\rm Id})$. Also set $L:=Q^{-1/2}K$. Note
that $L$ is again a closed, convex subset of $\mathbb{R}^{n}$. In
addition, one has 
\[
d_{\,\mathbb{R}^{n}}(o,L)=\inf_{x\in L}|x|_{\mathbb{R}^{n}}=\inf_{z\in K}\big|Q^{-1/2}z\big|_{\mathbb{R}^{n}}=\inf_{z\in K}\|z\|_{{\cal H}_{Z}}=r.
\]
Let $x_{*}\in L$ be the unique minimiser of $L\ni x\mapsto|x|_{\mathbb{R}^{n}}$.
Convexity implies that 
\[
\langle x_*,x-x_{*}\rangle_{\mathbb{R}^{n}}\geqslant0
\]
for all $x\in L$. As a result, one has 
\[
L\subseteq\big\{ x:\langle u,x\rangle_{\mathbb{R}^{n}}\geqslant r\big\},
\]
where $u:=x_{*}/|x_{*}|_{\mathbb{R}^{n}}$. It follows that 
\[
\mathbb{P}(Z\in K)=\mathbb{P}(X\in L)\leqslant\mathbb{P}\big(\langle u,X\rangle_{\mathbb{R}^{n}}\geqslant r\big)=1-\Phi(r),
\]
since $\langle u,X\rangle\sim N(0,1)$.
\end{proof}
\begin{proof}[Proof of Lemma \ref{lem:FixGdEst}]

Consider the nondegenerate Gaussian vector 
\[
Z:=\Big(\varepsilon_{k}(B_{t_{k}}-B_{s_{k}})\Big)_{k=1}^{n}.
\]
Define the projection 
\[
P_{Z}:\bar{{\cal H}}\rightarrow\mathbb{R}^{n},\quad P_{Z}(h):=\Big(\varepsilon_{k}(h_{t_{k}}-h_{s_{k}})\Big)_{k=1}^{n}.
\]
It is standard that ${\cal H}_{Z}=P_{Z}(\bar{{\cal H}})=\mathbb{R}^{n}$
and 
\begin{equation}
|z|_{{\cal H}_{Z}}=\inf\big\{\|h\|_{\bar{{\cal H}}}:P_{Z}(h)=z\big\}.\label{eq:CMProj}
\end{equation}
Let $K:=[a,\infty)^{n}.$ According to Lemma \ref{lem:CMLowFluct}
and (\ref{eq:CMProj}), one has 
\[
d_{{\cal H}_{Z}}(o,K)\geqslant C_{H}an^{H+1/2},
\]
where $C_{H}$ is a positive constant that does not depend on $a,n$
or the grid points $s_{i},t_{j}$. It follows from Proposition \ref{prop:GauIso}
that 
\[
\mathbb{P}(Z\in K)\leqslant2e^{-C_{H}'a^{2}n^{2H+1}}.
\]
This yields the desired inequality (\ref{eq:FixGdEst}).

\end{proof}

\subsubsection{The general probabilistic estimate}

Finally, we are going to remove the grid points in the estimate (\ref{eq:FixGdEst}).
This is achieved by a simple continuity argument. Given $n\geqslant1$,
let ${\cal P}_{n}$ denote the set of partitions 
\[
({\bf s},{\bf t}):0<s_{1}<t_{1}<\cdots<s_{n}<t_{n}<1.
\]
For each $a>0$ and $({\bf s},{\bf t})\in{\cal P}_{n}$, we define
$E_{{\bf s},{\bf t}}(a)$ to be the event that 
\[
\varepsilon_{k}\big(B_{t_{k}}-B_{s_{k}}\big)\geqslant a
\]
for every $k=1,\cdots,n.$
\begin{lem}
\label{lem:GenProbEst}There exist positive constants $C_{1},C_{2}$
depending only on $H$ and an integer $N$ depending only on $H,a$$,$
such that the estimate 
\[
\mathbb{P}\big(\big\{ E_{{\bf s},{\bf t}}(a)\ \text{occurs for some }\text{partition}({\bf s},{\bf t})\in{\cal P}_{n}\big\}\big)\leqslant C_{1}e^{-C_{2}a^{2}n^{2H+1}}
\]
holds for all $n>N$.
\end{lem}
\begin{proof}
Define 
\[
E_{n}(a):=\big\{ E_{{\bf s},{\bf t}}(a)\ \text{occurs for some }\text{partition }({\bf s},{\bf t})\in{\cal P}_{n}\big\}.
\]
Let $\delta_{n}:=n^{-q}$ with $q$ being specified later on. Let
${\cal D}(\delta_{n})$ be the partition of $[0,1]$ into equal subintervals
of length $\delta_{n}$. Consider the event 
\[
G_{n}(a):=\Big\{ m_{B}(2\delta_{n})<\frac{a}{4}\Big\},
\]
where 
\[
m_{B}(\delta):=\sup_{s,t\in[0,1],|t-s|\leqslant\delta}|B_{t}-B_{s}|
\]
is the modulus of continuity of $B$. 

We claim that 
\[
E_{n}(a)\cap G_{n}(a)\subseteq\bigcup_{\substack{(\hat{{\bf s}},\hat{{\bf t}})\in{\cal P}_{n}\\
\hat{s}_{i},\hat{t}_{j}\in{\cal D}(\delta_{n})
}
}E_{\hat{{\bf s}},\hat{{\bf t}}}(a/2).
\]
Indeed, let $\omega\in{\rm LHS}$. Then 
\begin{equation}
\varepsilon_{k}\big(B_{t_{k}}-B_{s_{k}}\big)\geqslant a\label{eq:OsIncfBM}
\end{equation}
holds with some $2n$-tuple $({\bf s},{\bf t})$ for every $k$. It
is obvious that $t_{k}-s_{k}>2\delta_{n}$ for every $k$, for otherwise
the relation (\ref{eq:OsIncfBM}) would contradict the fact that $m_{\omega}(2\delta_{n})<a/4$.
Define the $2n$-tuple $(\hat{{\bf s}},\hat{{\bf t}})$ where $\hat{s}_{k}>s_{k}$
(respectively, $\hat{t}_{k}<t_{k}$) is the partition point in ${\cal D}(\delta_{n})$
that is closest to $s_{k}$ from the right (respectively, closest
to $t_{k}$ from the left). The points $\hat{s}_{k},\hat{t}_{k}$
exist because $t_{k}-s_{k}>2\delta_{n}$. It follows that 
\begin{align*}
\varepsilon_{k}\big(B_{\hat{t}_{k}}-B_{\hat{s}_{k}}\big) & \geqslant\varepsilon_{k}\big(B_{t_{k}}-B_{t_{k}}\big)-\big|B_{s_{k}}-B_{\hat{s}_{k}}\big|-\big|B_{t_{k}}-B_{\hat{t}_{k}}\big|\\
 & \geqslant a-2\omega_{B}(\delta_{n})>a-\frac{a}{2}=\frac{a}{2}
\end{align*}
for every $k$. This proves the claim. 

Note that the cardinality of $2n$-tuples $(\hat{{\bf s}},\hat{{\bf t}})$
chosen from the partition ${\cal D}(\delta_{n})$ is at most 
\[
|{\cal D}(\delta_{n})|^{2n}=e^{2nq\log n}.
\]
According to Lemma \ref{lem:FixGdEst}, one obtains that 
\begin{equation}
\mathbb{P}\big(E_{n}(a)\cap G_{n}(a)\big)\leqslant e^{2nq\log n}\times C_{1}e^{-C_{2}a^{2}n^{2H+1}}.\label{eq:EnaMain}
\end{equation}
The first term is negligible comparing to the second term for all
$n>N(H,a,q).$

Now it remains to estimate $\mathbb{P}(G_{n}(a)^{c}).$ To this end,
we first claim that 
\begin{equation}
\mathbb{P}(m_{B}(r)>\lambda)\leqslant c_H r^{-1}e^{-C_{H}\lambda^{2}/r^{2H}}\label{eq:TailMod}
\end{equation}
for all $r\in(0,1)$ and $\lambda>0$. Indeed, note that 
\[
m_{B}(r)\leqslant2\max_{0\leqslant k\leqslant r^{-1}-2}\sup_{t\in J_{k}}\big|B_{t}-B_{kr}\big|
\]
where $J_{k}:=[kr,(k+2)r].$ By applying a simple union bound, the
fractional Brownian scaling and Borel's inequality (see \cite[Theorem 2.1]{Adl90}),
one finds that 
\[
\mathbb{P}(m_{B}(r)>\lambda)\leqslant  r^{-1}\mathbb{P}\Big(\sup_{t\in[0,2]}|\tilde{B}_{t}|>\frac{u}{2r^{H}}\Big)\leqslant c_H r^{-1}e^{-C_{H}u^{2}/r^{2H}}.
\]
This proves the estimate (\ref{eq:TailMod}). It follows that 
\[
\mathbb{P}\big(G_{n}(a)^{c}\big)\leqslant C_{3}n^{q}e^{-C_{4}a^{2}n^{2Hq}}.
\]
By choosing $q>1+\frac{1}{2H}$, the RHS of the above inequality is
negligible comparing to the main estimate (\ref{eq:EnaMain}). As
a consequence, one concludes that 
\[
\mathbb{P}(E_{n}(a))\leqslant\mathbb{P}\big(E_{n}(a)\cap G_{n}(a)\big)+\mathbb{P}(G_{n}(a)^{c})\leqslant C_{5}e^{-C_{6}a^{2}n^{2H+1}}
\]
for all $n>N(H,a)$.
\end{proof}
According to Lemma \ref{lem:RedOsEv} and Lemma \ref{lem:GenProbEst},
the proof of Theorem \ref{thm:YoungCounterEx} is now complete.

\subsection{A genuine 2D example}

A major unsatisfactory point about the 1D example of Theorem \ref{thm:YoungCounterEx} is that it is not a genuine elliptic system (the dimension of noise is larger than the dimension of solution). The availability of additional noise dimension has made the underlying oscillation mechanism quite transparent. On the other hand, any naive attempt of adding an additional state dimension to the system will turn it back into the Gaussian lower tail regime.

In this section, we present a genuine $2\times 2$ example for which a $(2H+1)$-Weibull upper tail remains valid. The main result is stated as follows.

\begin{thm}
\label{thm:2DCounter} Let $\eta:[0,\infty)\rightarrow[0,\infty)$ be a smooth function that is constant near the origin and such that  $\eta(r)=r$ for $r\geqslant2.$ Let $B_{t}=(B_{t}^{1},B_{t}^{2})$ be a fractional
Brownian motion with Hurst parameter $H>1/2$. Consider the SDE 
\begin{equation}
dX_{t}=V_{1}(X_{t})dB_{t}^{1}+V_{2}(X_{t})dB_{t}^{2},\quad X_{0}=0\in\mathbb{R}^{2},\label{eq:2DCounter}
\end{equation}
where 
\begin{equation}
\big(V_{1}(x),V_{2}(x)\big)=\left(\begin{array}{cc}
\cos\eta(|x|) & -\sin\eta(|x|)\\
\sin\eta(|x|) & \cos\eta(|x|).
\end{array}\right)\label{eq:2DVFld}
\end{equation}
Then there exist positive constants $C_{1},C_{2}$, such that 
\[
\mathbb{P}\Big(\sup_{0\leqslant t\leqslant1}|X_{t}|>R\Big)\leqslant C_{1}e^{-C_{2}R^{2H+1}}
\]
for all large $R$. 
\end{thm}

Our main idea of proving Theorem \ref{thm:2DCounter} can be vaguely
summarised as follows. Consider the solution map $w\mapsto(X_{t}^{w})_{0\leqslant t\leqslant1}$.
The key insight is that if $h$ is a Cameron-Martin path that drives
the dynamics to a distance of order $R$, i.e. 
\[
\sup_{0\leqslant t\leqslant1}|X_{t}^{h}|\geqslant R,
\]
the path $h$ has to be either longer than $O(R^{H+1/2})$ or it has
to make at least $O(R)$ number of ``effective turns''. The latter
case forces its Cameron-Martin norm to be at least $O(R^{H+1/2})$.
A technical part of the argument is to show that such a coercivity
property is robust under perturbation of $h$ by $w\in C^{p\text{-var}}([0,1];\mathbb{R}^{2})$
with bounded $p$-variation norm. This perturbative coercivity estimate
is needed to transfer the deterministic estimate to a probabilistic
one by using the Gaussian isoperimetric inequality. The 
mechanism here can be viewed as a two-dimensional analogue of Theorem
\ref{thm:YoungCounterEx}.

\subsubsection{A perturbative coercivity estimate}

The essential ingredient for the proof of Theorem \ref{thm:2DCounter}
is the following coercivity estimate. This is a more delicate two-dimensional
version of Lemma \ref{lem:CMLowFluct}, which is also robust under
perturbation. Given $w\in C^{p\text{-var}}([0,1];\mathbb{R}^{2})$
($1/H<p<2$), we use $X_{t}^{w}$ to denote the solution to the differential
equation (\ref{eq:2DCounter}) driven by the path $w$. 
\begin{prop}
\label{prop:KeyCoer}Let $p\in(H^{-1},2).$ For any $M>0,$ there
exist positive constants $c_{*}$ and $R_{*}$ depending only on $p,H,V$
and $M$, such that the following implication 
\begin{equation}
\sup_{0\leqslant t\leqslant1}|X_{t}^{w+h}|\geqslant R,\ R>R_{*}\implies\|h\|_{\bar{{\cal H}}}\geqslant c_{*}R^{H+1/2}\label{eq:KeyCoer}
\end{equation}
holds uniformly over all $h\in\bar{{\cal H}}$ and all $w\in C^{p\text{-var}}([0,1];\mathbb{R}^{2})$
satisfying $\|w\|_{p\text{-var}}\leqslant M$.
\end{prop}

The rest of this subsection is devoted to the proof of Proposition \ref{prop:KeyCoer}.

\subsubsection*{A lower estimate of Cameron-Martin norm }

In this part, we establish an estimate which controls the Cameron-Martin
norm from below by the number of ``effective'' turns. We first introduce
the definition of an effective turn. Let $a_{0}>0$ and $\theta_{0}\in(0,\pi)$
be given fixed numbers. Later on, we will specify their values (see
(\ref{eq:ParChoice})). For the current discussion, this is not important.
Given two unit vectors $v,w\in\mathbb{R}^{2},$ we use $\angle(v,w):=\arccos\langle v,w\rangle\in[0,\pi]$
to denote their angle. 
\begin{defn}
Let $h:[0,1]\rightarrow\mathbb{R}^{2}$ be an absolutely continuous
path. Let $\mu_{h}(dt):=|\dot{h}_{t}|dt$ denote its arclength measure.
An interval $I=[u,v]\subseteq[0,1]$ is said to \textit{contain an
$(a_{0},\theta_{0})$-effective turn} of $h$ if there are measurable
subsets $A,B\subseteq I$ such that 
\[
\mu_{h}(A)\geqslant a_{0},\quad\mu_{h}(B)\geqslant a_{0}
\]
and 
\[
\text{\ensuremath{\angle}\ensuremath{\Big(}\ensuremath{\ensuremath{\frac{\dot{h}_{s}}{|\dot{h}_{s}|}}},\ensuremath{\frac{\dot{h}_{t}}{|\dot{h}_{t}|}}\Big)\ensuremath{\geqslant}\ensuremath{\theta_{0}}}
\]
for $\mu_{h}\otimes\mu_{h}$-a.e. $(s,t)\in A\times B$ at which $\dot{h}_{s}\neq0$
and $\dot{h}_{t}\neq0$. We define ${\cal N}_{{\rm eff}}(h)$ to be
the maximal number of disjoint intervals in $[0,1]$ that contain
an $(a_{0},\theta_{0})$-effective turn of $h$. 
\end{defn}
\begin{rem}
Heuristically, an effective turn consists of two pieces of $h$ on
$I$ which have arclength at least $a_{0}$ and have angle separately
uniformly by at least $\theta_{0}.$ 
\end{rem}
The main Cameron-Martin estimate in this part is stated as follows.
\begin{lem}
\label{lem:CMEffTurn}There exists a positive constant $C_{H}$ depending
only on $H$, such that 
\[
\|h\|_{\bar{{\cal H}}}\geqslant C_{H}a_{0}\sin\Big(\frac{\theta_{0}}{2}\Big)\cdot{\cal N}_{{\rm eff}}(h)^{H+1/2}
\]
for all $h\in\bar{{\cal H}}$.
\end{lem}
\begin{proof}
In particular, 
\begin{equation}
\|h\|_{\bar{{\cal H}}}^{2}\geqslant C_{H}\int_{0}^{1}\int_{0}^{1}\frac{\big|\dot{h}_{t}-\dot{h}_{s}\big|^{2}}{|t-s|^{2H}}dsdt.\label{eq:CMLowFrac}
\end{equation}
Next, let $I\subseteq[0,1]$ be an interval which contains an $(a_{0},\theta_{0})$-effective
turn $(A,B)$ of $h$. By definition, it is not hard to see that 
\[
\big|\dot{h}_{s}-\dot{h}_{t}\big|^{2}\geqslant4\sin^{2}\Big(\frac{\theta_{0}}{2}\Big)\cdot|\dot{h}_{s}|\cdot|\dot{h}_{t}|
\]
for $(\mu_{h}\otimes\mu_{h})$-a.e. $(s,t)\in A,B$. As a consequence,
\begin{align}
\int_{I\times I}\frac{\big|\dot{h}_{t}-\dot{h}_{s}\big|^{2}}{|t-s|^{2H}}dsdt & \geqslant4\sin^{2}\Big(\frac{\theta_{0}}{2}\Big)\int_{A\times B}\frac{|\dot{h}_{s}|\cdot|\dot{h}_{t}|}{|t-s|^{2H}}dsdt\nonumber \\
 & \geqslant4\sin^{2}\Big(\frac{\theta_{0}}{2}\Big)\cdot|I|^{-2H}\cdot a_{0}^{2}.\label{eq:CMLowEachI}
\end{align}
Now suppose that $I_{1},\cdots,I_{N}$ are $N$ given disjoint intervals,
each containing an $(a_{0},\theta_{0})$-effective turn. It follows
from (\ref{eq:CMLowEachI}) that
\begin{align}
\int_{0}^{1}\int_{0}^{1}\frac{\big|\dot{h}_{t}-\dot{h}_{s}\big|^{2}}{|t-s|^{2H}}dsdt & \geqslant\sum_{k=1}^{N}\int_{I_{k}\times I_{k}}\frac{\big|\dot{h}_{t}-\dot{h}_{s}\big|^{2}}{|t-s|^{2H}}dsdt\nonumber \\
 & \geqslant4a_{0}^{2}\sin^{2}\Big(\frac{\theta_{0}}{2}\Big)\cdot\sum_{k=1}^{N}|I_{k}|^{-2H}.\label{eq:FracLow}
\end{align}
On the other hand, by using the convexity of $x\mapsto x^{-2H}$ and
the fact that $\sum_{k}|I_{k}|\leqslant1,$ one has
\begin{equation}
\frac{1}{N}\sum_{k=1}^{N}|I_{k}|^{-2H}\geqslant\Big(\frac{1}{N}\sum_{k=1}^{N}|I_{k}|\Big)^{-2H}\geqslant N^{2H}.\label{eq:Conv}
\end{equation}
After combining (\ref{eq:CMLowFrac}), (\ref{eq:FracLow}) and (\ref{eq:Conv}),
one arrives at 
\[
\|h\|_{\bar{{\cal H}}}^{2}\geqslant4C_{H}a_{0}^{2}\sin^{2}\Big(\frac{\theta_{0}}{2}\Big)\cdot N^{2H+1}.
\]
The result follows from the definition of ${\cal N}_{{\rm eff}}(h)$.
\end{proof}

\subsubsection*{\label{subsec:NonPert}A non-perturbative coercivity estimate}

In order to prove Proposition \ref{prop:KeyCoer}, one needs to control
the quantity ${\cal N}_{{\rm eff}}(h)$ from below. It is helpful
to first establish a non-perturbative lower estimate of ${\cal N}_{{\rm eff}}(h)$.
This part contains the essential idea and the perturbative version
follows easily by controlling an additional Young integration error.
The main lemma is stated as follows.
\begin{lem}
\label{lem:NonPertCoer}There exist universal positive constants $C_{1},C_{2}$,
such that the estimate
\begin{equation}
{\cal N}_{{\rm eff}}(h)\geqslant C_{1}R-\frac{C_{2}\|h\|_{1\text{-var}}}{R}\label{eq:NonPertCoer}
\end{equation}
holds uniformly for all $R>4$ and all $h\in\bar{{\cal H}}$ satisfying
$\sup_{0\leqslant t\leqslant1}|X_{t}^{h}|\geqslant R$. 
\end{lem}
The rest of this part is devoted to the proof of Lemma \ref{lem:NonPertCoer}. 

From now on, we fix three parameters $a_{0}>0$, $\theta_{0}\in(0,\pi)$
and $m_{0}\in\mathbb{N}$ (their exact values will be chosen by (\ref{eq:ParChoice})
later on). We divide the unit circle $S^{1}$ into $m_{0}$ equal
arcs $P_{1},P_{2},\cdots,P_{m_{0}}$. Let $h:[0,1]\rightarrow\mathbb{R}^{2}$
be an absolutely continuous path and let $I=[u,v]\subseteq[0,1]$
be a given interval. An arc $P_{k}$ is said to be \textit{heavy}
with respect to $h|_{I}$ if 
\[
\mu_{h}\Big(\Big\{ t\in I:\dot{h}_{t}\neq0,\ \frac{\dot{h}_{t}}{|\dot{h}_{t}|}\in P_{k}\Big\}\Big)\geqslant a_{0},
\]
where $\mu_{h}(dt):=|\dot{h}_{t}|dt$ is the arclength measure associated
with $h$. An arc that is not heavy is said to be \textit{light}. 

The following observation is important for the proof of Lemma \ref{lem:NonPertCoer}.
\begin{lem}
\label{lem:HeavyConcen}Suppose that $I$ does not contain an $(a_{0},\theta_{0})$-effective
turn of $h$ and there is at least one heavy arc with respect to $h|_{I}$.
Then there exists a measurable set $E\subseteq I$ and a unit vector
$v_{*}$, such that 
\begin{equation}
\mu_{h}(E)\leqslant m_{0}a_{0}\label{eq:LightLength}
\end{equation}
and 
\begin{equation}
\angle\Big(\frac{\dot{h}_{t}}{|\dot{h}_{t}|},v_{*}\Big)\leqslant\theta_{0}+\frac{4\pi}{m_{0}}\label{eq:HeavyConcen}
\end{equation}
for $\mu_{h}$-a.e. $t\in I\backslash E$ with $\dot{h}_{t}\neq0$. 
\end{lem}
\begin{proof}
This is a simple consequence of the definition of effective turn.
Let $w_{k}$ be the center of $P_{k}$ for every $k$. We claim that
\begin{equation}
\angle(w_{k},w_{l})\leqslant\theta_{0}+\frac{2\pi}{m_{0}}\label{eq:HeavClose}
\end{equation}
for any pair $P_{k},P_{l}$ of heavy arcs. Indeed, suppose that (\ref{eq:HeavClose})
is not true for some heavy pair $P_{k},P_{l}$. Let 
\[
A:=\Big\{ t\in I:\dot{h}_{t}\neq0,\ \frac{\dot{h}_{t}}{|\dot{h}_{t}|}\in P_{k}\Big\},\quad B:=\Big\{ t\in I:\dot{h}_{t}\neq0,\ \frac{\dot{h}_{t}}{|\dot{h}_{t}|}\in P_{l}\Big\}.
\]
Then one has
\begin{align*}
\ensuremath{\angle}\Big(\ensuremath{\frac{\dot{h}_{s}}{|\dot{h}_{s}|}},\ensuremath{\frac{\dot{h}_{t}}{|\dot{h}_{t}|}}\Big) & \geqslant\angle(w_{k},w_{l})-\ensuremath{\angle}\Big(\ensuremath{\frac{\dot{h}_{s}}{|\dot{h}_{s}|}},\ensuremath{\frac{\dot{h}_{t}}{|\dot{h}_{t}|}}\Big)\\
 & >\theta_{0}+\frac{2\pi}{m_{0}}-\frac{2\pi}{m_{0}}=\theta_{0}
\end{align*}
for $\mu_{h}\otimes\mu_{h}$-a.e. $(s,t)\in A\times B$. Since $P_{k},P_{l}$
are heavy, it follows that $(A,B)$ provides an $(a_{0},\theta_{0})$-effective
turn of $h$ in $I$. This is a contradiction. Therefore, the estimate
(\ref{eq:HeavClose}) holds.

Now let $v_{*}$ be the center of a given heavy arc (which exists
by assumption). Define 
\[
E:=\Big\{ t\in I:\dot{h}_{t}\neq0,\ \frac{\dot{h}_{t}}{|\dot{h}_{t}|}\in P_{{\rm light}}\Big\},
\]
where $P_{{\rm light}}$ denotes the union of all light arcs. Since
there are at most $m_{0}$ light arcs, one has by definition that
$\mu_{h}(E)\leqslant m_{0}a_{0}.$ This proves the first claim (\ref{eq:LightLength}).
In addition, according to the estimate (\ref{eq:HeavClose}), all
heavy arcs lie inside the cone centered at $v_{*}$ with $S^{1}$-radius
\[
\theta_{0}+\frac{2\pi}{m_{0}}+2\times\frac{\pi}{m_{0}}=\theta_{0}+\frac{4\pi}{m_{0}}.
\]
This proves the second claim (\ref{eq:HeavyConcen}).
\end{proof}
Now we fix the following four parameters:
\begin{equation}
a_{0}=\frac{1}{256}\pi,\quad\theta_{0}=\frac{1}{8}\pi,\quad m_{0}=32,\quad\gamma_{0}=\frac{1}{8}\pi.\label{eq:ParChoice}
\end{equation}
Given $h\in\bar{{\cal H}},$ recall that $X_{t}^{h}$ is the solution
to the control differential equation 
\[
X_{t}^{h}=V(X_{t}^{h})\cdot\dot{h}_{t}dt,\quad X_{0}^{h}=0,
\]
where $V$ is the vector field given by (\ref{eq:2DVFld}). Note that
$V$ is orthonormal at every point. Let $r_{t}:=|X_{t}^{h}|$.
\begin{lem}
\label{lem:ILength}Let $I=[u,v]\subseteq[0,1]$ and $a>2$. Suppose
that $I$ does not contain an $(a_{0},\theta_{0})$-effective turn
of $h$. Suppose further that 
\[
r_{u}=a,\quad r_{v}=a+3\pi
\]
and $r_{t}\in[a,a+3\pi]$ for all $t\in I$. Then $\mu_{h}(I)\geqslant\gamma_{0}a.$ 
\end{lem}
\begin{proof}
We first claim that there is at least one heavy arc with respect to
$h|_{I}.$ Indeed, since $V$ is everywhere orthonormal, one has 
\[
|\dot{X}_{t}^{h}|=|\dot{h}_{t}|\quad\text{for every }t\implies\mu_{h}(I)=\|h\|_{1\text{-var};I}=\|X^{h}\|_{1\text{-var};I}\geqslant3\pi
\]
by assumption. On the other hand, if all arcs were light, according
to (\ref{eq:ParChoice}) one would obtain that 
\[
\mu_{h}(I)\leqslant m_{0}a_{0}=\frac{\pi}{8},
\]
which is a contradiction. Therefore, there is at least one heavy arc.
It follows from Lemma \ref{lem:HeavyConcen} that there exists a measurable
subset $E\subseteq I$ and a unit vector $v_{*}$, such that (\ref{eq:LightLength})
and (\ref{eq:HeavyConcen}) hold.

To prove the lemma, let us write 
\begin{equation}
X_{t}^{h}=r_{t}e^{i\phi_{t}},\dot{h}_{t}=|\dot{h}_{t}|e^{i\alpha_{t}},\quad t\in I.\label{eq:XDotHEqn}
\end{equation}
Here $\phi_{t}$ is any continuous determination of the angular component
of $X_{t}^{h}$; one can fix any $\phi_{u}\in\arg X_{t}^{h}$ and
then determine a unique lifted path $(\phi_{t})_{t\in I}.$ The path
$\alpha_{t}$ is any measurable choice of ${\rm arg}(\dot{h}_{t}/|\dot{h}_{t}|)$,
which is $\mu_{h}$-a.e. defined on $\{t:\dot{h}_{t}\neq0\}$. Since
\[
V(x)=\left(\begin{array}{cc}
\cos|x| & -\sin|x|\\
\sin|x| & \cos|x|
\end{array}\right)
\]
when $|x|\geqslant2$, elementary computation shows that 
\begin{equation}
\dot{r}_{t}=|\dot{h}_{t}|\cos(r_{t}+\alpha_{t}-\phi_{t}),\quad\dot{\phi}_{t}=\frac{1}{r_{t}}|\dot{h}_{t}|\sin(r_{t}+\alpha_{t}-\phi_{t}).\label{eq:RPhiEqn}
\end{equation}
In particular, one has 
\begin{equation}
|\dot{r}_{t}|\leqslant|\dot{h}_{t}|,\quad|\dot{\phi}_{t}|\leqslant\frac{1}{a}|\dot{h}_{t}|\label{eq:DynEstNonPert}
\end{equation}
for all $t\in I$. 

Now suppose on the contrary that $\mu_{h}(I)<\gamma_{0}a.$ The second
relation in (\ref{eq:DynEstNonPert}) implies that 
\[
\sup_{s,t\in I}|\phi_{t}-\phi_{s}|\leqslant\frac{1}{a}\mu_{h}(I)\leqslant\gamma_{0}.
\]
In addition, writing $v_{*}=e^{i\alpha_{*}}$ one knows from (\ref{eq:HeavyConcen})
that 
\[
\angle\big(e^{i\alpha_{t}},e^{i\alpha_{*}}\big)\leqslant\theta_{0}+\frac{4\pi}{m_{0}}
\]
for $\mu_{h}$-a.e. $t\in I\backslash E$. As a result, one has 
\begin{equation}
{\rm dist}_{2\pi}(\alpha_{t}-\phi_{t},\alpha_{*}-\phi_{u})\leqslant\theta_{0}+\frac{4\pi}{m_{0}}+\gamma_{0}\label{eq:AlpPhiDist}
\end{equation}
for $\mu_{h}$-a.e. $t\in I\backslash E$. Here ${\rm dist}_{2\pi}$
denotes the distance modulo $2\pi$ (equivalently, the $S^{1}$-distance).
Let 
\[
\beta_{*}\in[a+m_{0}a_{0},a+3\pi-m_{0}a_{0}]
\]
be a point satisfying 
\begin{equation}
\beta_{*}=\pi-(\alpha_{*}-\phi_{u})\quad\mod2\pi.\label{eq:Beta*}
\end{equation}
Such a $\beta_{*}$ clearly exists since $3\pi-2m_{0}a_{0}>2\pi$
(see (\ref{eq:ParChoice})). It follows that 
\begin{align*}
{\rm dist}_{2\pi}(r_{t}+\alpha_{t}-\phi_{t},\pi) & ={\rm dist}_{2\pi}(r_{t}+\alpha_{t}-\phi_{t},\beta_{*}+(\alpha_{*}-\phi_{u}))\\
 & \leqslant m_{0}a_{0}+\theta_{0}+\frac{4\pi}{m_{0}}+\gamma_{0}\\
 & =\frac{\pi}{8}+\frac{\pi}{8}+\frac{\pi}{8}+\frac{\pi}{8}=\frac{\pi}{2}
\end{align*}
whenever $t\in I\backslash E$ and $r_{t}\in[\beta_{*}-m_{0}a_{0},\beta_{*}+m_{0}a_{0}]$
for $\mu_{h}$-a.e. In particular, 
\begin{equation}
\dot{r}_{t}=|\dot{h}_t|\cos(r_{t}+\alpha_{t}-\phi_{t})\leqslant0\label{eq:NegDotR}
\end{equation}
for all such $t.$ 

To reach a contradiction, let 
\begin{align}
t_{+} & :=\inf\{t\in I:r_{t}=\beta_{*}+m_{0}a_{0}\},\nonumber \\
t_{-} & :=\sup\{t\in I:t<t_{+},r_{t}=\beta_{*}-m_{0}a_{0}\}.\label{eq:tpm}
\end{align}
Both $t_{\pm}$ are well-defined since 
\[
[\beta_{*}-m_{0}a_{0},\beta_{*}+m_{0}a_{0}]\subseteq[a,a+3\pi]
\]
by the construction of $\beta_{*}$. According to the first relation
of (\ref{eq:DynEstNonPert}) along with the estimates (\ref{eq:NegDotR})
and (\ref{eq:LightLength}), one obtains that 
\begin{align*}
2m_{0}a_{0} & =r_{t_{+}}-r_{t_{-}}=\int_{t_{-}}^{t_{+}}\dot{r}_{s}ds\\
 & =\int_{E\cap[t_{-},t_{+}]}\dot{r}_{s}ds+\int_{(I\backslash E)\cap[t_{-},t_{+}]}\dot{r}_{s}ds\\
 & \leqslant\int_{E\cap[t_{-},t_{+}]}\dot{r}_{s}ds\leqslant\mu_{h}(E)\leqslant m_{0}a_{0}.
\end{align*}
This is absurd. Therefore, $\mu_{h}(I)\geqslant\gamma_{0}a.$
\end{proof}
We are now in a position to prove the non-perturbative coercivity
estimate (\ref{eq:NonPertCoer}).

\begin{proof}[Proof of Lemma \ref{lem:NonPertCoer}]

Let $R>4$. Let $h\in\bar{{\cal H}}$ be such that $r_{t}:=|X_{t}^{h}|\geqslant R$
for some $t\in[0,1]$. We decompose the annulus $\{x:R/2\leqslant|x|\leqslant R\}$
into disjoint concentric annuli
\[
A_{j}:=\{x:a_{j}\leqslant|x|\leqslant b_{j}\},\quad a_{j}:=\frac{R}{2}+3\pi(j-1),b_{j}:=a_{j}+3\pi.
\]
There are $N_{R}=\frac{R}{6\pi}$ of them. Using first and last hitting,
one can find a sequence of times 
\[
0<s_{1}<t_{1}<\cdots<s_{N_{R}}<t_{N_{R}}<1,
\]
such that $r_{s_{j}}=a_{j},r_{t_{j}}=b_{j}$ and $r_{t}\in[a_{j},b_{j}]$
for every $j$. Let $I_{j}:=[s_{j},t_{j}]$. Denote 
\[
{\cal B}:=\{j:I_{j}\text{ does not contain an }(a_{0},\theta_{0})\text{-effective turn of }h\}.
\]
According to Lemma \ref{lem:ILength}, one has 
\[
\mu_{h}(I_{j})\geqslant\gamma_{0}a_{j}\geqslant\frac{\gamma_{0}R}{2}
\]
for every $j\in{\cal B}$. It follows that 
\[
\|h\|_{1\text{-var}}\geqslant\sum_{j\in{\cal B}}\mu_{h}(I_{j})\geqslant\frac{\gamma_{0}R}{2}\cdot|{\cal B}|.
\]
Therefore, 
\[
|{\cal B}|\leqslant\frac{2\|h\|_{1\text{-var}}}{\gamma_{0}R}.
\]
The number of $j$'s such that $I_{j}$ contains an $(a_{0},\theta_{0})$-effective
turn is thus bounded below by 
\[
N_{R}-|{\cal B}|\geqslant\frac{R}{6\pi}-\frac{2\|h\|_{1\text{-var}}}{\gamma_{0}R}.
\]
The desired estimate (\ref{eq:NonPertCoer}) now follows from the
definition of ${\cal N}_{{\rm eff}}(h).$

\end{proof}

\subsubsection*{The perturbative estimate}

In this part, we obtain a perturbative version of Lemma \ref{lem:NonPertCoer}
which will easily lead to a proof of Proposition \ref{prop:KeyCoer}.

In what follows, we fix 
\begin{equation}
p=(1/H)^{+},\quad p':=\frac{p}{1-p},\quad q=(p')^{-},\quad\theta:=\frac{p}{q}.\label{eq:PThe}
\end{equation}
It is important to note that $p-1<\theta<1$ and $\theta$ can be
chosen such that 
\[
\theta(H+1/2)<1.
\]
This is possible precisely when $H>1/2$. Our main perturbative estimate
for the number ${\cal N}_{{\rm eff}}(h)$ of effective turns is stated
as follows.
\begin{lem}
\label{lem:TurnEstPert}There exist positive constants $C_{1},C_{2},R_{0}$
depending only on $p,\theta,V$, such that the estimate
\begin{equation}
{\cal N}_{{\rm eff}}(h)\geqslant C_{1}R-C_{2}\Big[M^{p}+M^{1+\theta}+MR^{1-1/p}+\frac{\|h\|_{1\text{-var}}}{R}+M\|h\|_{1\text{-var}}^{\theta}\Big]\label{eq:TurnEstPert}
\end{equation}
holds uniformly over all $R>R_{0}$, all $h\in\bar{{\cal H}}$ and
all $w\in C^{p\text{-var}}([0,1];\mathbb{R}^{2})$ satisfying
\[
\sup_{0\leqslant t\leqslant1}|X_{t}^{w+h}|\geqslant R
\]
and $\|w\|_{p\text{-var}}\leqslant M$.
\end{lem}
\begin{proof}[Proof of Proposition \ref{prop:KeyCoer} assuming Lemma \ref{lem:TurnEstPert}]

Let $R>R_{0}$ and let $(w,h)\in C^{p\text{-var}}([0,1];\mathbb{R}^{2})\times\bar{{\cal H}}$
be given such that 
\[
\sup_{0\leqslant t\leqslant1}|X_{t}^{w+h}|\geqslant R,\quad\|w\|_{p\text{-var}}\leqslant M.
\]
If $\|h\|_{1\text{-var}}>R^{H+1/2},$ one can use the simple inequality
(cf. (\ref{eq:EquivFracNorm}))
\[
\|h\|_{1\text{-var}}\leqslant C_{H}\|h\|_{\bar{{\cal H}}}
\]
to deduce that 
\[
\|h\|_{\bar{{\cal H}}}\geqslant C_{H}^{-1}R^{H+1/2}.
\]
Now consider the case when $\|h\|_{1\text{-var}}\leqslant R^{H+1/2}$.
According to Lemma \ref{lem:CMEffTurn} and Lemma \ref{lem:TurnEstPert},
\begin{align*}
\|h\|_{\bar{{\cal H}}}\geqslant & C_{H}\Big[C_{1}R-C_{2}\Big(M^{p}+M^{1+\theta}+MR^{1-1/p}\\
 & \ \ \ +R^{H-1/2}+MR^{\theta(H+1/2)}\Big)\Big]^{H+1/2}.
\end{align*}
Since $\theta(H+1/2)<1$, one has 
\[
C_{2}\Big(M^{p}+M^{1+\theta}+MR^{1-1/p}+R^{H-1/2}+MR^{\theta(H+1/2)}\Big)\leqslant\frac{C_{1}R}{2}
\]
whenever $R>R_{*}$ for some constant $R_{*}$ depending additionally
on $M$. It follows that 
\[
\|h\|_{\bar{{\cal H}}}\geqslant c_{*}R^{H+1/2}
\]
with a suitable constant $c_{*}>0$ for all $R>R_{M}.$

\end{proof}
\begin{rem}
It is clear from the above proof that the constant $c_{*}$ does not
depend on $M$. 
\end{rem}
The rest of this subsection is devoted to the proof of Lemma \ref{lem:TurnEstPert}.

Given $I\subseteq[0,1]$, we set
\begin{equation}
\omega_{I}:=\|w\|_{p\text{-var};I},\quad l_{I}:=\|h\|_{1\text{-var};I}.\label{eq:OmeL}
\end{equation}
From time to time, we will just write $X_{t}:=X_{t}^{w+h}$ to ease
notation. The first tool we need is the following simple continuity
estimate.
\begin{lem}
\label{lem:WHCty}There exist positive constants $C,\delta$ depending
only on $p$ and $V$, such that the estimate
\begin{equation}
\|X^{w+h}\|_{p\text{-var};I}\leqslant C(\omega_{I}+l_{I})\label{eq:WHCty}
\end{equation}
holds uniformly over all subintervals $I\subseteq[0,1]$, all $h\in C^{1\text{-var}}([0,1];\mathbb{R}^{2})$
and all $w\in C^{p\text{-var }}([0,1];\mathbb{R}^{2})$ satisfying
$\omega_{I}\leqslant\delta$.
\end{lem}
\begin{proof}
Let $I=[u,v]$. One has 
\[
X_{s,t}=\int_{s}^{t}V(X_{r})dw_{r}+\int_{s}^{t}V(X_{r})dh_{r}
\]
for all $s\leqslant t\in I.$ The $dh$-integral is easily estimated
as 
\[
\Big\|\int_{u}^{\cdot}V(X_{t})dh_{t}\Big\|_{p\text{-var};I}\leqslant\|V\|_{\infty}l_{I}.
\]
For the $dw$-integral, one has from \cite[Theorem 6.8]{FV10} (Young's
integral estimate) that 
\[
\Big\|\int_{u}^{\cdot}V(X_{r})dw_{r}\Big\|_{p\text{-var};I}\leqslant C_{p}\big(\|V\|_{\infty}+\|DV\|_{\infty}\omega_{I}\big)\omega_{I}.
\]
Therefore, 
\begin{align}
\|X\|_{p\text{-var};I} & \leqslant\|V\|_{\infty}l_{I}+C_{p}\big(\|V\|_{\infty}+\|DV\|_{\infty}\|X\|_{p\text{-var};I}\big)\omega_{I}\nonumber \\
 & \leqslant C_{1;p,V}(l_{I}+\omega_{I})+C_{2;p,V}\omega_{I}\|X\|_{p\text{-var};I}.\label{eq:pVarEstPf}
\end{align}
Choose $\delta$ to be such that 
\[
C_{2;p,V}\delta\leqslant\frac{1}{2}.
\]
It follows from (\ref{eq:pVarEstPf}) that 
\[
\|X\|_{p\text{-var};I}\leqslant C_{1;p,V}(l_{I}+\omega_{I})+\frac{1}{2}\|X\|_{p\text{-var};I}\iff\|X\|_{p\text{-var};I}\leqslant2C_{1;p,V}(l_{I}+\omega_{I}),
\]
provided that $\omega_{I}\leqslant\delta$. This yields the desired
estimate (\ref{eq:WHCty}).
\end{proof}
The next result provides the main Young perturbative estimates needed
for the proof of Lemma \ref{lem:TurnEstPert}.
\begin{lem}
\label{lem:YoungPert}Let $p,q,\theta$ be given exponents satisfying
(\ref{eq:PThe}). Let $F$ be a uniformly bounded, Lipschitz function
with Lipschitz constant $L_{F}.$ Then there exists a positive constant
$C$ depending only on $p,\theta,V$, such that the following estimate
\begin{equation}
\sup_{J\subseteq I}\Big|\int_{J}F(X_{t}^{w+h})dw_{t}\Big|\leqslant C\big(\|F\|_{\infty}+\|F\|_{\infty}^{1-\theta}L_{F}^{\theta}(\omega_{I}+l_{I})^{\theta}\big)\omega_{I},\label{eq:YoungPert}
\end{equation}
holds uniformly over all subintervals $I\subseteq[0,1],$ all $h\in C^{1\text{-var}}([0,1];\mathbb{R}^{2})$
and all $w\in C^{p\text{-var}}([0,1];\mathbb{R}^{2})$ satisfying
$\omega_{I}\leqslant\delta$, where $\delta$ is the constant appearing
in Lemma \ref{lem:WHCty}. 
\end{lem}
\begin{proof}
According to \cite[Theorem 6.8]{FV10}, one has 
\begin{equation}
\Big|\int_{J}F(X_{t})dw_{t}\Big|\leqslant C_{p,q}\big(\|F\|_{\infty}+\|F(X)\|_{q\text{-var};I}\big)\omega_{I}\label{eq:PertEsPf2}
\end{equation}
for all $J\subseteq I$, where $q\in(p,p')$ is given fixed (see (\ref{eq:PThe})).
On the other hand, by using the simple estimate 
\begin{align*}
|F(X_{t})-F(X_{s})|^{q} & =|F(X_{t})-F(X_{s})|^{q-p}|F(X_{t})-F(X_{s})|^{p}\\
 & \leqslant(2\|F\|_{\infty})^{q-p}(L_{F}|X_{t}-X_{s}|)^{p}
\end{align*}
where $L_{F}$ denotes the Lipschitz constant for $F$, one finds
that 
\[
\|F(X)\|_{q\text{-var};I}\leqslant(2\|F\|_{\infty})^{1-\theta}L_{F}^{\theta}\|X\|_{p\text{-var};I}^{\theta}.
\]
It follows from Lemma \ref{lem:WHCty} that 
\begin{equation}
\|F(X)\|_{q\text{-var};I}\leqslant C_{p,V}(2\|F\|_{\infty})^{1-\theta}L_{F}^{\theta}(\omega_{I}+l_{I})^{\theta},\label{eq:PertEsPf1}
\end{equation}
provided that $\omega_{I}\leqslant\delta$. By substituting (\ref{eq:PertEsPf1})
into (\ref{eq:PertEsPf2}), one obtains the desired estimate (\ref{eq:YoungPert}).
\end{proof}
To proceed further, one needs to examine the radial and angular dynamics
of $X_{t}^{w+h}$. For each $x=re^{i\phi}\in\mathbb{R}^{2},$ we define
\[
e_{1}(x):=\left(\begin{array}{c}
\cos\phi\\
\sin\phi
\end{array}\right),\quad e_{2}(x):=\left(\begin{array}{c}
-\sin\phi\\
\cos\phi
\end{array}\right)
\]
and set 
\[
A(x):=e_{1}(x)^{T}\cdot V(x),\quad B(x):=e_{2}(x)^{T}\cdot V(x).
\]
As a convention, we take $\phi=0$ if $r=0$. In other words, $A(x)$
(respectively, $B(x)$) is the coordinates of $e_{1}(x)$ (respectively,
$e_{2}(x)$) with respect to the orthonormal frame $V(x)$.

Given a control pair $(w,h)$ with $w\in C^{p\text{-var}}([0,1];\mathbb{R}^{2})$
and $h\in\bar{{\cal H}}$, let $r_{t}:=X_{t}^{w+h}$. Suppose that
$r_{t}\geqslant2$ for all $t\in I=[u,v]$. Similar to the equations
(\ref{eq:XDotHEqn}) and (\ref{eq:RPhiEqn}), one can write 
\[
X_{t}^{w+h}=r_{t}e^{i\phi_{t}},\quad\dot{h}_{t}=|\dot{h}_{t}|e^{i\alpha_{t}},
\]
where $(r_{t},\phi_{t})$ satisfies the controlled differential equations
\begin{equation}
dr_{t}=A(X_{t})dw_{t}+\cos(r_{t}+\alpha_{t}-\phi_{t})|\dot{h}_{t}|dt,\label{eq:RPertEqn}
\end{equation}
\begin{equation}
d\phi_{t}=\frac{1}{r_{t}}\big(B(X_{t})dw_{t}+\sin(r_{t}+\alpha_{t}-\phi_{t})|\dot{h}_{t}|dt\big).\label{eq:PhiPertEqn}
\end{equation}

\begin{cor}
\label{cor:YoungPertEst}Define 
\[
\Xi_{I}:=\omega_{I}\big(1+(\omega_{I}+l_{I})^{\theta}\big)
\]
for each $I=[u,v]\subseteq[0,1]$. There exists a positive constant
$K$ depending only on $p,\theta,V$, such that the following two estimates
\[
\sup_{J\subseteq I}\Big|\int_{J}A(X_{t}^{w+h})dw_{t}\Big|\leqslant K\Xi_{I},
\]
\[
\sup_{J\subseteq I}\Big|\int_{J}\frac{1}{r_t}B(X_{t}^{w+h})dw_{t}\Big|\leqslant \frac{K}{a}\Xi_{I}
\]
hold uniformly over all $I\subseteq[0,1],$ $a>2$, $h\in\bar{{\cal H}}$
and $w\in C^{p\text{-var}}([0,1];\mathbb{R}^{2})$ satisfying
$\omega_{I}\leqslant\delta$ and $r_{t}\geqslant a$ for all $t\in I$. 
\end{cor}
\begin{proof}
Take $F(x)=A(x)$ or $B(x)/|x|$ in Lemma \ref{lem:YoungPert} (in the latter case, one shall also apply a smooth cut-off near the origin which is harmless since the estimate only concerns with the region $|x|\geqslant 2$). Then $F$ is uniformly bounded and Lipschitz on $\{x:|x|\geqslant2\}$.
\end{proof}
Now we establish the perturbative version of Lemma \ref{lem:ILength}.
We will continue to use the same definitions, notation and exact values
of parameters (e.g. $a_{0},\theta_{0},m_{0},\gamma_{0}$ satisfying
(\ref{eq:ParChoice})) as in Section \ref{subsec:NonPert}. Recall
that $\delta=\delta_{p,V}$ is the constant appearing in Lemma \ref{lem:WHCty}.
\begin{lem}
\label{lem:ILengthPert}There exist positive constants $A_{0},\eta_{0}$
depending only on $p,\theta,V$, such that the following statement
holds. Given any control pair $(w,h)$, any $I=[u,v]\subseteq[0,1],$
any $a>2$, whenever $(w,h,X^{w+h})$ satisfies 
\begin{equation}
r_{u}=a,r_{v}=a+3\pi,\quad r_{t}\in[a,a+3\pi]\ \text{for all }t\in I,\label{eq:WHAs1}
\end{equation}
\begin{equation}
a\geqslant A_{0},\omega_{I}\leqslant\delta,\Xi_{I}\leqslant\eta_{0}\label{eq:WHAs2}
\end{equation}
and 
\begin{equation}
\mu_{h}(I)<\frac{\gamma_{0}a}{2},\label{eq:WHAs3}
\end{equation}
the interval $I$ must contain an $(a_{0},\theta_{0})$-effective
turn of $h$.
\end{lem}
\begin{proof}
By choosing $\eta_{0}$ small, one can always ensure that there is
at least a heavy arc with respect to $h|_{I}$. Indeed, since $\|A\|_{\infty}=1$,
one finds from the equation (\ref{eq:RPertEqn}) of $r_{t}$ that
\[
3\pi=r_{v}-r_{u}\leqslant K\Xi_{I}+l_{I}\leqslant K\eta_{0}+l_{I},
\]
provided that $\Xi_{I}\leqslant\eta_{0}.$ Suppose that 
\begin{equation}
\eta_{0}<\frac{1}{K}(3\pi-m_{0}a_{0})=\frac{23\pi}{8K}.\label{eq:Eta01}
\end{equation}
Then one has 
\[
l_{I}\geqslant3\pi-K\eta_{0}>m_{0}a_{0}.
\]
This implies by definition that there must be at least one heavy arc
(otherwise, $l_{I}\leqslant m_{0}a_{0}$).

Now suppose that the conditions (\ref{eq:WHAs1}, \ref{eq:WHAs2},
\ref{eq:WHAs3}) are all satisfied while $I$ does not contain an
$(a_{0},\theta_{0})$-effective turn of $h$. We will show that this
leads to a contradiction if $A_{0}$ and $\eta_{0}$ are chosen properly.
First of all, by using the given conditions, the $\phi$-equation
(\ref{eq:PhiPertEqn}) and Corollary \ref{cor:YoungPertEst}, one
finds that 
\[
\sup_{s,t\in I}|\phi_{t}-\phi_{s}|\leqslant\frac{K}{a}\Xi_{I}+\frac{1}{a}\mu_{h}(I)\leqslant\frac{K\eta_{0}}{A_{0}}+\frac{\gamma_{0}}{2}.
\]
By requiring that 
\begin{equation}
\frac{K\eta_{0}}{A_{0}}<\frac{\gamma_{0}}{2},\label{eq:A0}
\end{equation}
one has 
\[
\sup_{s,t\in I}|\phi_{t}-\phi_{s}|<\gamma_{0}.
\]
In addition, according to Lemma \ref{lem:HeavyConcen}, there exist
a measurable subset $E\subseteq I$ and a unit vector $v_{*}=e^{i\alpha_{*}}$
such that (\ref{eq:LightLength}) and (\ref{eq:HeavyConcen}) holds.
One can then carry the argument of Lemma \ref{lem:ILength} word by
word to show that 
\begin{equation}
{\rm dist}_{2\pi}(r_{t}+\alpha_{t}-\phi_{t},\pi)\leqslant\frac{\pi}{2}\label{eq:NegCosPertPf}
\end{equation}
for $\mu_{h}$-a.e. $t\in I\backslash E$ such that $\dot{h}_{t}\neq0$
and $r_{t}\in[\beta_{*}-m_{0}a_{0},\beta_{*}+m_{0}a_{0}]$, where
$\beta_{*}$ is defined according to the relation (\ref{eq:Beta*}).
Define $t_{\pm}$ as in (\ref{eq:tpm}). It follows from the $r$-equation
(\ref{eq:RPertEqn}), Corollary \ref{cor:YoungPertEst}, the estimates
(\ref{eq:LightLength}) and (\ref{eq:NegCosPertPf}) that 

\begin{align*}
2m_{0}a_{0} & =r_{t_{+}}-r_{t_{-}}=\int_{t_{-}}^{t_{+}}dr_{s}\\
 & =\int_{t_{-}}^{t_{+}}A(X_{s})dw_{s}+\int_{t_{-}}^{t_{+}}\cos(r_{s}+\alpha_{s}-\phi_{s})|\dot{h}_{s}|ds\\
 & \leqslant K\Xi_{I}+\mu_{h}(E)\leqslant K\eta_{0}+m_{0}a_{0}.
\end{align*}
By requiring $\eta_0$ to satisfy
\begin{equation}
K\eta_{0}<m_{0}a_{0},\label{eq:Eta02}
\end{equation}
one reaches a contradiction. 

To summarise, one first chooses $\eta_{0}$ that satisfies (\ref{eq:Eta01},
\ref{eq:Eta02}) and then chooses $A_{0}$ to satisfy (\ref{eq:A0}).
With such choices, for any $(I,a;w,h)$ satisfying the conditions
(\ref{eq:WHAs1}, \ref{eq:WHAs2}, \ref{eq:WHAs3}), the interval
$I$ must contain an $(a_{0},\theta_{0})$-effective turn of $h$.
This completes the proof of the lemma. 
\end{proof}
Finally, we are in a position to prove Lemma \ref{lem:TurnEstPert}.

\begin{proof}[Proof of Lemma \ref{lem:TurnEstPert}]

We continue to use the same notation as in the proof of Lemma \ref{lem:NonPertCoer}
(the annuli decomposition $[a_{j},b_{j}]$ and radial crossing $I_{j}=[s_{j},t_{j}]$).
Let $R>2A_{0}$ ($A_{0}$ is the constant appearing in Lemma \ref{lem:ILengthPert}
so that $a_{j}>A_{0}$ for every $j$). Our goal is to estimate the
number of $I_{j}$'s which do not contain an $(a_{0},\theta_{0})$-effective
turn of $h$. According to Lemma \ref{lem:ILengthPert}, at least
one of the constraints in (\ref{eq:WHAs2}, \ref{eq:WHAs3}) fails
for such $I_{j}$'s. We discuss the possibilities as follows.

\vspace{2mm}\noindent (i) \uline{\mbox{$\omega_{I_{j}}>\delta$}.} 

\vspace{2mm} Since $\|w\|_{p\text{-var}}\leqslant M$, one has 
\[
\delta^{p}\cdot\big|\{j:\omega_{I_{j}}>\delta\}\big|\leqslant\sum_{j=1}^{N_{R}}\omega_{I_{j}}^{p}\leqslant M^{p}.
\]
As a result, 
\begin{equation}
\big|\{j:\omega_{I_{j}}>\delta\}\big|\leqslant\delta^{-p}M^{p}.\label{eq:C1}
\end{equation}
(ii) \uline{\mbox{$l_{I_{j}}>\gamma_{0}a_{j}/2$}.}

\vspace{2mm} Since $a_{j}\geqslant R/2$, one has 
\[
\frac{\gamma_{0}R}{4}\cdot\big|\{j:l_{I_{j}}>\gamma_{0}a_{j}/2\}\big|\leqslant\sum_{j=1}^{N_{R}}l_{I_{j}}\leqslant\|h\|_{1\text{-var}}.
\]
Therefore, 
\begin{equation}
\big|\{j:l_{I_{j}}>\gamma_{0}a_{j}/2\}\big|\leqslant\frac{4\|h\|_{1\text{-var}}}{\gamma_{0}R}.\label{eq:C2}
\end{equation}
(iii) \uline{\mbox{$\Xi_{I_{j}}>\eta_{0}$}.} 

\vspace{2mm} By definition, one has 
\begin{equation}
\sum_{j=1}^{N_{R}}\Xi_{I_{j}}\leqslant\sum_{j=1}^{N_{R}}\omega_{I_{j}}+\sum_{j=1}^{N_{R}}\omega_{I_{j}}l_{I_{j}}^{\theta}+\sum_{j=1}^{N_{R}}\omega_{I_{j}}^{1+\theta}.\label{eq:SumXi}
\end{equation}
The first sum in (\ref{eq:SumXi}) is estimated as 
\[
\sum_{j=1}^{N_{R}}\omega_{I_{j}}\leqslant N_{R}^{1-1/p}\Big(\sum_{j=1}^{N_{R}}\omega_{I_{j}}^{p}\Big)^{1/p}\leqslant MN_{R}^{1-1/p}.
\]
For the second sum in (\ref{eq:SumXi}), one can first use H\"older's
inequality to obtain that 
\[
\sum_{j=1}^{N_{R}}\omega_{I_{j}}l_{I_{j}}^{\theta}\leqslant M\Big(\sum_{j=1}^{N_{R}}l_{I_{j}}^{\theta p'}\Big)^{1/p'}.
\]
Next, observe that $\theta p'>1$. Since the $l^{p}$-norm is nonincreasing
in $p$, one has 
\[
\Big(\sum_{j=1}^{N_{R}}l_{I_{j}}^{\theta p'}\Big)^{1/p'}\leqslant\Big(\sum_{j=1}^{N_{R}}l_{I_{j}}\Big)^{\theta}\leqslant\|h\|_{1\text{-var}}^{\theta}.
\]
As a consequence, 
\[
\sum_{j=1}^{N_{R}}\omega_{I_{j}}l_{I_{j}}^{\theta}\leqslant M\|h\|_{1\text{-var}}^{\theta}.
\]
The last sum in (\ref{eq:SumXi}) again easily estimated by monotonicity
of $l^{p}$-norm; since $1+\theta>p$, one has 
\[
\sum_{j=1}^{N_{R}}\omega_{I_{j}}^{1+\theta}\leqslant\Big(\sum_{j=1}^{N_{R}}\omega_{I_{j}}^{p}\Big)^{\frac{1+\theta}{p}}\leqslant M^{1+\theta}.
\]
To summarise, one concludes that 
\[
\sum_{j=1}^{N_{R}}\Xi_{I_{j}}\leqslant M\Big(\frac{R}{6\pi}\Big)^{1-1/p}+M\|h\|_{1\text{-var}}^{\theta}+M^{1+\theta}.
\]
It follows that 
\begin{equation}
\big|\{j:\Xi_{I_{j}}>\eta_{0}\}\big|\leqslant\eta_{0}^{-1}\Big[M\Big(\frac{R}{6\pi}\Big)^{1-1/p}+M\|h\|_{1\text{-var}}^{\theta}+M^{1+\theta}\Big].\label{eq:C3}
\end{equation}
The number of $j$'s such that $I_{j}$ does not contain an $(a_{0},\theta_{0})$-effective
turn is therefore bounded above by the sum of the RHS of (\ref{eq:C1},
\ref{eq:C2}, \ref{eq:C3}). Since $N_{R}=\frac{R}{6\pi}$, this immediately
yields a lower bound of ${\cal N}_{{\rm eff}}(h)$ in the form of
(\ref{eq:TurnEstPert}). The proof of Lemma \ref{lem:TurnEstPert}
is now complete. 

\end{proof}

\subsubsection{The  $(2H+1)$-Weibull upper estimate}

We now proceed to prove Theorem \ref{thm:2DCounter} by using the
coercivity estimate given by Proposition \ref{prop:KeyCoer}. The
argument is a simple application of Borell's isoperimetric inequality. 

\begin{proof}[Proof of Theorem \ref{thm:2DCounter}]

Let $({\cal W},\bar{{\cal H}},\mu)$ be the abstract Wiener space
associated with the 2D fBM (in particular, $\mu$ is the law of fBM).
Fix $p\in(1/H)^{+}$ as before and fix $M>0$ (its precise value is
of no importance). Define
\[
G_{M}:=\big\{ w\in{\cal W}:\|w\|_{p\text{-var}}\leqslant M\big\}.
\]
For each $R>0,$ we also set 
\[
A_{R}:=\Big\{ w\in{\cal W}:\sup_{0\leqslant t\leqslant1}|X_{t}^{w}|>R\Big\}
\]
According to Proposition \ref{prop:KeyCoer}, there exist positive
constants $c_{0},R_{0},$ such that 
\[
G_{M}+c_{0}R^{H+1/2}{\cal K}\subseteq A_{R}^{c}
\]
for all $R>R_{0},$ where ${\cal K}$ denotes the closed unit ball
of $\bar{{\cal H}}.$ Let 
\[
\kappa:=\Phi^{-1}(\mu(G_{M}))\in\mathbb{R},
\]
where $\Phi(u):=\frac{1}{\sqrt{2\pi}}\int_{-\infty}^{x}e^{-u^{2}/2}du$.
It follows from Borell's isoperimetric inequality (see \cite[Theorem 4.3]{Led94})
that 
\begin{align*}
\mu(A_{R}) & =1-\mu(A_{R}^{c})\leqslant1-\mu\big(G_{M}+c_{0}R^{H+1/2}{\cal K}\big)\\
 & \leqslant1-\Phi(\kappa+c_{0}R^{H+1/2})\leqslant C_{1}e^{-C_{2}R^{2H+1}}
\end{align*}
for all $R>R_{0}$ with suitable constants $C_{1},C_{2}>0.$ This
completes the proof of Theorem \ref{thm:2DCounter}.

\end{proof}

\newpage
\section{General lower tail estimates in Young's regime}\label{sec:LowGauCondYoung}

In this section, we consider the SDE (\ref{eq:SDEIntro}) with $H\in(1/2,1)$.
Our aim is to provide a simple sufficient condition on the vector fields, under which the solution admits a Gaussian lower tail. The main result is Theorem \ref{thm:GLowH}. As a quite straightforward application of the technique developed in the proof, we will also establish a general $(2H+1)$-Weibull lower tail for uniformly elliptic systems. This is the content of Theorem \ref{thm:GenLowYoung}.

\subsection{A sufficient condition for  Gaussian lower tail in Young's regime}

Let $X_t$ be the solution to the SDE (\ref{eq:SDEIntro}) with $H\in(1/2,1)$. Our main sufficient condition on the vector fields to ensure Gaussian lower tail is formulated as follows. 

\vspace{2mm}\noindent (\textbf{H}) There exists a function $F\in C^{2}(\mathbb{R}^{N})$
with uniformly bounded first two derivatives and a unit vector $\eta\in S^{d-1},$
such that 
\begin{equation}\label{eq:BalCond}
m:=\inf_{x\in\mathbb{R}^{N}}(W_{\eta}F)(x)>0
\end{equation}
where $W_\eta(x):= V(x)\eta := \sum_{i=1}^d V_i(x)\eta^i.$ 

\begin{rem}

Geometrically, Condition (H) asserts that level sets of the function $F$ are uniformly transversal to integral curves of the vector field $W_\eta$. 

\end{rem}

\vspace{2mm}\noindent The main result of this subsection is now stated as
follows. 
\begin{thm}
\label{thm:GLowH}Let $H\in(1/2,1)$ and $V_0,V_1,\cdots,V_d\in C_b^2$. Suppose that Condition (H) holds. Then there exist
positive constants $C_{1},C_{2}$ depending only on $H$ and the vector fields, such that
the estimate
\[
\mathbb{P}\big(\big|X_{1}-x_{0}\big|>R\big)\geqslant C_{1}e^{-C_{2}R^{2}}
\]
holds for all large $R$.
\end{thm}
\begin{proof}
As usual, we use $X_{t}^{w}$ to denote the solution to the same equation
(\ref{eq:SDEIntro}) with $B_{t}$ replaced by $w_{t}$. Given $R>0,$
we set $h_{t}^{R}:=Rt\eta$ ($0\leqslant t\leqslant1$) where $\eta$
is the unit vector given by Condition (H). This particular Cameron-Martin
path will play a critical role in the argument. Let $F$ be the function
given by Condition (H).

\vspace{2mm}\noindent (i) The key claim is that there exist positive
constants $M,C_{0},R_{0}$ depending only on $p,H,V,F$, such that
\begin{equation}
\|w\|_{p\text{-var}}\leqslant M\implies\big|X_{1}^{w+h^{R}}-x\big|\geqslant C_{0}R\label{eq:DispEstYoung}
\end{equation}
for all $R>R_{0}.$ 

To prove this claim, let us write $X_{t}:=X_{t}^{w+h^{R}}$ and $V=(V_{1},\cdots,V_{d})$
for simplicity. Note that $X_{t}$ satisfies the following equation:
\[
dX_{t}=RW_{\eta}(X_{t})dt+V(X_{t})dw_{t}+V_{0}(X_{t})dt.
\]
Let $F$ be the function given by Condition (H). By applying the chain
rule to $F(X_{t})$, one obtains that 
\begin{align*}
F(X_{1})-F(x_{0})= & R\int_{0}^{1}(W_{\eta}F)(X_{t})dt+\int_{0}^{1}\big(\nabla F(X_{t})\big)^{T}V(X_{t})dw_{t}\\
 & \ \ \ +\int_{0}^{1}\big(\nabla F(X_{t})\big)^{T}V_{0}(X_{t})dt=:A+B+C.
\end{align*}
According to Condition (H), one has $A\geqslant mR$. It is also obvious
that $|C|\leqslant\|DF\|_{\infty}\|V_{0}\|_{\infty}.$ 

To estimate $B$, let us set 
\[
x_{I}:=\|X\|_{p\text{-var};I},\quad\omega_{I}:=\|\omega\|_{p\text{-var};I}
\]
for any given interval $I\subseteq[0,1]$. By using \cite[Theorem 6.8]{FV10},
one easily finds that 
\[
x_{I}\leqslant\big(R\|V\|_{\infty}+\|V_{0}\|_{\infty}\big)|I|+C_{p}\big(\|V\|_{\infty}+\|DV\|_{\infty}x_{I}\big)\omega_{I}.
\]
Let $M$ be chosen such that 
\[
C_{p}\|DV\|_{\infty}M<\frac{1}{2}.
\]
Whenever $\|w\|_{p\text{-var}}\leqslant M$, one has 
\begin{equation}
x_{I}\leqslant2\big(R\|V\|_{\infty}+\|V_{0}\|_{\infty}\big)|I|+2C_{p}\|V\|_{\infty}M.\label{eq:xIEst}
\end{equation}
Now we divide the interval $[0,1]$ into equal subintervals $I_{1},\cdots,I_{N}$
($N\asymp R$) each having length $1/R$. According to (\ref{eq:xIEst}),
on each subinterval $I_{j}$ one has 
\begin{equation}
x_{I_{j}}\leqslant C_{p,V}\label{eq:xIjEst}
\end{equation}
with some constant $C_{p,V}$ depending on $p,V.$ It follows
again from \cite[Theorem 6.8]{FV10} and (\ref{eq:xIjEst}) that 
\begin{align*}
|B| & \leqslant\sum_{j=1}^{N}\Big|\int_{I_{j}}\big(\nabla F(X_{t})\big)^{T}V(X_{t})dw_{t}\Big|\\
 & \leqslant C_{p}\big(\|(\nabla F)^{T}\cdot V\|_{\infty}+C_{p,V}\big\|D\big((\nabla F)^{T}\cdot V\big)\big\|_\infty\big)\sum_{j=1}^{N}\omega_{I_{i}}\\
 & \leqslant C'_{p,V,F}R^{1-1/p}\Big(\sum_{j=1}^{N}\omega_{I_{j}}^{p}\Big)^{1/p}\\
 & \leqslant C'_{p,V,F}MR^{1-1/p}.
\end{align*}
As a consequence, 
\begin{align*}
\|\nabla F\|_{\infty}|X_{1}-x_{0}| & \geqslant F(X_{1})-F(x_{0})\\
 & \geqslant mR-C'_{p,V,F}MR^{1-1/p}-\|DF\|_{\infty}\|V_{0}\|_{\infty}.
\end{align*}
Let $R_{0}$ be chosen such that the RHS is larger than $mR/2$ for
all $R>R_{0}.$ The claim (\ref{eq:DispEstYoung}) thus follows with
$C_{0}:=\frac{m}{2\|\nabla F\|_{\infty}}.$

\vspace{2mm}\noindent (ii) The claim (\ref{eq:DispEstYoung}) shows
that 
\[
G_{M}:=\big\{ w:\|w\|_{p\text{-var}}\leqslant M\big\}\subseteq\Big\{\big|X_{1}^{w+h^{R}}-x_{0}\big|\geqslant C_{0}R\Big\}
\]
for all $R>R_{0}.$ As a consequence, 
\[
\mathbb{P}\big(\big|X_{1}^{B}-x_{0}\big|\geqslant C_{0}R\big)\geqslant\mu(h^{R}+G_{M}),
\]
where $\mu$ is the law of fBM on path space. By exactly the same
Cameron-Martin transformation argument as in the proof of Theorem
\ref{thm:RoughLowG}, one finds that 
\[
\mu(h^{R}+G_{M})=e^{-\|h^{R}\|_{\bar{{\cal H}}}^{2}/2}\mathbb{E}\big[e^{-\int_{0}^{1}\langle\theta_{t}^{R},dW_{t}\rangle}{\bf 1}_{G_{M}}(B)\big].
\]
Here $\theta^{R}:=K_{H}^{-1}h^{R}\in L^{2}([0,1];\mathbb{R}^{d})$
($K_{H}$ is the isomorphism defined by (\ref{eq:KHIso})) and $W$
is the Wiener process in the fBM representation (\ref{eq:fBMRep}).
By the definition of $h^{R},$ it is obvious that 
\begin{equation}\label{eq:hRNorm}
\|h^{R}\|_{\bar{{\cal H}}}=R\|[t\mapsto t\eta]\|_{\bar{{\cal H}}}=:\kappa R.
\end{equation}
In addition, due to the reflection symmetry of the set $G_{M}$, one
has 
\[
\mathbb{E}\big[e^{-\int_{0}^{1}\langle\theta_{t}^{R},dW_{t}\rangle}{\bf 1}_{G_{M}}(B)\big]=\mathbb{E}\big[e^{\int_{0}^{1}\langle\theta_{t}^{R},dW_{t}\rangle}{\bf 1}_{G_{M}}(B)\big].
\]
In particular, 
\begin{align*}
\mathbb{E}\big[e^{-\int_{0}^{1}\langle\theta_{t}^{R},dW_{t}\rangle}{\bf 1}_{G_{M}}(B)\big] & =\mathbb{E}\Big[\cosh\Big(\int_{0}^{1}\langle\theta_{t}^{R},dW_{t}\rangle\Big){\bf 1}_{G_{M}}(B)\Big]\\
 & \geqslant\mu(G_{M}):=\delta.
\end{align*}
As a result, one concludes that 
\[
\mathbb{P}\big(\big|X_{1}^{B}-x_{0}\big|\geqslant C_{0}R\big)\geqslant\delta e^{-\kappa^{2}R^{2}/2}
\]
for all $R>R_{0}$. 

Now the proof of Theorem \ref{thm:GLowH} is complete. 
\end{proof}

\subsection{General $(2H+1)$-Weibull lower tail for uniformly elliptic systems}

We now prove a general $(2H+1)$-Weibull lower tail estimate under
the uniform ellipticity assumption. Parallel to the rough regime in
Section \ref{sec:GenLowRough}, this part is also a simple adaptation of the proof of Theorem \ref{thm:GLowH}.
\begin{thm}\label{thm:GenLowYoung}
Consider the SDE (\ref{eq:SDEIntro}) with $H\in(1/2,1).$ Suppose that
the vector fields are $C_{b}^{\infty}$ and uniformly elliptic. Then
there exist positive constants $C_{1},C_{2}$ depending only on $H$
and the vector fields, such that the estimate 
\begin{equation}
\mathbb{P}\big(|X_{1}-x_{0}|>R\big)\geqslant C_{1}e^{-C_{2}R^{2H+1}}\label{eq:WeibullGenYoung}
\end{equation}
holds for all large $R$.
\end{thm}
\begin{proof}
The argument is a simple adaptation of the proof of Theorem \ref{thm:GLowH}.
It is more convenient to use H\"older norms here. Fix a unit vector
$\eta\in\mathbb{R}^{N-1}$ and $\rho\in(1/2,H)$. Given $R>1$, let
$X_{t}^{R}$ denote the solution to the differential equation 
\[
dX_{t}^{R}=R\eta dt+V(X_{t}^{R})dB_{t}+V_{0}(X_{t}^{R})dt,\quad X_{0}^{R}=x_{0}.
\]
Similar to Part (i) of the proof of Theorem \ref{thm:GLowH} with
$p$-variation replaced by $\rho$-H\"older norm (cf. Lemma \ref{lem:LinDisplace}),
there exist positive constants $M,C_{1},C_{2},R_{0},$ such that on
the event 
\[
A_{M}:=\big\{\|B\|_{\rho\text{-H\"ol};[0,1]}\leqslant M\big\},
\]
the estimates
\begin{equation}
\big|X_{t}^{R}-X_{s}^{R}\big|\leqslant C_{1}\big(R|t-s|+M|t-s|^{\rho}+|t-s|\big)\label{eq:CtyEstGenLowYoung}
\end{equation}
and 
\begin{equation}
|X_{1}^{R,w}-x_{0}|\geqslant C_{2}R\label{eq:GloDispGenYoung}
\end{equation}
hold for all $R>R_{0}$ and $s,t\in[0,1]$. 

Now let us define the (random) Cameron-Martin path
\[
h_{t}^{R}:=R\int_{0}^{t}a(X_{s}^{R})ds,
\]
where $a(x):=A(x)\eta$ and $A(x)$ is the right inverse of $V(x)$
defined by (\ref{eq:RightInv}). Note that $h^{R}$ is ${\cal F}^{W}$-adapted
(cf. (\ref{eq:fBMRep})). According to the norm-equivalence relation
(\ref{eq:EquivFracNorm}), one has 
\begin{align*}
\|h^{R}\|_{\bar{{\cal H}}}^{2} & \leqslant C_{H}\Big(\|\dot{h}^{R}\|_{L^{2}}^{2}+\int_{0}^{1}\int_{0}^{1}\frac{|\dot{h}_{t}^{R}-\dot{h}_{s}^{R}|^{2}}{|t-s|^{2H}}dsdt\Big)\\
 & =C_{H}R^{2}\Big(\|u^{R}\|_{L^{2}}^{2}+\int_{0}^{1}\int_{0}^{1}\frac{|\dot{u}_{t}^{R}-\dot{u}_{s}^{R}|^{2}}{|t-s|^{2H}}dsdt\Big)
\end{align*}
where $u_{t}^{R}:=a(X_{t}^{R})$. The $L^{2}$-norm in the first term
is uniformly bounded by $\|A\|_{\infty}$. For the second term, by
using (\ref{eq:CtyEstGenLowYoung}) one easily finds that 
\[
\big|u_{t}^{R}-u_{s}^{R}\big|\leqslant C_{3}\big(\min\{1,R|t-s|\}+M|t-s|^{\rho}\big)
\]
for all $R>R_{0}$ and $s,t\in[0,1]$. It follows from explicit analysis
that 
\begin{align*}
 & \int_{0}^{1}\int_{0}^{1}\frac{|\dot{u}_{t}^{R}-\dot{u}_{s}^{R}|^{2}}{|t-s|^{2H}}dsdt\\
 & \leqslant C_{3}\int_{0}^{1}\int_{0}^{1}\Big(\frac{\min\{1,R^{2}|t-s|^{2}\}}{|t-s|^{2H}}+M|t-s|^{2\rho-2H}\Big)\\
 & \leqslant C_{4}R^{2H-1}
\end{align*}
on the event $A_{M}.$ As a consequence, 
\begin{equation}
\|h^{R}\|_{\bar{{\cal H}}}^{2}\leqslant C_{H}R^{2}\big(\|A\|_{\infty}^{2}+C_{4}R^{2H-1}\big)\leqslant C_{5}R^{2H+1}\label{eq:CMGenYoung}
\end{equation}
on $A_{M}.$

Let \[\tau:=\inf\big\{t\geqslant 0:\|B\|_{\rho\text{-H\"ol};[0,t]}>M\big\}\wedge 1\] and define 
\[\hat{h}^R_t := R\int_{0}^{t\wedge \tau}a(X_{s}^{R})ds.\]Note that $\hat{h}^R$ is also $\mathcal{F}^W$-adapted and $\hat{h}^R = h^R$ on $A_M$. From the above analysis, it is not hard to see that with probability one, 
\begin{equation}\label{eq:CMGenYoung}
\big\|\dot{\hat{h}}^R\big\|_{\bar{\cal H}}\leqslant C_6 R^{2H+1}
\end{equation}for all $R$. With the global Cameron-Martin estimate (\ref{eq:CMGenYoung}), one can now carry through the Cameron-Martin transformation argument as in the
proof of Theorem \ref{thm:GLowH} (cf. the proof of Theorem \ref{thm:RoughLowG}) to conclude the desired $(2H+1)$-Weibull
lower bound (\ref{eq:WeibullGenYoung}). 
\end{proof}

\newpage
\section{2D periodic elliptic systems always have Gaussian lower tail}\label{sec:2DPeriodic}

In this section, we prove that a $2\times 2$ periodic, uniformly
elliptic system always possesses a Gaussian lower tail, which is a striking
result in contrast to the radially periodic counterexample of Theorem
\ref{thm:2DCounter}.

\begin{thm}
\label{thm:2DPeriodic}Let $V_{1},V_{2}$ be given smooth vector fields
on the torus $\mathbb{T}^{2}=\mathbb{R}^{2}/\mathbb{Z}^{2}.$ Suppose
that they are uniformly elliptic, namely, 
\[
\dim{\rm Span}\{V_{1}(x),V_{2}(x)\}=2
\]
for every $x\in\mathbb{T}^{2}$. Consider the SDE 
\begin{equation}
dX_{t}=\tilde{V}_{1}(X_{t})dB_{t}^{1}+\tilde{V}_{2}(X_{t})dB_{t}^{2},\quad X_{0}=x_{0}\in\mathbb{R}^{2}\label{eq:SDETorus}
\end{equation}
on the universal cover $\mathbb{R}^{2}.$ Here $\tilde{V}_{i}$ is
the canonical periodic lifting of $V_{i}$ on $\mathbb{R}^{2}$ and
$B_{t}=(B_{t}^{1},B_{t}^{2})$ is an fBM with Hurst parameter $H\in(1/2,1)$.
Then there exist positive constants $C_{1},C_{2}$, such that 
\begin{equation}
\mathbb{P}\big(|X_{1}-x_{0}|>R\big)\geqslant C_{1}e^{-C_{2}R^2}\label{eq:GaussLowTorus}
\end{equation}
for all large $R.$ 
\end{thm}

\subsection{Idea of proof}\label{sec:TorusStrategy}

From the proof of Theorem \ref{thm:GLowH}, the main mechanism leading
to a Gaussian lower tail is the existence of a Cameron-Martin control
$h^{R}$ for each $R$, such that the following two properties hold
true:

\vspace{2mm}\noindent (i) $\|h^{R}\|_{\bar{{\cal H}}}=O(R)$;

\vspace{2mm}\noindent (ii) Whenever $\|w\|_{p\text{-var}}\leqslant M$,
one has 
\begin{equation}
|X_{1}^{w+h^{R}}-x_{0}|\geqslant C_{M}\cdot R\label{eq:RobPert}
\end{equation}
for all large $R$. In other words, there always exists a control
with Cameron-Martin norm $O(R)$ which creates a ballistic tunnel
for the dynamics to travel a linear distance that is robust under
perturbation. 

Our main strategy for establishing the aforementioned mechanism relies on the theory of dynamical systems, in particular,
a fundamental classification theorem for torus flows (see Theorem
\ref{thm:1DFoliClass} below). For each $u\in S^{1},$ the nonvanishing
vector field $W_{u}:=V\cdot u$ generates a 1D foliation ${\cal F}_{u}$
of $\mathbb{T}^{2}$ (the family of integral curves of $W_{u}$).
The classification theorem asserts that there are precisely two possibilities
for ${\cal F}_{u}$.

\vspace{2mm}\noindent \textit{Case I: ${\cal F}_{u}$ is conjugate
to a suspension} (see Definition \ref{def:SuspensionMain} and Definition
\ref{def:Conj}). In this case, one can easily construct a closed
$1$-form $\eta$ on $\mathbb{T}^{2}$ with nontrivial cohomology,
whose integral along the lifted flow of $W_{u}$ on the universal
cover $\mathbb{R}^{2}$ produces global linear displacements. In this
regime, one is precisely in the setting of Condition (H) in Theorem
\ref{thm:GLowH}.

\vspace{2mm}\noindent \textit{Case II: ${\cal F}_{u}$ contains
a Reeb component} (see Definition \ref{def:Reeb}). In this case,
$W_{u}$ must have a periodic orbit with nontrivial homology. The
lifting of this orbit to the universal cover $\mathbb{R}^{2}$ naturally
provides a ballistic trajectory for the deterministic dynamics controlled
by $h_{t}^{R}:=Rtu.$ However, a major difficulty in this regime is
that such a ballistic trajectory may not necessarily create a ballistic
tunnel for the stochastic dynamics perturbed by fractional Brownian
paths. It is necessary to establish perturbative estimates like (\ref{eq:RobPert})
in order to turn the deterministic Cameron-Martin estimate into a
probabilistic Gaussian lower bound (cf. Section \ref{sec:LowGauCondYoung}).
For this purpose, one needs some dynamical stability of the periodic
orbit. 

To overcome the aforementioned difficulty, a key insight is that one
can choose another unit vector $u^{*}$ that is close to $u$, such
that the flow lines of $W_{u^{*}}$ contain a \textit{hyperbolic}
periodic orbit (see Definition \ref{def:HypOrb}). Hyperbolicity will
enforce an attraction property for such an orbit, which provides the
exact level of stability that enables one to establish a robust perturbative
estimate. 

\vspace{2mm} The rest of this section is devoted to the
development of the above strategy in a precise mathematical way.

\subsection{A classification theorem for 1D foliations on $\mathbb{T}^{2}$ }

In this subsection, we recall a fundamental classification theorem
for nondegenerate torus flows. We
begin by recalling the definition of a foliation (see \cite{Lee13}).
\begin{defn}\label{def:Folia}
Let $M$ be an $n$-dimensional differentiable manifold. Let ${\cal F}$
be a collection of immersed $k$-dimensional submanifolds. A chart
$(U,\varphi)$ is said to be \textit{flat} for ${\cal F}$ if the
following two properties hold:

\vspace{2mm}\noindent (i) $\varphi(U)$ is a cube in $\mathbb{R}^{n}$;

\vspace{2mm}\noindent (ii) for each submanifold $N\in{\cal F}$ with
nonempty intersection with $U$, the set $\varphi(N\cap U)$ is a
countable union of $k$-dimensional slices of the form 
\[
x^{k+1}=c^{k+1},\cdots,x^{n}=c^{n}
\]
in $\varphi(U).$ 

\vspace{2mm}\noindent A \textit{$k$-dimensional (smooth) foliation}
on $M$ is a collection ${\cal F}$ of disjoint, connected, immersed
$k$-dimensional submanifolds such that the union of members of ${\cal F}$
is $M$ and for every $x\in M$, there exists a chart around $x$
which is flat for ${\cal F}.$ Members of a foliation ${\cal F}$
are called \textit{leaves} of ${\cal F}.$ 
\end{defn}
\begin{example}
Let $M$ be a closed surface. Let $W$ be a smooth vector field on
$M$ which is everywhere nonvanishing. According to the global Frobenius
theorem, the integral curves of $W$ form a one-dimensional (1D) foliation
on $M$ (see \cite[Theorem 19.21]{Lee13}). This is the main example
that is relevant to us. 
\end{example}
From now on, we restrict ourselves to the case $M=\mathbb{T}^{2}.$
We need two more concepts in order to state the classification theorem
of 1D foliations on $\mathbb{T}^{2}$. The first one is the notion
of a Reeb component (see \cite[Definition 2.4]{BZ17} and \cite[Chapter 1, Section 3.3]{HH86}). 

\begin{defn}\label{def:Reeb}

Let ${\cal F}$ be a 1D foliation on $\mathbb{T}^{2}.$ We say that
${\cal F}$ \textit{has a Reeb component} if there exists an embedded
annulus $A\cong S^{1}\times[0,1]$ inside $\mathbb{T}^{2}$ such that:

\vspace{2mm}\noindent (i) $\partial A$ is the union of two compact
leaves of ${\cal F}$ that are in opposite homotopy classes;

\vspace{2mm}\noindent (ii) there are no compact leaves of ${\cal F}$
in the interior of $A$.

\end{defn}

The next one is the notion of a suspension (see \cite[Chapter 1, Section 3.1]{HH86}).
Let $f:S^{1}\rightarrow S^{1}$ be an orientation-preserving diffeomorphism.
The integer group $\mathbb{Z}$ acts on $S^{1}\times\mathbb{R}$ by
\begin{equation}
k\cdot(x,t):=\big(f^{k}(x),t-k\big),\quad k\in\mathbb{Z},(x,t)\in S^{1}\times\mathbb{R}.\label{eq:ZActSusp}
\end{equation}
Let $\Sigma_{f}:=(S^{1}\times\mathbb{R})/_{f}\mathbb{Z}$ be the quotient
space with canonical projection map $\pi_{f}:S^{1}\times\mathbb{R}\rightarrow\Sigma_{f}.$
It is clear that $\Sigma_{f}$ is diffeomorphic to the torus $\mathbb{T}^{2}$.
The image of the vertical leaves $\{x\}\times\mathbb{R}$ under $\pi_{f}$,
i.e.
\[
{\cal F}_{f}:=\big\{\pi_{f}\big(\{x\}\times\mathbb{R}\big):x\in S^{1}\big\},
\]
defines a 1D foliation on $\Sigma_f$.
\begin{defn}
\label{def:SuspensionMain}The foliation $(\Sigma_{f},{\cal F}_{f})$
is called the \textit{suspension} of the diffeomorphism $f$. 
\end{defn}
We also need the concept of conjugated foliations (see \cite[Definition 2.3]{BZ17} and \cite[Chapter 1, Section 2.2]{HH86}). 

\begin{defn}\label{def:Conj}

Let ${\cal F},{\cal F}'$ be $k$-dimensional foliations on $n$-dimensional
manifolds $M,M'$ respectively. We say that ${\cal F}$ and ${\cal F}'$
are \textit{conjugate} (or \textit{isomorphic}) if there exists a
diffeomorphism $\Psi:M\rightarrow M'$ such that $\Psi({\cal F})={\cal F}'$.

\end{defn}

We now state the main classification theorem for 1D foliations that will be used in our study. 
\begin{thm}
\label{thm:1DFoliClass}Let ${\cal F}$ be a 1D foliation on $\mathbb{T}^{2}$.
Then either ${\cal F}$ has a Reeb component or it is conjugate to
the suspension of an orientation-preserving diffeomorphism $f:S^{1}\rightarrow S^{1}$.
\end{thm}
Theorem \ref{thm:1DFoliClass} was stated and proved by Hector-Hirsch \cite[Proposition 4.3]{HH86} under the homeomorphism setting. The smooth version was stated by Bonatti-Zhang \cite[Theorem 2.6]{BZ17} but the actual theorem the authors referred to was the continuous version of Hector-Hirsch. For completeness, in the appendix we provide a slef-contained proof of Theorem \ref{thm:1DFoliClass} under the diffeomorphism setting.

\subsection{\label{subsec:PoinMap}Transverse characteristic number of the Poincar\'e
map}

In this subsection, we recall the definition of a hyperbolic periodic
orbit, which will play an essential role in our later analysis for
the Reeb component regime. To this end, we first need to recall the
concepts of a Poincar\'e map and its transverse characteristic number. 

Let $W$ be a smooth, everywhere nonvanishing vector field on $\mathbb{T}^{2}.$
Let 
\[
\big\{\phi_{t}(x):x\in\mathbb{T}^{2},t\in\mathbb{R}\big\}
\]
be its flow of diffeomorphisms. 
\begin{defn}
For each $x\in\mathbb{T}^{2},$ the set $\{\phi_{t}(x):t\in\mathbb{R}\}$
is called the \textit{orbit} of $W$ through $x$. An orbit $\gamma$
is said to be \textit{periodic} if there exist $T>0$ and $x\in\gamma$
(equivalently, for all $x\in\gamma$) such that $\phi_{T}(x)=x$.
The smallest positive number satisfying the above relation is called
the \textit{period} of $\gamma$. A periodic orbit is \textit{noncontractible}
if its homology class in $H_{1}(\mathbb{T}^{2};\mathbb{Z})$ is nonzero. 
\end{defn}
Note that a periodic orbit needs not always exist. However, if it
exists it must be noncontractible (see \cite[Chapter 1, Proposition 4.2.2]{HH86} or Proposition \ref{prop:POTop}).

We now define the (local) Poincar\'e map. Suppose that $\gamma$
is a periodic orbit of $W$ through a given point $x$. Let $\Sigma$
be a smooth, 1D curve that intersects $\gamma$ transversally at $x$
(by definition, it means that $T_{x}\Sigma\oplus T_{x}\gamma=\mathbb{R}^{2}$).
There is a well-defined map 
\[
P:U\subseteq\Sigma\rightarrow\Sigma
\]
on a sufficiently small neighbourhood $U\subseteq\Sigma$ of $x$,
which is constructed as follows. Let $y\in\Sigma$. Consider the flow
line $t\mapsto\phi_{t}(y)$. Let $\tau(y)>0$ be the first time that
$\phi_{t}(y)$ returns to $\Sigma$. We then define 
\[
P(y):=\phi_{\tau(y)}(y).
\]
Since $P(x)=x$, by continuity the map $P$ is well-defined in a small
neighbourhood $U$ of $x$ (see Lemma \ref{lem:PoinMap}). In addition, since $\Sigma$ is one-dimensional,
the linear map $(dP)_{x}:T_{x}\Sigma\rightarrow T_{x}\Sigma$ is a
scalar multiple of the identity, i.e. $(dP)_{x}=\lambda_{\gamma}{\rm id}$
for some $\lambda_{\gamma}.$
\begin{defn}
The map $P$ constructed above is known as the \textit{Poincar\'e
map} at $x\in\gamma$ with respect to the transversal $\Sigma$.
The number $\lambda_{\gamma}$ is called the \textit{transverse characteristic
number} of $\gamma.$
\end{defn}
Although the Poincar\'e map depends on the point $x\in\gamma$ and
the transversal $\Sigma$, it is important to note that $\lambda_{\gamma}$
depends only on the periodic orbit $\gamma$ (see \cite[Sections 3.4, 3.5]{Per13} or Proposition \ref{prop:CharaNum}).

\begin{defn}
\label{def:HypOrb}Let $\gamma$ be a periodic orbit of a smooth,
everywhere nonvanishing vector field $W$ on $\mathbb{T}^{2}$. We
say that $\gamma$ is \textit{hyperbolic} if $\lambda_{\gamma}\neq1$.
\end{defn}

\subsection{The suspension case}

We now proceed to develop the proof of Theorem \ref{thm:2DPeriodic}.
Let $V=(V_{1},V_{2})$ be uniformly elliptic, smooth vector fields
on $\mathbb{T}^{2}.$ Given fixed $\alpha_{0}\in[0,2\pi)$, we consider
the vector field 
\begin{equation}
W_{\alpha_{0}}(x):=V(x)\cdot\left(\begin{array}{c}
\cos\alpha_{0}\\
\sin\alpha_{0}
\end{array}\right),\quad x\in\mathbb{T}^{2}.\label{eq:WAlp}
\end{equation}
By ellipticity, $W_{\alpha_{0}}$ is everywhere nonvanishing and thus
induces a 1D foliation ${\cal F}_{\alpha_{0}}$ on $\mathbb{T}^{2}.$
According to Theorem \ref{thm:1DFoliClass}, either ${\cal F}_{\alpha_{0}}$
\textit{contains a Reeb component} or \textit{it is conjugate to a
suspension}. In this subsection, we deal with the suspension case.
The dynamical mechanism for this case turns out to be rather simple;
as we will see, one is led precisely to the regime of Section \ref{sec:LowGauCondYoung}.
The Reeb component case is much more involved and will be handled
in the next two subsections.

\begin{proof}[Proof of Theorem \ref{thm:2DPeriodic}: Suspension case]We
therefore assume that there exists a diffeomorphism $\Psi:\mathbb{T}^{2}\rightarrow\Sigma_{f}$,
where $\Sigma_{f}$ is the suspension of an orientation-preserving
diffeomorphism $f:S^{1}\rightarrow S^{1},$ such that $\Psi({\cal F}_{\alpha_{0}})={\cal F}_{f}$
(see Definition \ref{def:Suspension} and Definition \ref{def:Conj}). 

\vspace{2mm}\noindent (i) Consider the projection 
\[
\rho:\Sigma_{f}\rightarrow S^{1}=\mathbb{R}/\mathbb{Z},\quad\rho([x,t])\mapsto t\ \ ({\rm mod}\ \mathbb{Z}).
\]
Let $d\theta$ denote the canonical volume form on $S^{1}$ such that
$\int_{S^{1}}d\theta=1$. Define the closed $1$-form $\beta:=\rho^{*}(d\theta)$
on $\Sigma_{f}.$ Let $Y$ be the suspension vector field on $\Sigma_{f}$;
the vertical vector field $\partial_{t}$ on $S^{1}\times\mathbb{R}$
descends to $Y$ on $\Sigma_{f}$ under the quotient of the $\mathbb{Z}$-action
(\ref{eq:ZActSusp}), which is tangential to the foliation ${\cal F}_{f}$.
By definition, one has 
\[
\langle\beta(\xi),Y(\xi)\rangle=\langle d\theta,(d\rho)(Y)\rangle_{\rho(\xi)}=\langle d\theta,\partial_{t}\rangle_{\rho(\xi)}=1
\]
for every $\xi\in\Sigma_{f}$.

\vspace{2mm}\noindent (ii) Since $\Psi({\cal F}_{\alpha_{0}})={\cal F}_{f}$,
there exists a nonvanishing function $\lambda\in C^{\infty}(\mathbb{T}^{2})$
such that
\[
(d\Psi)_{x}(W_{\alpha_{0}}(x))=\lambda(x)Y(\Psi(x))
\]
for every $x\in\mathbb{T}^{2}.$ By connectedness, one may assume without loss of generality that 
\[
m:=\inf_{x\in\mathbb{T}^{2}}\lambda(x)>0.
\]
Let $\eta:=\Psi^{*}(\beta).$ It follows that $\eta$ is closed and
one has 
\[
\langle\eta(x),W_{\alpha_{0}}(x)\rangle=\lambda(x)\langle\beta(\Psi(x)),Y(\Psi(x))\rangle=\lambda(x)
\]
for every $x\in\mathbb{T}^{2}.$ In particular, 
\begin{equation}
\inf_{x\in\mathbb{T}^{2}}\langle\eta(x),W_{\alpha_{0}}(x)\rangle=m>0.\label{eq:NonDegEta}
\end{equation}
(iii) We claim that the (de Rham) cohomology class $[\eta]$ is nontrivial.
Indeed, if $\eta$ were exact, say $\eta=dq$ for some $q\in C^{\infty}(\mathbb{T}^{2})$,
one would have $\eta(x^{*})=0$ at the location $x^{*}$ where $q$
is maximised. This contradicts the fact (\ref{eq:NonDegEta}). Therefore,
$[\eta]\neq0.$

\vspace{2mm}\noindent (iv) Let $[\alpha_{1}],[\alpha_{2}]$ denote
a basis of $H_{1}(\Sigma_{f};\mathbb{Z})$ such that 
\begin{equation}
\int_{[\alpha_{1}]}[\beta]=0,\quad\int_{[\alpha_{2}]}[\beta]=1.\label{eq:BetaBasis}
\end{equation}
One can construct them in the following way. The loop $\alpha_{1}=([x,t])_{x\in S^{1}}$
is a circle positively winding around the horizontal $S^{1}$-component
once. The loop $\alpha_{2}$ is defined by a vertical curve $\alpha_{2}':=[(x,t)]_{0\leqslant t\leqslant1}$
concatenated with a horizontal arc $\alpha_{2}''$ joining $[(f(x),0)]$
back to $[(x,0)]$ along the horizontal $S^{1}$-component. This concatenation
is needed since $\alpha_{2}'$ may fail to be a closed curve (leaves
of ${\cal F}_{f}$ needs not be compact). The relation (\ref{eq:BetaBasis})
is now obvious from the construction. 

Let $\{[l_{1}],[l_{2}]\}$ be the canonical basis of $H_{1}(\mathbb{T}^{2};\mathbb{Z})$
that is dual to $\{[dx^{1}],[dx^{2}]\}$. Let $\zeta_{1},\zeta_{2}$
be the preimages of $\alpha_{1},\alpha_{2}$ under the diffeomorphism
$\Psi$. Then $\{[\zeta_{1}],[\zeta_{2}]\}$ is a basis of $H_{1}(\mathbb{T}^{2};\mathbb{Z})$
and one has 
\[
\int_{[\zeta_{1}]}[\eta]=0,\quad\int_{[\zeta_{2}]}[\eta]=1.
\]
In particular, one can write 
\[
[\zeta_{1}]=a[l_{1}]+b[l_{2}],\quad[\zeta_{2}]=c[l_{1}]+d[l_{2}]
\]
with some matrix 
\[
A=\left(\begin{array}{cc}
a & b\\
c & d
\end{array}\right)\in{\rm GL}_{2}(\mathbb{Z}).
\]
It follows that 
\[
[\eta]=k_{1}[dx^{1}]+k_{2}[dx^{2}],
\]
where 
\[
\left(\begin{array}{c}
k_{1}\\
k_{2}
\end{array}\right):=A^{-1}\cdot\left(\begin{array}{c}
0\\
1
\end{array}\right)\in\mathbb{Z}^{2}\backslash\{(0,0)\}.
\]
As a result, one has
\[
\eta=k_{1}dx^{1}+k_{2}dx^{2}+d\varphi
\]
for some $\varphi\in C^{\infty}(\mathbb{T}^{2}).$

\vspace{2mm}\noindent (v) Consider the function 
\[
F(x^{1},x^{2}):=k_{1}x^{1}+k_{2}x^{2}+\tilde{\varphi}(x^{1},x^{2}),\quad x=(x^{1},x^{2})\in\mathbb{R}^{2},
\]
where $\tilde{\varphi}$ is the pullback of $\varphi$ (i.e. its periodic
extension) on $\mathbb{R}^{2}$. It is clear that $F$ has bounded
first and second derivatives on $\mathbb{R}^{2}$. In addition, let
$\eta:=(\cos\alpha_{0},\sin\alpha_{0})^{T}$ and let $\tilde{V}$
be the periodic extension of $V$ on $\mathbb{R}^{2}$ (same for $\tilde{W}_{\alpha_{0}}$).
Then one has 
\[
((\tilde{V}\cdot\eta)F)(x)=\langle dF(x),\tilde{W}_{\alpha_{0}}(x)\rangle=\langle\eta([x]),W_{\alpha_{0}}([x])\rangle\geqslant m>0
\]
for all $x\in\mathbb{R}^{2}$ ($[x]\in\mathbb{T}^{2}$ denotes the
projected image of $x$). As a consequence, Condition (H) in Section
\ref{sec:LowGauCondYoung} is satisfied with the above choices of
$\eta$ and $F$. The Gaussian lower tail estimate (\ref{eq:GaussLowTorus})
now follows from Theorem {\ref{thm:GLowH}}.

\end{proof}

\subsection{The Reeb component case}

We now investigate the Reeb component case. Suppose that the foliation
${\cal F}_{\alpha_{0}}$ of $W_{\alpha_{0}}$ (see (\ref{eq:WAlp}))
contains a Reeb component. By definition, there is an embedded annulus
$A\cong S^{1}\times[0,1]$ inside $\mathbb{T}^{2}$ with two compact
leaves from $\partial A$ (periodic orbits) and all leaves in $\mathrm{int}A$
are diffeomorphic to $\mathbb{R}$. The key insight here is that one
can perturb $\alpha_0$ into a new selection $W_{\alpha_*}$, whose flow indeed contains a \textit{hyperbolic} periodic orbit. The main lemma is
stated as follows.
\begin{lem}
\label{lem:HyperSelect}Let $\alpha_{0}\in[0,2\pi)$ be such that
the foliation of $W_{\alpha_{0}}$ contains a Reeb component. Then
there exists $\alpha_{*}\in[0,2\pi)$ such that $W_{\alpha_{*}}$
admits a hyperbolic periodic orbit.
\end{lem}
We will prove Lemma \ref{lem:HyperSelect} in Section \ref{subsec:HyperSelectPf}.
In what follows, we complete the proof of Theorem \ref{thm:2DPeriodic}
assuming the correctness of this lemma. We first state a simple continuity
estimate for Young's differential equations. 
\begin{lem}
\label{lem:YoungCty}Let $p\in(1/H,2)$. Consider the differential equation 
\[
\begin{cases}
dY_{s}=\tilde{W}(Y_{s})ds+\tilde{V}(Y_{s})dv_{s}, & 0\leqslant s\leqslant T,\\
Y_{0}=y\in\mathbb{R}^{2},
\end{cases}
\]
where $\tilde{W},\tilde{V}$ are the periodic extension of $W,V\in C^{\infty}(\mathbb{T}^{2})$
respectively and $v\in C^{p\text{-var}}([0,T];\mathbb{R}^{2})$. For
any $M>0,$ there exists $C>0$ depending only on $V,W,p,T$ and $M,$
such that 
\[
\|v\|_{p\text{-var}}\leqslant M\implies\sup_{0\leqslant s\leqslant T}\big|Y_{s}-\tilde{\phi}_{s}(y)\big|\leqslant C\|v\|_{p\text{-var}}.
\]
Here $\tilde{\phi}_{s}(y)$ denotes the flow of $\tilde{W}$ on $\mathbb{R}^{2}$
starting from $y$. 
\end{lem}
\begin{proof}
This is immediate from standard continuity estimates for Young's differential
equations (see e.g. \cite[Theorem 10.38]{FV10}).
\end{proof}
Now let $\gamma$ be a hyperbolic periodic orbit of $W:=W_{\alpha_{*}}$
from Lemma \ref{lem:HyperSelect}. Let $T$ be the period of $\gamma.$
Recall from Proposition \ref{prop:CharaNum} that the transversal
characteristic number $\lambda_{\gamma}$ is strictly positive. Due
to hyperbolicity, one has either $\lambda_{\gamma}\in(0,1)$ or $\lambda_{\gamma}>1$.
By changing $W$ to $-W$ if necessary ($\lambda_{\gamma}$ will be
changed to its reciprocal), one may assume without loss of generality
that $\lambda_{\gamma}\in(0,1)$. 

For sufficiently small $\varepsilon,$ let $U_{\varepsilon}$ be a
tubular neighbourhood of $\gamma$ in $\mathbb{T}^{2}$. More precisely,
there is a diffeomorphism 
\[
\Phi:\mathbb{R}/T\mathbb{Z}\times(-\varepsilon,\varepsilon)\rightarrow U_{\varepsilon},\quad\Phi(t,r):=\gamma_{t}+r{\bf n}(t),
\]
where ${\bf n}$ is a fixed unit normal vector field along $\gamma.$
For $x=\Phi(t,r)\in U_{\varepsilon},$ we denote $z(x):=r$ as the
transversal coordinate of $x$ in $U_{\varepsilon}.$ Let $\tilde{\gamma}$
be a given lifting of $\gamma$ on $\mathbb{R}^{2}.$ The tubular
neighbourhood map $\Phi$ also lifts to a tubular neighbourhood map
$\tilde{\Phi}:\mathbb{R}\times(-\varepsilon,\varepsilon)\rightarrow\tilde{U}_{\varepsilon}$
of $\tilde{\gamma}$, whose transversal coordinate map is denoted
as $\tilde{z}:\tilde{U}_{\varepsilon}\rightarrow(-\varepsilon,\varepsilon)$.

The following lemma is a consequence of the hyperbolicity of $\gamma,$
which is crucial to us. 
\begin{lem}
\label{lem:OrbitStab}Let $\lambda^{*}\in(\lambda_{\gamma},1)$ be
a given fixed number. Then there exists $\delta<\varepsilon$, such
that 
\begin{equation}
\phi_{s}(U_{\delta})\subseteq U_{\varepsilon/4}\label{eq:Tub1}
\end{equation}
for all $s\in[0,T]$ and 
\begin{equation}
|z(\phi_{T}(x))|\leqslant\lambda^{*}|z(x)|\label{eq:Tub2}
\end{equation}
for all $x\in U_{\delta}$. Here $\phi_{s}$ denotes the flow of $W$.
\end{lem}
\begin{proof}

The property (\ref{eq:Tub1}) is obvious by continuity. To prove the
second property, let 
\[
G(t,r):=z\big(\phi_{T}(\Phi(t,r))\big),\quad(t,r)\times\mathbb{R}/T\mathbb{Z}\times(-\varepsilon,\varepsilon).
\]
Note that $\partial_{r}G(t,0)=\lambda_{\gamma}.$ By continuity, one
has 
\[
\sup_{\substack{t\in\mathbb{R}/T\mathbb{Z}\\
r\in(-\delta,\delta)
}
}\big|\partial_{r}G(t,r)\big|\leqslant\lambda^{*},
\]
provided that $\delta$ is small enough. Since $G(t,0)=0,$ it follows
that 
\[
|z(\phi_{T}(x))|=\big|G(t,r)\big|\leqslant\int_{0}^{r}\big|\partial_{r}G(t,\rho)\big|d\rho\leqslant\lambda^{*}|r|=\lambda^{*}|z(x)|
\]
for any $x=\Phi(t,r)\in U_{\delta}$. This proves the property (\ref{eq:Tub2}). 

\end{proof}

\begin{rem}
\label{rem:LiftedTube}Due to periodicity, it is clear that both properties
(\ref{eq:Tub1}) and (\ref{eq:Tub2}) remain valid for the lifted
objects on $\mathbb{R}^{2}$.
\end{rem}
We are now in a position to prove Theorem \ref{thm:2DPeriodic} in
the Reeb component case. 

\begin{proof}[Proof of Theorem \ref{thm:2DPeriodic}: Reeb component case]
We divide the argument into several major steps. Let $\lambda^{*},U_{\delta}$
be given as in Lemma \ref{lem:OrbitStab}. Let $R>2T$. 

\vspace{2mm}\noindent (i) Consider the differential equation 
\begin{equation}
\begin{cases}
dX_{t}^{R,w}=R\tilde{W}(X_{t}^{R,w})dt+\tilde{V}(X_{t})dw_{t}, & 1/2\leqslant t\leqslant1,\\
X_{1/2}^{R,w}=x\in\tilde{U}_{\delta},
\end{cases}\label{eq:YoungDE}
\end{equation}
where $w\in C^{p\text{-var}}([0,1];\mathbb{R}^{2})$. We divide $[1/2,1]$
into equal subintervals 
\[
[t_{j-1},t_{j}]:=[1/2+(j-1)T/R,1/2+jT/R],\quad j=1,\cdots,\lfloor\frac{R}{2T}\rfloor:=N_{R}
\]
as well as the last subinterval $[t_{N_{R}},1].$ We claim that there
exists $M_{1}>0$ depending on $\delta,$ such that 
\begin{equation}
\|w\|_{p\text{-var};[1/2,1]}\leqslant M_{1}\implies X_{t_{j}}^{R,w}\in\tilde{U}_{\delta}\label{eq:TunStab}
\end{equation}
for all $j=1,\cdots,N_{R}.$

To see this, suppose that $X_{t_{j-1}}^{R,w}\in\tilde{U}_{\delta}$.
One can reparametrise the equation (\ref{eq:YoungDE}) on $[t_{j-1},t_{j}]$
as 
\[
\begin{cases}
dY_{s}^{j}=\tilde{W}(Y_{s}^{j})ds+\tilde{V}(Y_{s}^{j})dv_{s}^{j}, & 0\leqslant s\leqslant T,\\
Y_{0}^{j}=X_{t_{j-1}}^{R,w},
\end{cases}
\]
where $v_{s}^{j}:=w_{t_{j-1}+s/R}-w_{t_{j-1}}.$ Note that 
\[
\|v^{j}\|_{p\text{-var};[0,T]}=\|w\|_{p\text{-var};[t_{j-1},t_{j}]}\leqslant\|w\|_{p\text{-var};[1/2,1]}.
\]
By the relation (\ref{eq:Tub1}) (cf. Remark \ref{rem:LiftedTube}),
one has 
\[
\big\{\tilde{\phi}_{s}(X_{t_{j-1}}^{R,w}):0\leqslant s\leqslant T\big\}\subseteq\tilde{U}_{\varepsilon/4}.
\]
In addition, according to Lemma \ref{lem:YoungCty}, one has 
\[
\|w\|_{p\text{-var};[1/2,1]}\leqslant1\implies\sup_{0\leqslant s\leqslant T}\big|Y_{s}^{j}-\tilde{\phi}_{s}(X_{t_{j-1}}^{R,w})\big|\leqslant C_{1}\|w\|_{p\text{-var};[1/2,1]}.
\]
In particular, $X_{t}^{R,w}\in\tilde{U}_{\varepsilon/2}$ for all
$t\in[t_{j-1},t_{j}]$, provided that $\|w\|_{p\text{-var};[1/2,1]}\leqslant M_{1}$
and $M_{1}<\frac{\varepsilon}{4C_{1}}\wedge1.$ 

Note that the transversal coordinate map $z:U_{\varepsilon}\rightarrow(-\varepsilon,\varepsilon)$
is Lipschitz on $\bar{U}_{\varepsilon/2}$. Let $L$ denote its Lipschitz
constant (same constant for $\tilde{z}$). The previous continuity
estimate also yields that 
\[
\big|\tilde{z}(X_{t_{j}}^{R,w})-\tilde{z}(\tilde{\phi}_{T}(X_{t_{j-1}}^{R,w}))\big|\leqslant C_{1}L\|w\|_{p\text{-var};[1/2,1]},
\]
provided that $\|w\|_{p\text{-var};[1/2,1]}\leqslant M_{1}.$ It follows
from the contraction property (\ref{eq:Tub2}) that 
\begin{align*}
\big|\tilde{z}(X_{t_{j}}^{R,w})\big| & \leqslant\big|\tilde{z}(\tilde{\phi}_{T}(X_{t_{j-1}}^{R,w}))\big|+\big|\tilde{z}(X_{t_{j}}^{R,w})-\tilde{z}(\tilde{\phi}_{T}(X_{t_{j-1}}^{R,w}))\big|\\
 & \leqslant\lambda^{*}\delta+C_{1}L\|w\|_{p\text{-var};[1/2,1]}\leqslant\lambda^{*}\delta+C_{1}LM_{1}.
\end{align*}
By further requiring 
\[
M_{1}<\frac{(1-\lambda^{*})\delta}{C_{1}L},
\]
one obtains that 
\[
\big|\tilde{z}(X_{t_{j}}^{R,w})\big|<\delta\iff X_{t_{j}}^{R,w}\in\tilde{U}_{\delta}.
\]
Since $X_{1/2}^{R,w}=x\in\tilde{U}_{\delta}$ by assumption, this
proves the claim (\ref{eq:TunStab}) by induction on $j$. 

\vspace{2mm}\noindent (ii) Since $\gamma$ has nontrivial homology,
there exists $k\in\mathbb{Z}^{2}\backslash\{(0,0)\}$ such that 
\[
\tilde{\gamma}(s+T)=\tilde{\gamma}(s)+k
\]
for all $s\in\mathbb{R}.$ Consider the function 
\[
F(x):=\frac{k\cdot x}{|k|^{2}},\quad x\in\mathbb{R}^{2}.
\]
Note that $F(\tilde{\phi}_{T}(x)-x)=1$ for all $x\in\tilde{\gamma}$.
By shrinking $\delta$ if necessary, one can ensure that 
\begin{equation}
F\big(\tilde{\phi}_{T}(x)-x\big)\geqslant\frac{3}{4}\label{eq:LinDispPerOrb}
\end{equation}
for all $x\in\tilde{U}_{\delta}.$

Let $X_{t}^{R,w}$ be the solution to the equation (\ref{eq:YoungDE}).
We claim that there exists $M_{1}>0$, such that 
\begin{equation}
\|w\|_{p\text{-var};[1/2,1]}\leqslant M_{1}\implies F\big(X_{t_{j}}^{R,w}-X_{t_{j-1}}^{R}\big)\geqslant\frac{1}{2}\label{eq:LinDispTorus}
\end{equation}
for all $i=1,\cdots,N_{R}.$ Indeed, with $M_{1}$ chosen as in Part
(i) one has 
\[
\|w\|_{p\text{-var};[1/2,1]}\leqslant M_{1}\implies\big|F\big(X_{t_{j}}^{R,w}-\tilde{\phi}_{T}(X_{t_{j-1}}^{R,w})\big)\big|\leqslant\frac{C_{1}}{|k|}\|w\|_{p\text{-var};[1/2,1]}.
\]
Since $X_{t_{j-1}}^{R,w}\in\tilde{U}_{\delta},$ it follows from (\ref{eq:LinDispPerOrb})
that 
\begin{align*}
F\big(X_{t_{j}}^{R,w}-X_{t_{j-1}}^{R,w}\big) & \geqslant F\big(\tilde{\phi}_{T}(X_{t_{j-1}}^{R,w})-X_{t_{j-1}}^{R,w}\big)-\frac{C_{1}}{|k|}\|w\|_{p\text{-var};[1/2,1]}\geqslant\frac{3}{4}-\frac{C_{1}M_{1}}{|k|}.
\end{align*}
By further requiring 
\[
M_{1}<\frac{|k|}{4C_{1}},
\]
one concludes that (\ref{eq:LinDispTorus}) holds. As a consequence,
\[
F\big(X_{t_{N_{R}}}^{R,w}-X_{1/2}^{R,w}\big)\geqslant\frac{1}{2}N_{R}=\frac{1}{2}\lfloor\frac{R}{2T}\rfloor
\]
provided that $X_{1/2}^{R,w}\in\tilde{U}_{\delta}$ and $\|w\|_{p\text{-var};[1/2,1]}\leqslant M_{1}.$ 

Under the same assumptions as above, the last piece $[t_{N_{R}},1]$
is easily estimated as 
\begin{align*}
\big|F\big(X_{1}^{R,w}-X_{t_{N_{R}}}^{R,w}\big)\big|\leqslant & C_{2}\sup_{0\leqslant s\leqslant T}\Big[\big|Y_{s}^{N_{R}+1}-\tilde{\phi}_{s}(X_{t_{N_{R}}}^{R,w})\big|\\
 & \ \ \ +\big|\tilde{\phi}_{s}(X_{t_{N_{R}}}^{R,w})-X_{t_{N_{R}}}^{R,w}\big|\Big]\leqslant C_{3}.
\end{align*}
By using the Lipschitz property of $F$ again, one therefore obtains
that 
\begin{equation}
\big|X_{1}^{R,w}-X_{1/2}^{R,w}\big|\geqslant\frac{|k|}{2}\lfloor\frac{R}{2T}\rfloor-C_{3}|k|,\label{eq:LinearTorus}
\end{equation}
provided that $X_{1/2}^{R,w}\in\tilde{U}_{\delta}$ and $\|w\|_{p\text{-var};[1/2,1]}\leqslant M_{1}.$
The estimate (\ref{eq:LinearTorus}) is uniform for all large $R$
(with $M$ being chosen as before). 

\vspace{2mm}\noindent (iii) Recall that $x_{0}\in\mathbb{R}^{2}$
is the given initial condition of the SDE (\ref{eq:SDETorus}). Let
$x$ be any fixed point on $\tilde{\gamma}$ and let $\beta:[0,1/2]\rightarrow\mathbb{R}^{2}$
be any fixed smooth path (e.g. the straight line) such that $\beta_{0}=x_{0},$
$\beta_{1/2}=x$. Define 
\[
h_{\text{int}}(t):=\int_{0}^{t}V^{-1}(\beta_{s})\dot{\beta_{s}ds,\quad}0\leqslant t\leqslant\frac{1}{2}.
\]
By definition, the control $h_{{\rm int}}$ drives the dynamics 
\[
dy_{t}=V(y_{t})\dot{h}_{{\rm int}}(t)dt
\]
from $x_{0}$ to $x$ along $\beta$. By the same continuity estimate
leading to Lemma \ref{lem:YoungCty}, one has 
\[
||w\|_{p\text{-var};[0,1/2]}\leqslant M_{0}\implies X_{1/2}^{w+h_{{\rm int}}}\in\tilde{U}_{\delta}
\]
for some small $M_{0}$. 

\vspace{2mm}\noindent (iv) Let us summarise what we have obtained
so far. For each $R>2T,$ define 
\[
h^{R}(t)\triangleq\hat{h}_{{\rm int}}(t)+R\cdot\hat{h}_{{\rm main}}(t),\quad0\leqslant t\leqslant1,
\]
where 
\[
\hat{h}_{{\rm int}}(t):=\begin{cases}
h_{{\rm int}}(t), & 0\leqslant t\leqslant1/2,\\
h_{{\rm int}}(1/2), & 1/2\leqslant t\leqslant1,
\end{cases}\quad\hat{h}_{{\rm main}}(t):=\begin{cases}
0, & 0\leqslant t\leqslant1/2,\\
(t-1/2)\eta_{*}, & 1/2\leqslant t\leqslant1
\end{cases}
\]
and $\eta_{*}:=(\cos\alpha_{*},\sin\alpha_{*})^{T}.$ Let $M:=M_{0}\wedge M_{1}.$
What we have shown is that whenever $w\in C^{p\text{-var}}([0,1];\mathbb{R}^{2})$
satisfies $\|w\|_{p\text{-var};[0,1]}\leqslant M$, one has 
\[
\big|X_{1}^{w+h^{R}}-x_{0}\big|\geqslant\frac{|k|}{2}\lfloor\frac{R}{2T}\rfloor-C_{3}|k|-|x-x_{0}|
\]
for all large $R$. On the other hand, it is easy to see that $\hat{h}_{{\rm int}},\hat{h}_{{\rm main}}\in\bar{{\cal H}}$
and 
\[
\|h^{R}\|_{\bar{{\cal H}}}\leqslant\|\hat{h}_{{\rm int}}\|_{\bar{{\cal H}}}+R\|\hat{h}_{{\rm main}}\|_{\bar{{\cal H}}}.
\]
These are essentially the properties (\ref{eq:DispEstYoung}) and
(\ref{eq:hRNorm}). Therefore, one can carry through exactly the same
Cameron-Martin transformation argument from the proof of Theorem \ref{thm:GLowH}
to conclude that the solution to the SDE (\ref{eq:SDETorus}) at time
one admits a Gaussian lower tail. 

The proof of Theorem \ref{thm:2DPeriodic} is now complete. 

\end{proof}

\subsection{\label{subsec:HyperSelectPf}Proof of the hyperbolic orbit selection
lemma}

In this subsection, we develop the proof of Lemma \ref{lem:HyperSelect}. 

\subsubsection{Parametric Poincar\'e map}

We first set up some basic definitions. For $\alpha\in[0,2\pi)$,
we define 
\[
W_{\alpha}(x):=V(x)\cdot\left(\begin{array}{c}
\cos\alpha\\
\sin\alpha
\end{array}\right),\quad x\in\mathbb{T}^{2}.
\]
Suppose that $W_{\alpha_{0}}$ contains a Reeb component $A$ for
some given $\alpha_{0}\in[0,2\pi)$. Let $\gamma$ be one of the two
periodic orbits from $\partial A$ with period $T>0$ and fix $x\in\gamma$.
Let $\Sigma$ be a smooth curve intersecting $\gamma$ transversally
at $x$. Since our argument is local, we assume that $\Sigma$ is
parametrised by $\sigma:(-\varepsilon,\varepsilon)\rightarrow\Sigma$
with $\sigma(0)=x$. We also assume that the orientation of $\Sigma$
is chosen such that $\sigma((0,\varepsilon))$ is contained in the
interior of the Reeb component $A$. 

Recall that the Poincar\'e map is defined in Section \ref{subsec:PoinMap}.
For our purpose, we need a parametric version of the Poincar\'e map
with respect to the parameter $\alpha$ near $\alpha_{0}.$ More specifically,
we define two functions
\[
P:(-\delta,\delta)\times(\alpha_{0}-\theta,\alpha_{0}+\theta)\rightarrow(-\varepsilon,\varepsilon),
\]
\[
\tau:(-\delta,\delta)\times(\alpha_{0}-\theta,\alpha_{0}+\theta)\rightarrow(0,\infty)
\]
with small $\delta,\theta$ in the following way. Given $u\in(-\delta,\delta)$
and $\alpha\in(\alpha_{0}-\theta,\alpha_{0}+\theta)$, the number
$\tau(u,\alpha)$ is the first return time to $\Sigma$ of the flow
line $t\mapsto\phi_{t}^{\alpha}(\sigma(u))$, where $\phi_{t}^{\alpha}(y)$
denotes the flow of $W_{\alpha}$ starting from $y$. Correspondingly,
$P(u,\alpha)$ is the location of first return, i.e. 
\begin{equation}
\sigma\big(P(u,\alpha)\big)=\phi_{\tau(u,\alpha)}^{\alpha}(\sigma(u)).\label{eq:ParaPoinMap}
\end{equation}
Since 
\[
P(0,\alpha_{0})=0,\quad\tau(0,\alpha_{0})=T,
\]
by continuity the maps $P,\tau$ are both well-defined for sufficiently
small $\delta<\varepsilon$ and $\theta.$ In fact, $P$ and $\tau$
are smooth in both variables. We define the displacement $D(u,\alpha):=P(u,\alpha)-u$. 

A key observation is the following lemma, which is a consequence of
the special structure of the family $\alpha\mapsto W_{\alpha}$. 
\begin{lem}
\label{lem:NonDegD}One has $\partial_{\alpha}D(0,\alpha_{0})\neq0$.
\end{lem}
\begin{proof}
Recall that $\gamma:[0,T]\rightarrow\mathbb{T}^{2}$ ($\gamma_{0}=x$)
is the chosen periodic orbit of $W_{\alpha_{0}}$ from a Reeb component.
Consider the function 
\[
d(t):=\det\big(W_{\alpha_{0}}(\gamma_{t}),\xi(t)\big)
\]
where $\xi(t):=\partial_{\alpha}\phi_{t}^{\alpha}(x)\big|_{\alpha=\alpha_{0}}$.
Let $Z(y):=\partial_{\alpha}W_{\alpha}(y)|_{\alpha=\alpha_{0}}$ for
$y\in\mathbb{T}^{2}.$ By using the flow equation of $\phi_{t}^{\alpha}$,
it is easily seen that 
\[
\dot{\xi}(t)=DW_{\alpha_{0}}(\gamma_{t})\xi(t)+Z(\gamma_{t}),\quad\xi(0)=0.
\]
Therefore, one has 
\begin{align}
\dot{d}(t)= & \det\big((DW_{\alpha_{0}}\cdot W_{\alpha_{0}})(\gamma_{t}),\xi(t)\big)+\det\big(W_{\alpha_{0}}(\gamma_{t}),DW_{\alpha_{0}}(\gamma_{t})\xi(t)\big)\nonumber \\
 & \ \ \ +\det\big(W_{\alpha_{0}}(\gamma_{t}),Z(\gamma_{t})\big)\nonumber \\
= & \big({\rm div}W_{\alpha_{0}}\big)(\gamma_{t})\cdot d(t)+\det\big(W_{\alpha_{0}}(\gamma_{t}),Z(\gamma_{t})\big).\label{eq:dtODE}
\end{align}
To reach the last equality, we used the simple relation that 
\[
\det(A\cdot{\bf u},{\bf v})+\det({\bf u},A\cdot{\bf v})={\rm tr}[A]\det({\bf u},{\bf v})
\]
for any $2\times2$ matrix $A$ and $2\times1$ row vectors ${\bf u},{\bf v}$.
Since $d(0)=0,$ it follows from (\ref{eq:dtODE}) that 
\[
d(T)=\int_{0}^{T}\exp\Big[\int_{s}^{T}({\rm div}W_{\alpha_{0}})(\gamma(r))dr\Big]\det\big(W_{\alpha_{0}}(\gamma_{s}),Z(\gamma_{s})\big)ds.
\]
Explicit computation shows that 
\[
\det\big(W_{\alpha_{0}}(y),Z(y)\big)=\det V(y)
\]
for all $y\in\mathbb{T}^{2}$. Since $V$ is uniformly elliptic, the above determinant has a constant sign. As a 
consequence, $d(T)\neq 0$.

On the other hand, by differentiating the relation (\ref{eq:ParaPoinMap})
with respect to $\alpha$ at $(u,\alpha)=(0,\alpha_{0})$, one finds
that 
\[
\xi(T)+\partial_{\alpha}\tau(0,\alpha_{0})W_{\alpha_{0}}(x)=\partial_{\alpha}P(0,\alpha_{0})\cdot\sigma'(0).
\]
By taking determinant with $W_{\alpha_{0}}(\gamma_{T})=W_{\alpha_{0}}(x),$
it follows that
\[
d(T)=\partial_{\alpha}P(0,\alpha_{0})\cdot\det\big(W_{\alpha_{0}}(x),\sigma'(0)\big).
\]
Since $\Sigma$ intersects $\gamma$ transversally at $x$, one knows
that 
\[
\det\big(W_{\alpha_{0}}(x),\sigma'(0)\big)\neq0.
\]
Therefore, 
\[
\partial_{\alpha}D(0,\alpha_{0})=\partial_{\alpha}P(0,\alpha_{0})\neq0.
\]
This completes the proof of the lemma. 
\end{proof}

\subsubsection{Hyperbolic orbit selection}

Recall that $D(0,\alpha_{0})=0$ and $\partial_{\alpha}D(0,\alpha_{0})\neq0$
by Lemma \ref{lem:NonDegD}. According to the implicit function theorem,
there exists $\delta_{1}<\delta$ and a smooth map 
\[
A:(-\delta_{1},\delta_{1})\rightarrow(\alpha_{0}-\theta,\alpha_{0}+\theta),
\]
such that 
\begin{equation}
D(u,A(u))=0\label{eq:IFT}
\end{equation}
for all $u\in(-\delta_{1},\delta_{1})$. In other words, the orbit
of $W_{A(u)}$ passing through $\sigma(u)$ is periodic for every
$u\in(-\delta_{1},\delta_{1})$. The following lemma shows that these
orbits must contain a hyperbolic one. 
\begin{lem}
There exists a $u_{*}\in(0,\delta_{1})$, such that the orbit of $W_{A(u_{*})}$
passing through $\sigma(u_*)$ is hyperbolic. 
\end{lem}
\begin{proof}
By differentiating (\ref{eq:IFT}), one finds that 
\[
\partial_{u}D(u,A(u))+\partial_{\alpha}D(u,A(u))A'(u)=0
\]
for all $u\in(-\delta_{1},\delta_{1})$. For $u\in(0,\delta_{1})$,
recall that the orbit of $W_{\alpha_{0}}$ through $\sigma(u)$ is
nonperiodic by assumption (because it is a leaf inside the interior of the Reeb component
$A$). In particular, one has $D(u,\alpha_{0})\neq0$ and
therefore, $A(u)\neq\alpha_{0}$ for any $u\in(0,\delta_{1})$. Since
$A(0)=\alpha_{0}$, it follows that for any $\eta>0,$ there exists
$u_{\eta}\in(0,\eta)$ such that $A'(u_{\eta})\neq0$. On the other
hand, one knows from Lemma \ref{lem:NonDegD} and continuity that
$\partial_{\alpha}D(u,A(u))\neq0$ for all sufficiently small $u$.
As a consequence, there must exist $u_{*}\in(0,\delta_{1})$, such
that 
\[
\partial_{u}D(u_{*},A(u_{*}))\neq0,
\]
or equivalently, $\partial_{u}P(u_{*},\alpha_{*})\neq1$ where $\alpha_{*}:=A(u_{*})$.
By definition, one concludes that the periodic orbit of $W_{\alpha_{*}}$
through $\sigma(u_{*})$ is hyperbolic.
\end{proof}
The proof of Lemma \ref{lem:HyperSelect} is now complete.
\begin{rem}
Although the suspension case is relatively simple, one cannot expect
avoiding the Reeb component case by hoping that ``there always exists
$\alpha\in[0,2\pi)$ such that $W_{\alpha}$ does not contain a Reeb
component''. For instance, consider the elliptic system
\[
V(x)=\left(\begin{array}{cc}
\cos2\pi x^{1} & -\sin2\pi x^{1}\\
\sin2\pi x^{1} & \cos2\pi x^{1}
\end{array}\right),\quad x=(x^{1},x^{2})\in\mathbb{T}^{2}.
\]
One has 
\[
W_{\alpha}(x)=\left(\begin{array}{c}
\cos(2\pi x^{1}+\alpha)\\
\sin(2\pi x^{1}+\alpha)
\end{array}\right)
\]for every $\alpha\in[0,2\pi)$. 
In particular, the foliation of $W_{\alpha}$ contains precisely two
Reeb components bounded by the compact leaves (periodic orbits) 
\[
L_{+}:x^{1}=\frac{\pi/2-\alpha}{2\pi},\quad L_{-}:x^{1}=\frac{3\pi/2-\alpha}{2\pi}
\]
respectively, regardless of the value of $\alpha$.
\end{rem}

\begin{rem}

Both Theorem \ref{thm:2DCounter} and Theorem \ref{thm:2DPeriodic} are concerned with periodic systems. The essential reason behind their qualitatively different tail behaviours can be summarised as follows. 

For the radially periodic example, when the system propagates through annuli, one encounters new geometry and there is no simple ballistic tunnel that the system could use. Any large distance propagation has to cause sufficiently many effective turns for the control.

On the other hand, for the torus system, there is always either a globally monotone direction (suspension case) or a periodic orbit on the fundamental domain $\mathbb{T}^2$ with nontrivial homology,  which can be reused by the system precisely due to the repeating geometry  the system encounters when it propagates through different fundamental domains. Both cases lead to a global ballistic tunnel with a small growth of Cameron-Martin norm.
\end{rem}

\newpage
\begin{appendices}

\section{Classification of nondegenerate torus flows}

In this appendix, we provide a self-contained proof of Theorem \ref{thm:1DFoliClass}. The argument is essentially an adaptation of the proof given by Hector-Hirsch \cite{HH86} to the smooth setting.

\subsection{Nondegenerate torus flow as 1D foliation}

Let $M=\mathbb{T}^{2}:=\mathbb{R}^2/\mathbb{Z}^2$. We always fix a smooth,
 nonvanishing vector field $W$ on $\mathbb{T}^{2}$. Given $x\in\mathbb{T}^{2},$
define the flow $t\mapsto\phi_{t}(x)$ by 
\[
\frac{d}{dt}\phi_{t}(x)=W\big(\phi_{t}(x)\big),\quad\phi_{0}(x)=x.
\]
Since $\mathbb{T}^{2}$ is compact, the flow $\phi_{t}(x)$ is well-defined
for all $t\in\mathbb{R}$. Let ${\cal F}_{W}$ denote the collection
of integral curves (flow lines) of $W$, namely, 
\[
{\cal F}_{W}:=\big\{\{\phi_{t}(x)\}_{t\in\mathbb{R}}:x\in\mathbb{T}^{2}\big\}.
\]
The starting point of the whole discussion is the following basic
fact.
\begin{prop}
\label{prop:VFFolia}The family ${\cal F}_{W}$ defines a 1D foliation
on $\mathbb{T}^{2}$ in the sense of Definition \ref{def:Folia}.
\end{prop}
\begin{proof}
It is clear that ${\cal F}_{W}$ is a family of immersed 1D manifolds
whose union is $\mathbb{T}^{2}.$ To prove the local flatness property,
let $x\in\mathbb{T}^{2}.$ Since $W(x)\neq0,$ one can pick a nonzero
vector $v\in T_{x}\mathbb{T}^{2}$ that is linearly independent of
$W(x).$ Let $\sigma:(-\delta,\delta)\rightarrow\mathbb{T}^{2}$ be
a smooth arc such that $\sigma(0)=x,$ $\dot{\sigma}(0)=v$. Define
\[
F:\mathbb{R}\times(-\delta,\delta)\rightarrow\mathbb{T}^{2},\quad F(t,r):=\phi_{t}(\sigma(r)).
\]
It is clear that $F(0,0)=x$ and $(dF)_{(0,0)}$ is nonsingular. Therefore,
$F$ is a local diffeomorphism from $V:=(-\varepsilon,\varepsilon)\times(-\delta',\delta')$
onto $F(V)$ with some small $\varepsilon$ and $\delta'<\delta$.
The map $F|_{V}$ provides a flat chart of $\mathbb{T}^{2}$ around
$x$; each horizontal arc $s\mapsto(s,r)$ corresponds to a flow line
by definition and every flow line inside $F(V)$ must come from a
horizontal arc since $W=\partial_{s}$ in $F(V)$.\textcolor{blue}{{} }
\end{proof}
Proposition \ref{prop:VFFolia} asserts that around every point $x\in\mathbb{T}^{2}$,
one can choose a \textit{flow box} $(U;s,r)$ ($s\in(-\delta_{1},\delta_{1})$,
$r\in(-\delta_{2},\delta_{2})$) such that $x=(0,0)$ and $W=\partial_{s}$
on $U$. In particular, flow lines of $W$ inside $U$ are precisely
the horizontal slices $r=c$ ($c\in(-\delta_{2},\delta_{2})$). This
is a basic but important (technical) fact that will be used for several
times. 
\begin{rem}
\label{rem:FlowVSFolia}The actual flow $\{\phi_{t}(x)\}$ carries
more information than the foliation ${\cal F}_{W}$ itself; the foliation
is determined by the \textit{1D distribution} $\{{\rm Span}W(x):x\in\mathbb{T}^{2}\}$
while the flow also records the orientation and the strength of the
field. 
\end{rem}
\begin{defn}
The leaf ${\cal O}_{x}:=\{\phi_{t}(x):t\in\mathbb{R}\}$ is called
the \textit{orbit} of $W$ passing through $x$. An orbit ${\cal O}_{x}$
is said to be \textit{periodic} if it is a compact leaf.
\end{defn}
It is a simple exercise that for a periodic orbit ${\cal O}_{x}$,
there exist a smallest $T>0$ such that $\phi_{t+T}(y)=\phi_{t}(y)$
for all $y\in{\cal O}_{x}$ and $t\in\mathbb{R}$. This number $T$
is called the \textit{period} of ${\cal O}_{x}$ (with respect to
$W$). Note that the period is not a foliation concept; it depends
on the actual flow $\phi_{t}(x)$ or equivalently, on the vector field
$W$ (cf. Remark \ref{rem:FlowVSFolia}). Another important feature
here is that the vector field $W$ orients every orbit ${\cal O}_{x}$
in a natural way.

\subsubsection{Periodic orbits}

In this subsection, we study basic topological and dynamical properties
of periodic orbits. This provides some essential tools for the later
proof of the main classification theorem.

\subsubsection*{Transversal intersection and the intersection number}

We first recall some notions from intersection theory (see \cite{Bre13,GP25}). 
\begin{defn}
Let $P,Q$ be two oriented, embedded 1D submanifolds of $\mathbb{T}^{2}.$
We say that $P$ is \textit{transversal} to $Q$ if $T_{x}P$ and
$T_{x}Q$ (as 1D subspaces of $T_{x}\mathbb{T}^{2}$) are linearly
independent at every $x\in P\cap Q$. In this case, the \textit{local
intersection number} of $P$ and $Q$ at $x\in P\cap Q$ is defined
by 
\[
i_{x}(P,Q):=\begin{cases}
1, & \text{if }\det(v_{x},w_{x})>0;\\
-1, & \text{if }\det(v_{x},w_{x})<0,
\end{cases}
\]
where $v_{x}$ (respectively, $w_{x}$) is any positively oriented
vector in $T_{x}P$ (respectively, $T_{x}Q$). The \textit{intersection
number} of $P$ and $Q$ is defined by 
\[
I(P,Q):=\sum_{x\in P\cap Q}i_{x}(P,Q),
\]
provided further that $P\cap Q$ is a finite set. 
\end{defn}
For our purpose, $P,Q$ are always assumed to be smooth, simple and
closed curves. Due to compactness, if they intersect transversally
their intersection must be a finite set and $I(P,Q)$ is thus well-defined.
In particular, $I(P,Q)=0$ if they do not intersect. It is also known
that the intersection number is a topological quantity, i.e. it depends
only on the homology classes of $P,Q$. Let $\{\sigma_{1},\sigma_{2}\}$
be the canonical basis of $H_{1}(\mathbb{T}^{2};\mathbb{Z})\cong\mathbb{Z}^{2}$.
In other words, 
\begin{equation}
\sigma_{1}:=\big[(t,0)_{0\leqslant t\leqslant1}\big],\quad\sigma_{2}:=\big[(0,t)_{0\leqslant t\leqslant1}\big].\label{eq:H1Basis}
\end{equation}
One has the following formula for the intersection number. 
\begin{lem}
\label{lem:IntNumForm}Let $P,Q$ be smooth, simple and closed curves
which intersect transversally and whose homology classes are given
by 
\begin{equation}
[P]=p_{1}\sigma_{1}+p_{2}\sigma_{2},\quad[Q]=q_{1}\sigma_{1}+q_{2}\sigma_{2},\label{eq:IntNumAssump}
\end{equation}
with $p_{i},q_{j}\in\mathbb{Z}$. Then $I(P,Q)=p_{1}q_{2}-p_{2}q_{1}$.
\end{lem}
\begin{proof}
Let 
\[
P=\left(\begin{array}{c}
P_{1}\\
P_{2}
\end{array}\right),Q=\left(\begin{array}{c}
Q_{1}\\
Q_{2}
\end{array}\right):S^{1}\rightarrow\mathbb{T}^{2}
\]
be parametrised by $S^{1}:=\mathbb{R}/\mathbb{Z}.$ Consider the map
$D:S^{1}\times S^{1}\rightarrow\mathbb{T}^{2}$ defined by 
\[
D(s,t):=P(s)-Q(t)\quad({\rm mod\ \mathbb{Z}}).
\]
At each intersection point $(s,t)\in D^{-1}(0),$ by the transversal
intersection assumption one has 
\[
\det\big((dD)_{(s,t)}\big)=-\det\big(\dot{P}(s),\dot{Q}(t)\big)\neq0.
\]
Therefore, $0$ is a regular value of $D$. It follows that $D^{-1}(0)$
is a zero-dimensional submanifold of $S^{1}\times S^{1}$. Hence,
it is a finite set of points, say 
\[
D^{-1}(0)=\{(s_{1},t_{1}),\cdots,(s_{k},t_{k})\}.
\]
There exists a neighbourhood $U_{k}$ of $(s_{k},t_{k})$ for each
$k$ and a neighbourhood $V$ of $0$, such that $D|_{U_{k}}$ is
a diffeomorphism from $U_{k}$ onto $V$. In addition, one has 
\[
{\rm sgn}\det\big((dD)_{(s_{k},t_{k})}\big)=-{\rm sgn}\det\big(\dot{P}(s_{k}),\dot{Q}(t_{k})\big)=-i_{x_{k}}(P,Q),
\]
where $x_{k}:=P(s_{k})=Q(t_{k})$. 

Let $\omega$ be a smooth $2$-form that is compactly supported inside
$V$ and satisfies $\int_{V}\omega=1$. Since $D|_{U_{k}}:U_{k}\rightarrow V$
is a diffeomorphism, one has 
\[
\int_{U_{k}}D|_{U_{k}}^{*}\omega={\rm sgn}\det\big((dD)_{(s_{k},t_{k})}\big)\int_{V}\omega=-i_{x_{k}}(P,Q).
\]
Therefore, 
\[
\int_{S^{1}\times S^{1}}D^{*}\omega=\sum_{k=1}^{m}\int_{U_{k}}D|_{U_{k}}^{*}\omega=-I(P,Q).
\]
Let $\nu=dx\wedge dy\in\Omega^{2}(\mathbb{T}^{2})$. Since both $\omega$
and $\nu$ are closed and have the same integral on $\mathbb{T}^{2}$
(which is equal to one), they differ by an exact form. As a result,
\begin{equation}
\int_{S^{1}\times S^{1}}D^{*}\nu=\int_{S^{1}\times S^{1}}D^{*}\omega=-I(P,Q).\label{eq:IntNumPf1}
\end{equation}
Now it remains to compute the integral of $D^{*}\nu$.

By using the explicit expressions of $D$ and $\nu,$ it is plain
to check that 
\[
D^{*}\nu=\big(P_{2}'(s)Q_{1}'(t)-P_{1}'(s)Q_{2}'(t)\big)ds\wedge dt.
\]
Therefore, 
\begin{align*}
\int_{S^{1}\times S^{1}}D^{*}\nu & =\int_{0}^{1}\int_{0}^{1}\big(P_{2}'(s)Q_{1}'(t)-P_{1}'(s)Q_{2}'(t)\big)dsdt\\
 & =\Big(\int_{0}^{1}P_{2}'(s)ds\Big)\Big(\int_{0}^{1}Q_{1}'(t)dt\Big)-\Big(\int_{0}^{1}P_{1}'(s)ds\Big)\Big(\int_{0}^{1}Q_{2}'(t)dt\Big).
\end{align*}
Let $\tilde{P},\tilde{Q}$ be some lifts of $P,Q$ on the universal
cover $\mathbb{R}^{2}$. According to the assumption (\ref{eq:IntNumAssump}),
one has
\[
\int_{0}^{1}P'(s)ds=\int_{0}^{1}\tilde{P}'(s)ds=\tilde{P}(1)-\tilde{P}(0)=(p_{1},p_{2}),
\]
\[
\int_{0}^{1}Q'(t)dt=\int_{0}^{1}\tilde{Q}'(t)dt=\tilde{Q}(1)-\tilde{Q}(0)=(q_{1},q_{2}).
\]
As a consequence, 
\begin{equation}
\int_{S^{1}\times S^{1}}D^{*}\nu=p_{2}q_{1}-p_{1}q_{2}.\label{eq:IntNumPf2}
\end{equation}
The conclusion of the lemma follows from the relations (\ref{eq:IntNumPf1})
and (\ref{eq:IntNumPf2}). 
\end{proof}
\begin{rem}
\label{rem:IntNumNoInt}Lemma \ref{lem:IntNumForm} remains valid
if $P\cap Q=\emptyset.$ In this case, $I(P,Q)=p_{1}q_{2}-p_{2}q_{1}=0$.
\end{rem}

\subsubsection*{Topological properties of periodic orbits}

Next, we study periodic orbits of a smooth,  nonvanishing vector field
$W$. We first make the following definition.
\begin{defn}
A (continuous) closed curve $\gamma$ in $\mathbb{T}^{2}$ is said
to be \textit{essential} if its (singular) homology class $[\gamma]\in H_{1}(\mathbb{T}^{2};\mathbb{Z})$
is nonzero. It is \textit{primitive} if its homology class cannot
be expressed as $[\gamma]=m\sigma$ for some $\sigma\in H_{1}(\mathbb{T}^{2};\mathbb{Z})$
and $m\geqslant2.$ 
\end{defn}
Let $\{\sigma_{1},\sigma_{2}\}$ be the canonical basis of $H_{1}(\mathbb{T}^{2};\mathbb{Z})$
defined by (\ref{eq:H1Basis}). Write $[\gamma]=p\sigma_{1}+q\sigma_{2}$
for a closed curve $\gamma.$ Then $\gamma$ is primitive if and only
if $\gcd(|p|,|q|)=1$. This is also equivalent to saying that $[\gamma]$
can be extended to a basis of $H_{1}(\mathbb{T}^{2};\mathbb{Z})$.

The main topological properties of periodic orbits are summarised
in the following result. 
\begin{prop}
\label{prop:POTop}(i) Every periodic orbit is essential.

\vspace{2mm}\noindent (ii) Every essential, simple closed curve is
primitive. In particular, every periodic orbit is primitive. 

\vspace{2mm}\noindent (iii) For any pair of periodic orbits $\gamma_{1},\gamma_{2},$
one has $[\gamma_{1}]=\pm[\gamma_{2}]$.
\end{prop}
\begin{proof}
(i) Let $\gamma$ be a periodic orbit of $W$. Suppose on the contrary
that $[\gamma]=0.$ Since $\gamma$ is simple closed, it must bound
a topological disk $D$ in this case. According to the Poincar\'e-Hopf
theorem for the disk $D$ (see \cite[Chapter 3, Theorem 2]{Per13}),
one has 
\[
{\rm ind}_{W}(\partial D)=\sum_{z\in{\cal Z}(W)\cap \mathrm{int}D}{\rm ind}_{W}(z).
\]
Here ${\rm ind}_{W}(\partial D)$ denotes the index of $W$ on the
boundary $\partial D$ and ${\rm ind}_{W}(z)$ denotes the local index
of $W$ at the zero $z$ (${\cal Z}(W)$ denotes the set of zeros
of $W$). It is plain from the definition that ${\rm ind}_{W}(\partial D)=\pm1$
(the sign depends on the orientation convention of $\gamma$). As
a result, $W$ must have at least one zero in the interior of $D$.
This is a contradiction. Therefore, $[\gamma]\neq0$.

\vspace{2mm}\noindent (ii) Let $\gamma$ be an essential, simple
closed curve. We claim that $\mathbb{T}^{2}\backslash\gamma$ is connected.
Suppose on the contrary that $\mathbb{T}^{2}\backslash\gamma$ is
disconnected. Let $N\cong S^{1}\times(-\varepsilon,\varepsilon)$
be a tubular neighbourhood of $\gamma$. Then 
\[
N_{-}:=S^{1}\times(-\varepsilon,0),\quad N_{+}:=S^{1}\times(0,\varepsilon)
\]
must be contained in different components of $\mathbb{T}^{2}\backslash\gamma$;
otherwise, any other component $V$ would intersect this component
(because $\gamma\subseteq\bar{V}$) which contradicts disconnectedness.
Let $U_{+}$ be the component containing $N_{+}.$ Since $N_{-}$
is contained in a different component, $\overline{U_{+}}$ is a compact
surface with boundary $\gamma$. As a consequence, $\gamma=\partial c$
for some singular $2$-chain. It follows that $[\gamma]=0$ which
leads to a contradiction. Therefore, $\mathbb{T}^{2}\backslash\gamma$
is connected.

Pick two points $x_{+}\in N_{+}$ and $x_{-}\in N_{-}.$ Draw a smooth,
simple curve $\eta$ from $x_{+}$ to $x_{-}$ inside $\mathbb{T}^{2}\backslash\gamma$
(which is possible due to connectedness). Let $\sigma$ be a simple
arc which joins $x_{+}$ to $x_{-}$ inside $N$. By smoothing the
corners and removing loops from $\eta\sqcup\sigma$ if necessary,
one obtains a simple closed curve $\beta$ that has precisely one
(transversal) intersection with $\gamma$. Therefore, 
\begin{equation}
I(\gamma,\beta)=\pm1.\label{eq:PrimPf}
\end{equation}
Suppose that $\gamma$ is not primitive. Under the canonical basis
$\{\sigma_{1},\sigma_{2}\}$ of $H_{1}(\mathbb{T}^{2};\mathbb{Z})$,
one can write $[\gamma]=m(p,q)$ for some $m\geqslant2$ and $p,q\in\mathbb{Z}$.
Let $[\beta]=(r,s)$. According to Lemma \ref{lem:IntNumForm}, one
has 
\[
I(\gamma,\beta)=m(ps-qr).
\]
This contradicts the fact that (\ref{eq:PrimPf}). Therefore, $\gamma$
is primitive. 

\vspace{2mm}\noindent (iii) Let $\gamma_{1}=(p,q)$ and $\gamma_{2}=(r,s)$
be two periodic orbits. Since $\gamma_{1}$ is primitive by Part (ii),
the integers $p,q$ are coprime. Similarly, $r,s$ is also coprime.
Since $\gamma_{1}\cap\gamma_{2}=\emptyset$, one knows from Lemma
\ref{lem:IntNumForm} (cf. Remark \ref{rem:IntNumNoInt}) that 
\[
(p,q)=\frac{a}{b}(r,s)
\]
for some coprime pair $(a,b)\in\mathbb{Z}^{2}$. In particular, 
\[
ar=bp,\quad as=bq.
\]
It follows that $a|p$ and $a|q$ which implies that $a=\pm1$. Similarly,
$b=\pm1$. Therefore, $[\gamma_{1}]=\pm[\gamma_{2}]$.
\end{proof}

\subsubsection{\label{subsec:PoinMap}The Poincar\'e map }

In this subsection, we discuss an elementary but rather useful construction:
the Poincar\'e return map. This construction (or its variant) will
be used for several times in the proof of the main classification
theorem. 

Let $\gamma=\{\phi_{t}(x)\}_{t\in\mathbb{R}}$ be a periodic orbit
of $W$ with period $T$. Let $\Sigma$ be a short smooth arc that
intersects $\gamma$ transversally at $x$. Suppose that $\phi_{t}(x)\notin\Sigma$
for all $t\in(0,T)$. This is possible by shrinking $\Sigma$ if necessary;
for instance, in a suitable flow box $(U;s,r)$ around $x$ one can
assume that $\Sigma=\{s=0\}.$

We construction a map
\[
P:U\subseteq\Sigma\rightarrow\Sigma
\]
on a sufficiently small neighbourhood $U\subseteq\Sigma$ of $x$
in the following way. Pick $y\in\Sigma$. Consider the flow line $t\mapsto\phi_{t}(y)$.
Let $\tau(y)>0$ be the first time that $\phi_{t}(y)$ returns to
$\Sigma$. We then define 
\[
P(y):=\phi_{\tau(y)}(y).
\]

\begin{figure}[H] 
\begin{center}  
\includegraphics[scale=0.17]{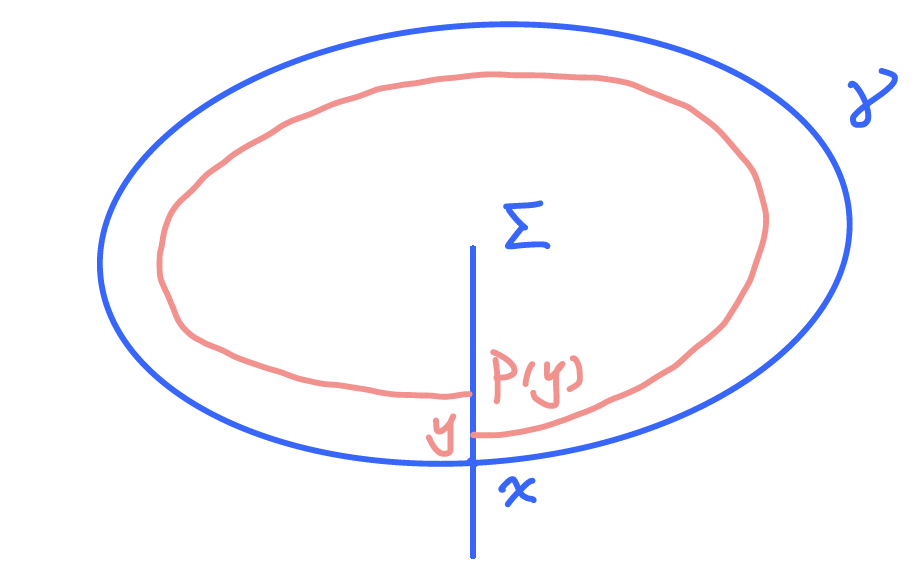} 
\end{center}
\end{figure}
\begin{lem}
\label{lem:PoinMap}There exists a small neighbourhood $U\subseteq\Sigma$
of $x$, such that $(\tau(y),P(y))$ is well-defined for all $y\in U$.
In addition, both $\tau(y)$, $P(y)$ are smooth on $U$ and one has
$\tau(x)=T,$ $P(x)=x.$ 
\end{lem}
\begin{proof}
First of all, there exists a neighbourhood $U_{0}\subseteq\Sigma$
of $x$, such that $U_{0}$ is represented by the level set $\{\rho=0\}$
of some smooth function $\rho$ near $x$. Now consider the function
\[
F:U_{0}\times(T-\delta_{0},T+\delta_{0})\rightarrow\mathbb{R},\quad F(y,t):=\rho(\phi_{t}(y))
\]
where $\delta_{0}$ is small. Note that $F(x,T)=0.$ In addition,
since $T_{x}\Sigma=\ker(d\rho)_{x}$ and $\Sigma$ intersects $\gamma$
transversally at $x$, one has 
\begin{equation}
\partial_{t}F(x,T)=\big\langle(d\rho)_{x},W(x)\big\rangle\neq0.\label{eq:PMPf}
\end{equation}
According to the implicit function theorem, there exists $U_{1}\subseteq U_{0}$,
$\delta_{1}<\delta_{0}$ and a smooth function $\tau:U_{1}\rightarrow(T-\delta_{1},T+\delta_{1})$
such that $\tau(x)=T$ and 
\[
F(y,\tau(y))=0
\]
for all $y\in U_{1}$. 

We claim that by shrinking $U_{1}$ to some smaller neighbourhood
$U$ of $x$ if necessary, one has 
\[
\tau(y)=\inf\{t>0:\phi_{t}(y)\in\Sigma\}
\]
for all $y\in U.$ First of all, according to (\ref{eq:PMPf}) there
exists $U_{2}\subseteq U_{1}$ and $\delta_{2}<\delta_{1},$ such
that $t\mapsto F(y,t)$ is strictly monotone on $[T-\delta_{2},T+\delta_{2}]$
for all $y\in U_{2}.$ As a result, $\tau(y)$ is the unique $t\in[T-\delta_{2},T+\delta_{2}]$
such that $F(y,t)=0$. We fix this $\delta_{2}$ in what follows.
Next, by using a flow box argument, there exists $U_{3}\subseteq U_{2}$
and $\varepsilon>0,$ such that $\phi_{t}(y)\notin\Sigma$ for all
$y\in U_{3}$ and $t\in(0,\varepsilon]$. In addition, since 
\[
\big\{\phi_{t}(x):\varepsilon\leqslant t\leqslant T-\delta_{2}\big\}\cap\Sigma=\emptyset,
\]
one knows from continuity that $\phi_{t}(y)\notin\Sigma$ for all
$y$ in some smaller $U\subseteq U_{3}$ and all $t\in[\varepsilon,T-\delta_{2}]$.
Combining the above three scenarios, it follows that $\tau(y)$ is
the first return time of $\phi_{t}(y)$ to $\Sigma$ for every $y\in U$. 

Consequently, the map 
\[
P:U\rightarrow\Sigma,\quad P(y):=\phi_{\tau(y)}(y)
\]
is well-defined. Its smoothness follows immediately from the smoothness
of $\tau$ and the joint smoothness of $(t,y)\mapsto\phi_{t}(y)$.
This completes the proof of the lemma. 
\end{proof}
\begin{defn}
The map $P$ constructed above is known as the \textit{Poincar\'e
map} at $x\in\gamma$ with respect to the transversal $\Sigma$. 
\end{defn}
Since $\Sigma$ is 1D and $P(x)=x,$ the linear map $(dP)_{x}:T_{x}\Sigma\rightarrow T_{x}\Sigma$
is a scalar multiple of the identity, i.e. $(dP)_{x}=\lambda_{\gamma}{\rm id}$
for some $\lambda_{\gamma}\in\mathbb{R}.$ Although the Poincar\'e
map $P$ depends on the point $x\in\gamma$ and the transversal $\Sigma$,
it is a remarkable fact that the number $\lambda_{\gamma}$ is intrinsic
to the periodic orbit $\gamma.$ 
\begin{prop}
\label{prop:CharaNum}The number $\lambda_{\gamma}$ does not depend
on the point $x\in\gamma$ or the transversal $\Sigma.$ In addition,
one has $\lambda_{\gamma}>0$.
\end{prop}
\begin{proof}
We give an alternative intrinsic construction of $\lambda_{\gamma}$. 

\vspace{2mm}\noindent (i) Define the function $J_{t}\in C^{\infty}(\mathbb{T}^{2})$
by 
\[
\phi_{t}^{*}\mu_{g}=J_{t}\mu_{g},
\]
where $\mu_{g}$ denotes the Riemannian volume form on the (oriented)
annulus $\mathbb{T}^{2}$. It is not hard to see that $J_{t}(x)=\det(d\phi_{t})_{x}$,
where the determinant is defined with respect to any positively oriented
orthonormal bases of $T_{x}\mathbb{T}^{2}$ and $T_{\phi_{t}(x)}\mathbb{T}^{2}$.
We claim that 
\begin{equation}
J_{t}(x)=\exp\Big[\int_{0}^{t}({\rm div}_{g}W)(\phi_{s}(x))ds\Big],\quad\text{(Liouville's formula)}\label{eq:Jtx}
\end{equation}
where ${\rm div}_{g}$ denotes the Riemannian divergence. To this
end, one first recalls the following standard Riemannian geometry
fact: 
\[
{\cal L}_{W}\mu_{g}=({\rm div}_{g}W)\mu_{g}.
\]
In addition, by the definition of the Lie derivative, one has 
\[
\frac{d}{dt}\big(\phi_{t}^{*}\mu_{g}\big)=\phi_{t}^{*}\big({\cal L}_{W}\mu_{g}\big).
\]
As a consequence, 
\begin{equation}
\dot{J}_{t}(x)=\big({\rm div}_{g}W\big)(\phi_{t}(x))J_{t}(x).\label{eq:JODE}
\end{equation}
Since $J_{0}(x)=1$, the desired relation (\ref{eq:Jtx}) follows
by solving the ODE (\ref{eq:JODE}). 

\vspace{2mm}\noindent (ii) Suppose that $\gamma$ is a periodic orbit
of $W$ with period $T.$ According to Part (i), one has 
\begin{equation}
\det(d\phi_{T})_{x}=\exp\Big[\int_{0}^{T}({\rm div}_{g}W)(\phi_{s}(x))ds\Big].\label{eq:Detdphi}
\end{equation}
In particular, $\det(d\phi_{T})_{x}$ is independent of $x$. 

\vspace{2mm}\noindent (iii) Under the assumption of Part (ii), let
$\Sigma$ be a transversal curve through $x\in\gamma.$ Let $v\in T_{x}\Sigma$
be a nonzero vector. By the definition of the Poincar\'e map, one
has 
\[
P(y)=\phi_{\tau(y)}(y),\quad y\in U\subseteq\Sigma.
\]
It follows from the definition of $\lambda_{\gamma}$ that 
\[
\lambda_{\gamma}v=(dP)_{x}v=(d\phi_{T})_{x}v+\langle d\tau,v\rangle_{x}W(x).
\]
Also note that $(d\phi_{T})_{x}W(x)=W(x)$ since $\gamma$ is a periodic
orbit. Therefore, the matrix of $(d\phi_{T})_{x}:T_{x}\mathbb{T}^{2}\rightarrow T_{x}\mathbb{T}^{2}$
with respect to the basis $\{W(x),v\}$ is given by 
\[
(d\phi_{T})_{x}=\left(\begin{array}{cc}
1 & -\langle d\tau,v\rangle_{x}\\
0 & \lambda_{\gamma}
\end{array}\right).
\]
This shows that $\det(d\phi_{T})_{x}=\lambda_{\gamma}.$ In particular,
$\lambda_{\gamma}$ is independent of $x$ and $\Sigma$. The positivity
of $\lambda_{\gamma}$ follows directly from the relation (\ref{eq:Detdphi}).
\end{proof}
\begin{defn}
The number $\lambda_{\gamma}$ is called the \textit{transverse characteristic
number} of $\gamma.$ A periodic orbit $\gamma$ is said to be hyperbolic
if $\lambda_{\gamma}\neq1$.
\end{defn}
\begin{rem}
Hyperbolic periodic orbits have strong dynamical stability (exponential
attraction of nearby orbits in the case $\lambda_{\gamma}\in(0,1)$).
Such a stability property is used in the paper \cite{BGW26} to establish
a Gaussian lower tail estimate for 2D periodic, uniformly elliptic
systems driven by a fractional Brownian motion. 
\end{rem}

\subsubsection{The Poincar\'e-Bendixson theorem}

In general, periodic orbits of $W$ need not always exist. However,
if one knows that an end of a nonperiodic orbit is trapped in a compact
region \textit{inside a planar domain}, its limit set must be a periodic
orbit. This is the content of the \textit{Poincar\'e-Bendixson theorem},
which provides a general way of identifying periodic orbits. 

To state the theorem, we first need to introduce some definitions.
\begin{defn}
Let $x\in\mathbb{T}^{2}$. The \textit{forward orbit} passing through
$x$ is defined by 
\[
{\cal O}_{x}^{+}:=\{\phi_{t}(x):t\geqslant0\}.
\]
The \textit{forward limit set} of $x$ is defined by 
\[
\omega(x):=\bigcap_{T>0}\overline{\{\phi_{t}(x):t\geqslant T\}}.
\]
One can define the \textit{backward orbit} ${\cal O}_{x}^{-}$ and
the \textit{backward limit set} $\alpha(x)$ in a similar way. A set
$\Lambda\subseteq\mathbb{T}^{2}$ is \textit{$W$-invariant} if $\phi_{t}\Lambda\subseteq\Lambda$
for all $t\in\mathbb{R}$. 
\end{defn}
By compactness, both $\omega(x)$ and $\alpha(x)$ are nonempty, $W$-invariant
compact subsets of $\mathbb{T}^{2}.$ In addition, a point $y\in\omega(x)$
(respectively, $y\in\alpha(x)$) if and only if $\phi_{t_{n}}(x)\rightarrow y$
for some sequence $t_{n}\rightarrow\infty$ (respectively, $t_{n}\rightarrow-\infty$).

\begin{thm}[The Poincar\'e-Bendixson Theorem]\label{thm:PB} Let $X$
be a smooth,  nonvanishing vector field on a planar domain $U\subseteq\mathbb{R}^{2}.$
Let $x\in U$ be given fixed. Suppose that the flow $\phi_{t}(x)$
of $X$ exists for all $t\geqslant0$ and 
\[
\{\phi_{t}(x):t\geqslant0\}\subseteq K
\]
for some compact set $K\Subset U$. Then there exists a periodic orbit
$C\subseteq K$ such that $\omega(x)=C$. A similar assertion holds
for $\alpha(x)$.

\end{thm}

\begin{rem}
The planar assumption is crucial here; it is certainly possible for
a vector field on $U=\mathbb{T}^{2}$ to have no periodic orbits. 
\end{rem}
Before proving Theorem \ref{thm:PB}, we first state a key lemma which
indicates the essential role of the planar assumption. 
\begin{lem}
\label{lem:MonoLem}Let $X$ be a smooth, nonvanishing vector field
on a planar domain $U\subseteq\mathbb{R}^{2}$. Let $\gamma$ be a
nonperiodic orbit of $X$ which exists in $U$ for all $t$. Suppose
that 
\[
V=\big\{(s,r):-\delta_{1}<s<\delta_{2},-\delta_{2}<r<\delta_{2}\big\}
\]
is a local flow box, so that $X=\partial_{s}$ in $V$ and 
\[
\Sigma:=\big\{(s,r):s=0,-\delta_{2}<r<\delta_{2}\big\}
\]
is the vertical transversal. Let $0\leqslant t_{1}<t_{2}<\cdots$
be a sequence of times such that $z_{j}:=\gamma_{t_{j}}\in\Sigma$
and $\gamma_{t}\notin\Sigma$ for all $t\in(t_{j-1},t_{j})$. Then
either 
\[
z_{1}<z_{2}<z_{3}<\cdots
\]
or 
\[
z_{1}>z_{2}>z_{3}>\cdots.
\]
Here we identify $\Sigma$ with its parametrisation interval $(-\delta_{2},\delta_{2})$.
\end{lem}
\begin{proof}
Let $z_{0}<z_{1}$. Suppose on the contrary that $z_{2}<z_{1}$. There
are two cases to discuss.

\vspace{2mm}\noindent \textit{\uline{Case I: \mbox{$z_{2}\in(z_{0},z_{1}).$}}} 

\vspace{2mm} Let $\Gamma_{01}$ be the concatenation of $\gamma_{(t_{0},t_{1})}\subseteq\Sigma^{c}$
and the arc connecting $z_{1}$ to $z_{0}$ along $\Sigma$. Then
$\Gamma_{01}$ is a simple closed curve. According to the Jordan curve
theorem, $\Gamma_{0,1}$ divides the plane into a bounded domain $D_{1}$
and an unbounded domain $D_{2}$. Since $\gamma_{t_{1}+\varepsilon}\in D_{2}$
for small $\varepsilon$, one knows by connectedness that $\gamma|_{(t_{1},t_{2})}\subseteq D_{2}$.
However, since $z_{2}\in(z_{0},z_{1})$, one must have $\gamma_{t_{2}-\varepsilon}\in D_{1}$
due to the flow line pattern inside the flow box $V$. This is a contraction. 

\vspace{2mm}\noindent \textit{\uline{Case II: \mbox{$z_{2}<z_{0}$}.}} 

\vspace{2mm} Let $\Gamma_{12}$ be the concatenation of $\gamma_{(t_{1},t_{2})}\subseteq\Sigma^{c}$
and the arc connecting $z_{2}$ to $z_{1}$ along $\Sigma$. Then
$\Gamma_{12}$ divides the plane into a bounded domain $D_{1}$ and
an unbounded one $D_{2}.$ Note that $\gamma_{t_{1}-\varepsilon}\in D_{1}$
for small $\varepsilon$ and thus $\gamma_{(t_{0},t_{1})}\subseteq D_{1}.$
This is a contradiction since $\gamma_{t_{0}+\varepsilon}\in D_{2}$
for small $\varepsilon$. \begin{figure}[H] 
\begin{center}  
\includegraphics[scale=0.16]{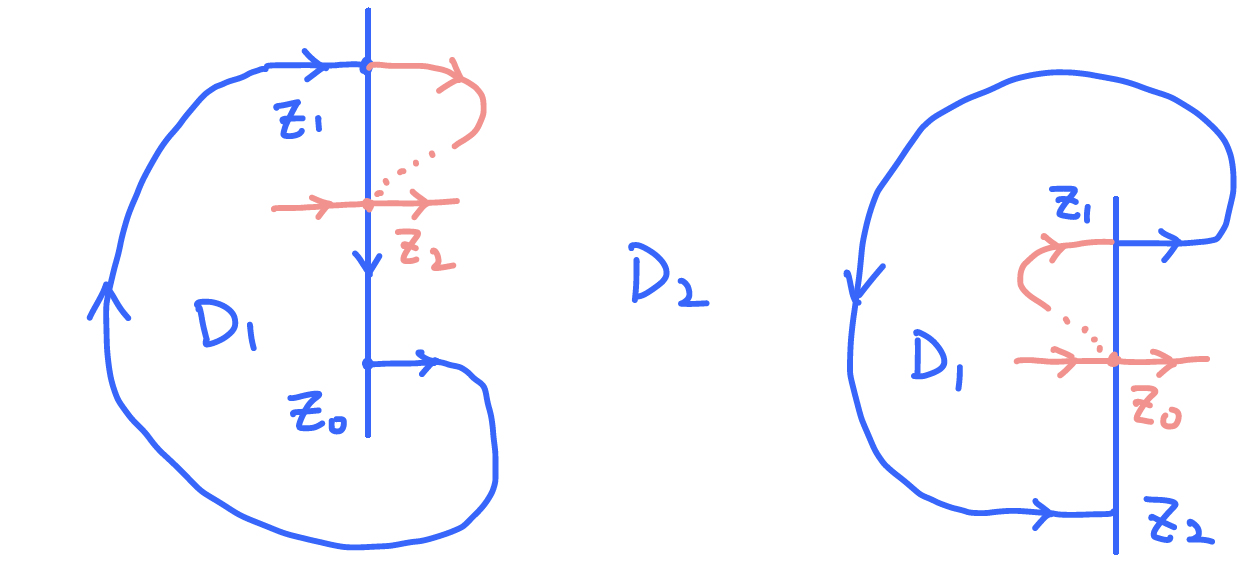} 
\end{center}
\end{figure}\noindent Therefore, one concludes that $z_{2}>z_{1}$. The case
when $z_{0}>z_{1}$ is treated in a similar way. 
\end{proof}
Lemma \ref{lem:MonoLem} has the following two important corollaries,
which will be used in the proof of Theorem \ref{thm:PB}. A point
$x$ is said to be \textit{positively recurrent} (respectively, \textit{negatively
recurrent}) if $x\in\omega(x)$ (respectively, if $x\in\alpha(x)$).
\begin{lem}
\label{lem:PRPerOrb}Let $X$ be a smooth,  nonvanishing vector field
on a planar domain $U\subseteq\mathbb{R}^{2}$. Let $y\in U$ be a
positively recurrent point of $X$ and assume that $\phi_{t}(y)$
exists in $U$ for all time. Then ${\cal O}_{y}$ is a periodic orbit.
\end{lem}
\begin{proof}
Suppose on the contrary that ${\cal O}_{y}$ is a nonperiodic orbit.
Let $V$ be a flow box centered at $y$. Since $y\in\omega(y),$ there
exists a sequence $t_{n}\rightarrow\infty$ such that $\phi_{t_{n}}(y)\rightarrow y$.
Inside the flow box $V$, one may adjust the time $t_{n}$ to some
$s_{n}$ (for all large $n$, say $n>N$) so that 
\[
s_{n}\rightarrow\infty,\quad\phi_{s_{n}}(y)\in\Sigma\text{ are all distinct},\quad\phi_{s_{n}}(y)\rightarrow y.
\]
Note that it takes a uniformly positive amount of time for the flow
line to leave the flow box once it is on $\Sigma$. As a result, there
are at most finitely many return times to $\Sigma$ between $0$ and
$s_{N+1}$. It follows from Lemma \ref{lem:MonoLem} that the sequence
\[
y=\phi_{0}(y),\phi_{s_{N+1}}(y),\phi_{s_{N+2}}(y),\cdots
\]
is monotone on $\Sigma$. This contradicts the fact that $\phi_{s_{n}}(y)\rightarrow y$
as $n\rightarrow\infty$. Therefore, ${\cal O}_{y}$ is a periodic
orbit.
\end{proof}
\begin{lem}
\label{lem:OmeLimPerOrb}Let $X$ be a smooth nonvanishing vector
field on a planar domain $U\subseteq\mathbb{R}^{2}.$ Let $x\in U$
be given fixed. Suppose that the flow $\phi_{t}(x)$ of $X$ exists
for all $t\geqslant0$. Let $C$ be a periodic orbit such that $C\cap\omega(x)\neq\emptyset$.
Then $\omega(x)=C.$
\end{lem}
\begin{proof}
If ${\cal O}_{x}$ is periodic, one trivially has $\omega(x)={\cal O}_{x}=C$.
We therefore assume that ${\cal O}_{x}$ is nonperiodic. By assumption,
let $p\in C\cap\omega(x)$. Then $\phi_{t_{n}}(x)\rightarrow p$ for
some $t_{n}\rightarrow\infty.$ We first claim that $C\subseteq\omega(x)$.
Let $q\in C$. One can write $q=\phi_{t}(p)$ for some $t\geqslant0.$
It follows that $\phi_{t_{n}+t}(x)\rightarrow q$, which implies that
$q\in\omega(x)$. Therefore, $C\subseteq\omega(x)$. 

Conversely, let $q\notin C.$ Our aim is to show that $q\notin\omega(x)$.
Since $C$ is compact, it suffices to show that 
\begin{equation}
\lim_{t\rightarrow\infty}d(\phi_{t}(x),C)=0.\label{eq:OrbAsymC}
\end{equation}
To this end, let $V$ be a flow box centered at $p$ with transversal
$\Sigma$. Since $\phi_{t_{n}}(x)\rightarrow p$ and ${\cal O}_{x}$
is nonperiodic, similar to the proof of Lemma \ref{lem:PRPerOrb},
one can find a sequence $s_{n}\rightarrow\infty$ such that the following
properties hold true:

\vspace{2mm}\noindent (i) $z_{n}:=\phi_{s_{n}}(x)\in\Sigma$ and
$\phi_{s_{n}}(x)\rightarrow p$;

\vspace{2mm}\noindent (ii) $\phi_{t}(x)\notin\Sigma$ for $t\in(s_{n-1},s_{n})$;

\vspace{2mm}\noindent (iii) the sequence $z_{1},z_{2},z_{3},\cdots$
is monotone on $\Sigma.$\begin{figure}[H] 
\begin{center}  
\includegraphics[scale=0.14]{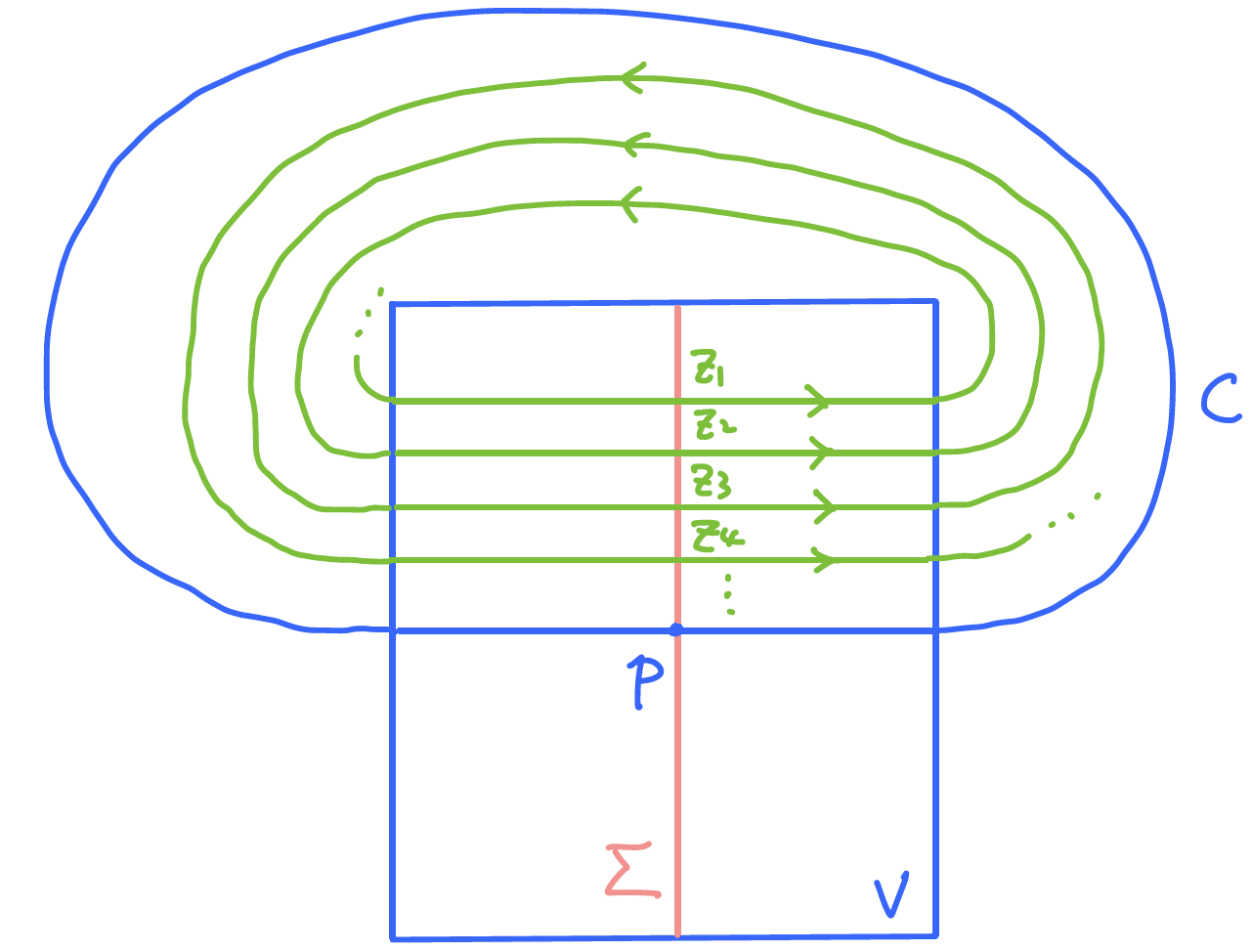} 
\end{center}
\end{figure}\noindent Consequently, one has 
\begin{equation}
z_{n}\downarrow p\quad\text{or}\quad z_{n}\uparrow p\quad\text{on }\Sigma.\label{eq:MonoPerObLemPf}
\end{equation}
Note that $z_{n+1}=P(z_{n})$ for all sufficiently large $n,$ where
$P$ is the Poincar\'e map defined in Section \ref{subsec:PoinMap}.
By the continuity of the Poincar\'e map, one has 
\begin{equation}
\lim_{z\rightarrow p}\sup_{0\leqslant t\leqslant\tau(z)}d(\phi_{t}(z),C)=0,\label{eq:1IterAsymp}
\end{equation}
where $\tau(z)$ is the first return time of $\phi_{t}(z)$ to $\Sigma$.
The claim (\ref{eq:OrbAsymC}) is now a simple consequence of the
monotone convergence property (\ref{eq:MonoPerObLemPf}) and the one-iteration
asymptotics (\ref{eq:1IterAsymp}). 
\end{proof}
\begin{rem}
Lemma \ref{lem:OmeLimPerOrb} remains valid on the torus $\mathbb{T}^{2}.$
In fact, by cutting $\mathbb{T}^{2}$ along the periodic orbit $C,$
one obtains an annulus $A\cong S^{1}\times[0,1]$ whose two boundary
curves are copies of $C$ and ${\cal O}_{x}\subseteq{\rm int}A$ (assuming
${\cal O}_{x}$ is nonperiodic). By thickening $A$ across its boundary
to an open annulus $U$ ($U$ is diffeomorphic to a planar domain),
one is now reduced to the situation of Lemma \ref{lem:OmeLimPerOrb}.
\end{rem}
We are now in a position to prove the Poincar\'e-Bendixson theorem.

\begin{proof}[Proof of Theorem \ref{thm:PB}] Suppose that 
\[
{\cal O}_{x}^{+}\subseteq K\Subset U.
\]
Then $\omega(x)\subseteq K.$ Since $\omega(x)$ is nonempty, compact
and $W$-invariant, it contains a minimal compact, $W$-invariant
subset (a nonempty, compact $W$-invariant subset which contains no
proper, nontrivial invariant subsets), say $C$. Let $y\in C.$ Then
\[
\omega(y)\subseteq C\implies\omega(y)=C
\]
since $C$ is minimal. In particular, $y\in\omega(y)$. According
to Lemma \ref{lem:PRPerOrb}, ${\cal O}_{y}$ is a periodic orbit.
Hence $\omega(x)$ contains a periodic orbit. It follows from Lemma
\ref{lem:OmeLimPerOrb} that $\omega(x)=C$.

\end{proof}
\begin{rem}
Here is a simple variant of the Poincar\'e-Bendixson theorem. Let
$W$ be a smooth, nonvanishing vector field on a compact surface $M$
without boundary. Suppose that there exists 
\[
K\ \text{compact},\quad U\ \text{open},\quad K\Subset U\subseteq M
\]
and $t_{0}\geqslant0,$ such that $U$ is diffeomorphic to a planar
domain and $\phi_{t}(x)\in K$ for all $t\geqslant t_{0}$. Then $\omega(x)$
is a periodic orbit of $W$. A similar assertion holds for $\alpha(x)$.
\end{rem}
We discuss a few applications of the Poincar\'e-Bendixson
theorem. Let $W$ be a smooth, nonvanishing vector field on $\mathbb{T}^{2}$.
\begin{lem}
\label{lem:DistEndPO} Let $A\cong S^{1}\times[0,1]$ be an embedded
annulus in $\mathbb{T}^{2}$ whose two boundary curves are compact
leaves of ${\cal F}_{W}$. Let $x\in{\rm int}A$ and assume that ${\cal O}_{x}$
is nonperiodic. Then $\alpha(x)$ and $\omega(x)$ are distinct periodic
orbits. 
\end{lem}
\begin{proof}
Let $U\supseteq A$ be a slightly larger open annulus which is diffeomorphic
to a planar domain. According to the Poincar\'e-Bendixson theorem,
both $\alpha(x)$ and $\omega(x)$ are periodic orbits and they are
contained in $A$. Suppose on the contrary that $\alpha(x)=\omega(x)=C.$
Let $p\in C$. Then both tails of ${\cal O}_{x}$ are arbitrarily
close to $C$. In particular, there exists $T_{0}>0$ such that $\phi_{t}(x)\in U$
for all $|t|>T_{0}.$ 

Now choose a small flow box centered at $p$ with transversal $\Sigma$.
Similar to the proof of Lemma \ref{lem:PRPerOrb}, one can find a
two-sided sequence 
\[
-\infty\downarrow s_{-n}<\cdots<s_{-2}<s_{-1}<s_{1}<s_{2}<\cdots<s_{n}\uparrow\infty
\]
with the following properties:

\vspace{2mm}\noindent (i) $z_{n}:=\phi_{s_{n}}(x)\in\Sigma$ and
$\phi_{t}\notin\Sigma$ for all $t\in(s_{n-1},s_{n})$ and for all
$n\in\mathbb{Z}$;

\vspace{2mm}\noindent (ii) $z_{n}\rightarrow p$ as $n\rightarrow\pm\infty.$

\vspace{2mm}\noindent (iii) The sequence $\{z_{n}\}_{n=-\infty}^{\infty}$
is monotone on $\Sigma.$

\vspace{2mm}\noindent But the second and third properties force the
sequence to be constant, which would imply that ${\cal O}_{x}$ is
periodic. This leads to a contradiction. Therefore, $\alpha(x)$ and
$\omega(x)$ are distinct periodic orbits.
\end{proof}
\begin{rem}
Lemma \ref{lem:DistEndPO} may fail on the whole torus $\mathbb{T}^{2}$;
it is certainly possible for a nonperiodic orbit to approach the same
periodic orbit $C$ in both tails. However, if one cuts $\mathbb{T}^{2}$
along $C$ to obtain a compact annulus whose two boundary curves are
copies of $C$ (and they are considered as distinct periodic orbits),
Lemma \ref{lem:DistEndPO} is indeed applicable to this situation. 
\end{rem}
\begin{prop}
\label{prop:OrbitSetClose}(i) Let $A\cong S^{1}\times[0,1]$ be an
embedded annulus in $\mathbb{T}^{2}$ whose two boundary curves are
compact leaves of ${\cal F}_{W}$. Let ${\cal P}_{A}:=\{x\in A:{\cal O}_{x}\text{ is periodic}\}$.
Then ${\cal P}_{A}$ is a closed subset of $A$.

\vspace{2mm}\noindent (ii) Let ${\cal P}:=\{x\in\mathbb{T}^{2}:{\cal O}_{x}\text{ is periodic}\}$.
Then ${\cal P}$ is closed. 
\end{prop}
\begin{proof}
(i) Let $x_{n}\in{\cal P}_{A}$ and $x_{n}\rightarrow x$. Suppose
on the contrary that ${\cal O}_{x}$ is nonperiodic. According to
Lemma \ref{lem:DistEndPO}, 
\begin{equation}
C_{-}:=\alpha(x),\quad C_{+}:=\omega(x)\label{eq:POClosednessPf}
\end{equation}
are distinct periodic orbits on $A$. Recall from Proposition \ref{prop:POTop}
that both $C_{\pm}$ are essential, simple closed curves. As a result,
they bound a subannulus $A'$ and $x\in{\rm int}A'$. Since $x_{n}\rightarrow x,$
there is a large $n$ such that the periodic orbit ${\cal O}_{x_{n}}$
is contained inside ${\rm int}A'$, hence dividing ${\rm int}A'$
into two connected components. The point $x$ lies inside one of these
two components. In order for the orbit ${\cal O}_{x}$ to approach
$C_{\pm}$ in both tails (see (\ref{eq:POClosednessPf})), it must
cross the periodic orbit ${\cal O}_{x_{n}}$. Hence ${\cal O}_{x}\cap{\cal O}_{x_{n}}\neq\emptyset$,
which is a contradiction. Therefore, ${\cal O}_{x}$ is periodic.\begin{figure}[H] 
\begin{center}  
\includegraphics[scale=0.15]{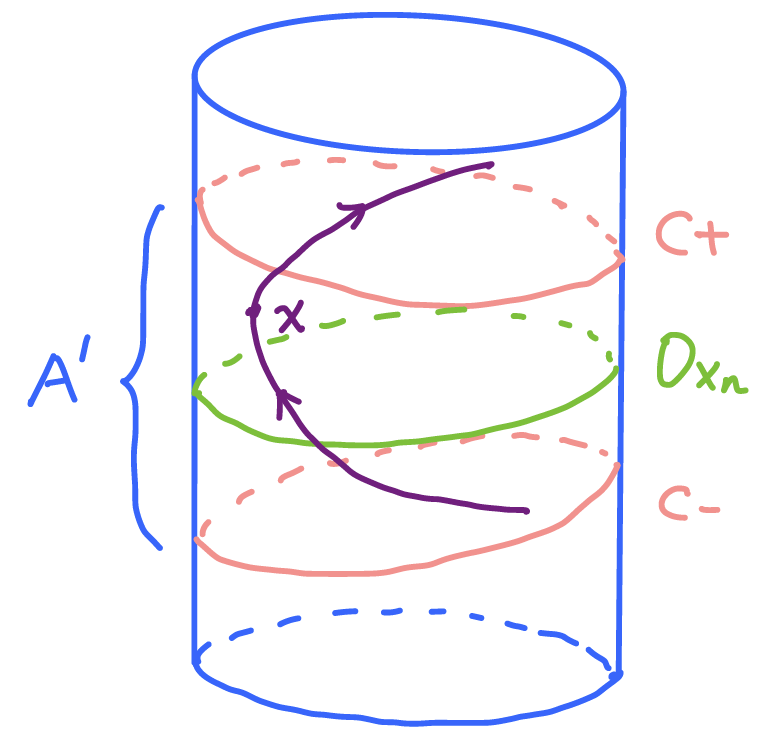} 
\end{center}
\end{figure}\noindent (ii) Let $x_{n}\in{\cal P}$ and $x_{n}\rightarrow x$.
Suppose that $x\notin{\cal O}_{x_{1}}$; otherwise, ${\cal O}_{x}$
is periodic and there is nothing to prove. Cutting $\mathbb{T}^{2}$
along ${\cal O}_{x_{1}}$ yields a compact annulus $A$ with $x_{n}\in{\cal P}_{A}$
and $x_{n}\rightarrow x\in{\rm int}A.$ The result of Part (i) then
implies that ${\cal O}_{x}$ is periodic. 
\end{proof}

\subsection{Suspension and Reeb component}

In this section, we define the concepts of a suspension and a Reeb
component, which are the two basic regimes appearing in the main classification
theorem. 

\subsubsection{Suspension}

We first define the notion of a suspension. Let $S^{1}:=\mathbb{R}/\mathbb{Z}$
and let $f:S^{1}\rightarrow S^{1}$ be an orientation-preserving diffeomorphism.
We define an equivalence relation on $S^{1}\times\mathbb{R}$ by identifying
\begin{equation}
(x,t)\sim(f^{k}(x),t-k)\label{eq:SuspEquiv}
\end{equation}
for any $(x,t)\in S^{1}\times\mathbb{R}$ and $k\in\mathbb{Z}$. Let
$\Sigma_{f}:=(S^{1}\times\mathbb{R})/_{f}\mathbb{Z}$ be the quotient
space under the above identification, whose quotient map is denoted
as $\pi_{f}:S^{1}\times\mathbb{R}\rightarrow\Sigma_{f}.$ It is clear
that $\Sigma_{f}$ is diffeomorphic to the torus $\mathbb{T}^{2}$.
The image of the vertical leaves $\{x\}\times\mathbb{R}$ under $\pi_{f}$,
i.e.
\[
{\cal F}_{f}:=\big\{\pi_{f}\big(\{x\}\times\mathbb{R}\big):x\in S^{1}\big\},
\]
defines a 1D foliation on $\Sigma_{f}$. This foliation is induced
by the vector field 
\[
\partial_{t}\big|_{[(x,t)]}:=\frac{d}{d\varepsilon}\Big|_{\varepsilon=0}[(x,t+\varepsilon)],\quad[(x,t)]\in\Sigma_{f}.
\]

\begin{defn}
\label{def:Suspension}The foliation $(\Sigma_{f},{\cal F}_{f})$
is called the \textit{suspension} of the diffeomorphism $f$. The
vector field $\partial_{t}$ is called the \textit{suspension field}. 
\end{defn}
\begin{figure}[H] 
\begin{center}  
\includegraphics[scale=0.16]{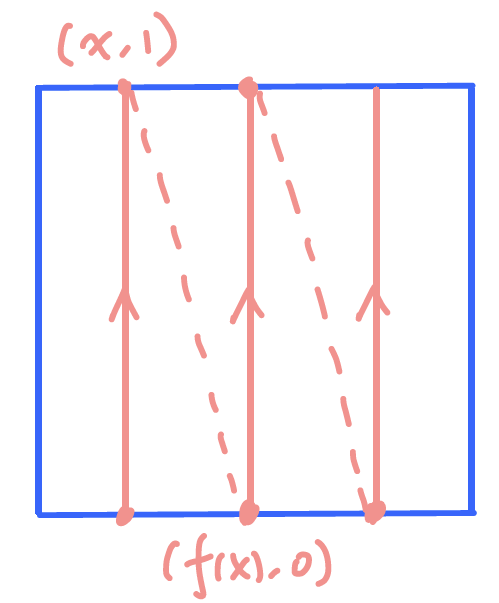} 
\end{center}
\end{figure}
\begin{example}
(i) Let $f={\rm id}_{S^{1}}$. The orbits of $\partial_{t}$ are the
vertical circles and they are all periodic with period one. 

\vspace{2mm}\noindent (ii) (Rational rotation) Let $f(x):=x+p/q$
$({\rm mod\ \mathbb{Z}})$ where $p,q$ are coprime integers ($q>0$).
Then $f^{q}={\rm id}_{S^{1}}$. All orbits are periodic with period
$q$.

\vspace{2mm}\noindent (iii) (Irrational rotation) Let $f(x):=x+\alpha$
$({\rm mod\ \mathbb{Z}})$ where $\alpha\notin\mathbb{Q}$. Every
orbit is diffeomorphic to $\mathbb{R}$ and is dense in $\mathbb{T}^{2}$. 

\vspace{2mm}\noindent (iv) Suppose that $f$ has precisely one fixed
point $[a]\in S^{1}$. Then $f$ can be viewed as a diffeomorphism
on $(a,a+1)$ with no fixed points. Suppose further that $f(x)>x$
for all $x\in(a,a+1)$. Then there is precisely one periodic orbit
\[
C:t\mapsto[(a,t)]
\]
which has period one. Every orbit other than $C$ is nonperiodic (hence
diffeomorphic to $\mathbb{R}$) and approaches $C$ in both ends.
 \begin{figure}[H] 
\begin{center}  
\includegraphics[scale=0.14]{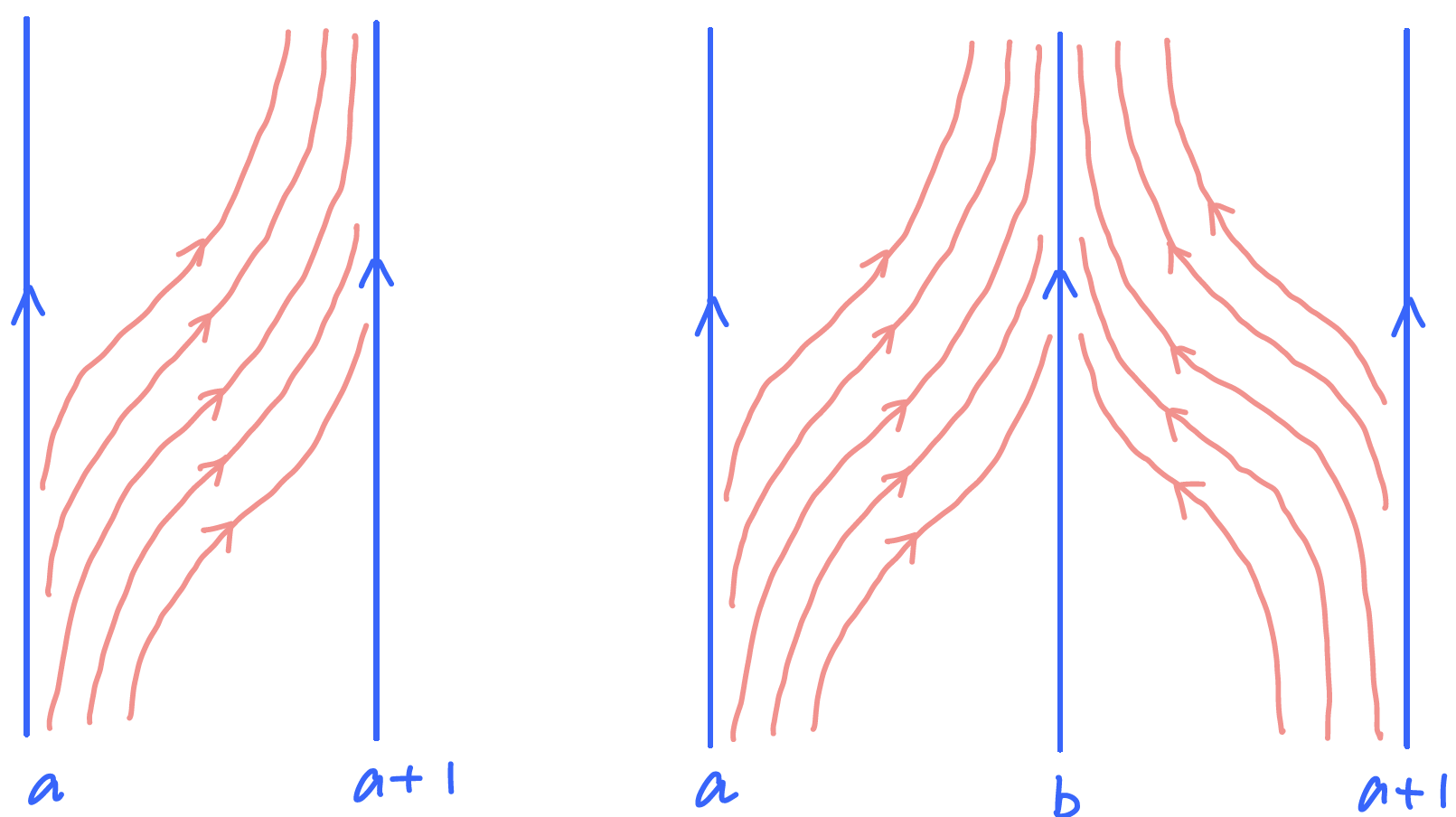}
\protect\caption{Suspension examples.}\label{fig:susp_exam}
\end{center}
\end{figure}\noindent (v) Suppose that $f$ has precisely two fixed points $[a],[b]\in S^{1}.$
With suitable choice of representatives, let us assume that $b\in(a,a+1)$.
Suppose that $f(x)>x$ on $(a,b)$ and $f(x)<x$ on $(b,a+1)$. Then
there are precisely two periodic orbits 
\[
C_{a}:t\mapsto[(a,t)],\quad C_{b}:t\mapsto[(b,t)].
\]
All other orbits are nonperiodic. Figure \ref{fig:susp_exam} sketches
the flow lines on the universal cover $\mathbb{R}^{2}$ for Cases
(iv) and (v).\textcolor{blue}{{} }
\end{example}
\begin{rem}
\label{rem:PeriodicPt}Periodic orbits of $(\Sigma_{f},{\cal F}_{f})$
correspond to periodic points of $f$. A point $x\in S^{1}$ is a
\textit{periodic point} of $f$ if there exists a positive integer
$q$ such that $f^{q}(x)=x$. The smallest such $q$ is called the
\textit{period} of $x.$ For an orientation-preserving diffeomorphism
$f\in{\rm Diff}(S^{1})$, the period of a periodic point is intrinsic
to $f$ and does not depend on the actual point $x$. 
\end{rem}
Let $W$ be a smooth, nonvanishing vector field $W$ on $\mathbb{T}^{2}.$
We give a simple criterion under which the foliation ${\cal F}_{W}$
is conjugate to a suspension (see Definition \ref{def:Conj}). 
\begin{defn}
A smooth, simple closed curve $\Sigma$ is a \textit{complete transversal}
if it intersects every orbit of $W$ transversally. 
\end{defn}
\begin{prop}
\label{prop:SuspChara}The foliation ${\cal F}_{W}$ is conjugate
to a suspension if and only if there exists a complete transversal. 
\end{prop}
\begin{proof}
(i) Necessity is obvious. In fact, the horizontal circle 
\[
S^{1}\ni t\mapsto[(t,0)]\in\Sigma_{f}
\]
is a complete transversal of the suspension foliation. It remains
to prove sufficiency. 

\vspace{2mm}\noindent (ii) Let $\Sigma$ be a given complete transversal.
For each $x\in\Sigma,$ we define the first return time
\[
\tau(x):=\inf\{t>0:\phi_{t}(x)\in\Sigma\}.
\]
We first show that $\tau(x)<\infty$. Suppose on the contrary that
$\phi_{t}(x)$ never meets $\Sigma$ again at any positive time. We
claim that there exists $T_{0}>0$ and a small neighbourhood $N$
of $\Sigma,$ such that $\phi_{t}(x)\notin N$ for all $t>T_{0}$.
Indeed, if this were not the case, there would exist a sequence of
times $t_{n}\rightarrow\infty$ such that $\phi_{t_{n}}(x)$ converges
to some $y\in\Sigma$. Draw a flow box $(V;s,r)$ centered at $y=(0,0)$
so that $\Sigma\cap V=\{s=0\}$. One can write $\phi_{t_{n}}(x)=(s_{n},r_{n})$
under the flow box coordinates and note that $(s_{n},r_{n})\rightarrow(0,0)$.
Since the flow lines are precisely $s\mapsto(s,r)$ inside $V,$ one
has 
\[
\phi_{-s_{n}}\big(\phi_{t_{n}}(x)\big)=\phi_{t_{n}-s_{n}}(x)=(0,r_{n})\in\Sigma
\]
for all large $n$. As a result, the time $t_{n}-s_{n}$ is a positive
return time for large $n$ (because $t_{n}\rightarrow\infty$ and
$s_{n}\rightarrow0$) which leads to a contradiction. Therefore, $\{\phi_{t}(x):t>0\}$
is contained outside a tubular neighbourhood $N$ of $\Sigma$ after
some short period of time. Similar to the proof of Proposition \ref{prop:POTop}
(i), it can be shown that any complete transversal must be essential.
As a result, $N^{c}$ (as a compact annulus) is contained in a planar
domain. By applying the Poincar\'e-Bendixson theorem, one obtains
a periodic orbit that is disjoint from $\Sigma.$ This contradicts
the completeness of $\Sigma.$ Consequently, $\tau(x)<\infty$.

We define the first return map
\begin{equation}
Q:\Sigma\rightarrow\Sigma,\quad Q(x):=\phi_{\tau(x)}(x).\label{eq:FirstRetMapQ}
\end{equation}
By realising a neighbourhood of $Q(x)$ on $\Sigma$ as a level set
$\{\rho=0\}$ for some smooth function $\rho,$ one can apply the
same argument as in the proof of Lemma \ref{lem:PoinMap} (construction
of the Poincar\'e map) to conclude that $\tau(x)$ is a smooth in
$x$. Therefore, $Q\in C^{\infty}(\Sigma)$. It is indeed a diffeomorphism
since $Q^{-1}$ is defined in the same way by the return-to-$\Sigma$
map of the inverse flow $-W$. Let $\sigma:S^{1}\stackrel{\cong}{\rightarrow}\Sigma$
be an $S^{1}$-parametrisation of $\Sigma.$ Then $f:=\sigma^{1}\circ Q\circ\sigma\in{\rm Diff}(S^{1})$.
\begin{figure}[H] 
\begin{center}  
\includegraphics[scale=0.12]{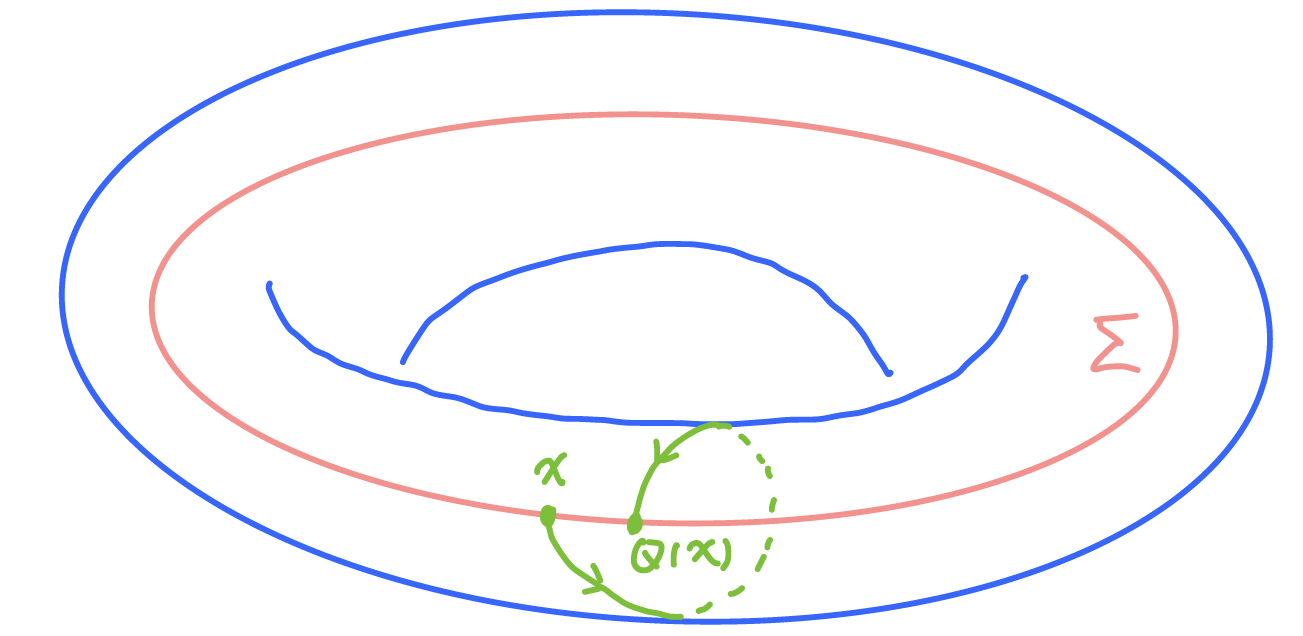} 
\end{center}
\end{figure}\noindent (iii) We claim that $f$ is orientation preserving, or
more precisely, $f'(r)>0$ for all $r$ (here we identify $S^{1}:=\mathbb{R}/\mathbb{Z}$
so that one has the canonical isomorphisms $T_{r}S^{1}\cong T_{f(r)}S^{1}\cong\mathbb{R}$).
The argument is in spirit similar to the proof of Proposition \ref{prop:CharaNum}
(positivity of the transversal characteristic number).

To this end, let $\Omega$ be a volume form on $\mathbb{T}^{2}$ such
that 
\[
g(r):=\Omega\big(W(\sigma(r)),\dot{\sigma}(r)\big)>0
\]
for all $r\in S^{1}$. This is possible since $\mathbb{T}^{2}$ is
orientable and $\Sigma$ is transversal to $W$. Let 
\begin{equation}
F(s,r):=\phi_{s\tau(\sigma(r))}(\sigma(r))\label{eq:Fsr}
\end{equation}
and define 
\[
J(s,r):=\Omega\big(W(F(s,r)),\partial_{r}F(s,r)\big).
\]
It is plain to see that 
\begin{align}
\partial_{s}F(s,r) & =\tau(\sigma(r))W(F(s,r)),\nonumber \\
\partial_{r}F(s,r) & =(d\phi_{s\tau(\sigma(r))})\dot{\sigma}(r)+s\tau'(r)W(F(s,r)).\label{eq:DrF}
\end{align}
Since $\dot{\sigma}(r)$ and $W(\sigma(r))$ are linearly independent,
so are 
\[
(d\phi_{s\tau(\sigma(r))})\dot{\sigma}(r)\quad\text{and}\quad(d\phi_{s\tau(\sigma(r))})W(\sigma(r))=W(F(s,r)).
\]
As a result, 
\[
J(s,r)=\Omega\big(W(F(s,r)),(d\phi_{s\tau(\sigma(r))})\dot{\sigma}(r)\big)
\]
has constant sign. Note that $J(0,r)=g(r)>0$. It follows that 
\[
J(1,r)=\Omega\big(W(\sigma\circ f(r)),f'(r)\dot{\sigma}(f(r))\big)=f'(r)g(f(r))>0.
\]
Therefore, $f'(r)>0$ and $f$ is orientation preserving.

Now let $\Sigma_{f}$ be the suspension of $f$. Define $\Psi:\Sigma_{f}\rightarrow\mathbb{T}^{2}$
by 
\[
\Psi([r,t]):=\phi_{t\tau(\sigma(r))}(\sigma(r)),\quad r\in S^{1},t\in[0,1].
\]
The definition of $\Psi$ is consistent at the junction $t=0=1$ (${\rm mod}\ \mathbb{Z}$).
In addition, $\Psi$ is a homeomorphism and maps the suspension foliation
to ${\cal F}_{W}$. The issue is that $\Psi$ may fail to be a diffeomorphism.
This is because 
\[
\frac{\partial}{\partial t}\Big|_{t=1^{-}}\phi_{t\tau(\sigma(r))}(\sigma(r))\quad\text{and}\quad\frac{\partial}{\partial t}\Big|_{t=0^{+}}\phi_{t\tau(\sigma\circ f(r))}(\sigma\circ f(r))
\]
need not coincide. They coincide precisely when $\tau(\sigma(r))=\tau(\sigma\circ f(r))$,
which needs not be true in general.

The above issue can be fixed by a simple leaf reparametrisation. Let
\[
\theta(r,t):=ct+(\tau(\sigma(r))-c)\beta(t),\quad r\in S^{1},t\in[0,1].
\]
Here 
\[
0<c<\inf_{r\in S^{1}}\tau(\sigma(r))
\]
and $\beta:[0,1]\rightarrow[0,1]$ is smooth, nondecreasing and satisfies
\[
\beta=0\text{ near }0,\quad\beta=1\text{ near }1.
\]
Then one has 
\[
\theta(r,0)=0,\quad\theta(r,1)=\tau(\sigma(r)),
\]
\[
\theta(r,t)=ct\text{ near }t=0,
\]
\[
\theta(r,t)=\tau(\sigma(r))+c(t-1)\text{ near }t=1.
\]
This ensures that all $t$-derivatives match at the junction $t=0=1$
(${\rm mod}\ \mathbb{Z}$). Now we adjust the definition of $\Psi$
to be 
\[
\hat{\Psi}:\Sigma_{f}\rightarrow\mathbb{T}^{2},\quad\hat{\Psi}([r,t]):=\phi_{\theta(r,t)}(\sigma(r)).
\]
This new map $\hat{\Psi}$ is now easily seen to be a diffeomorphism,
which provides a smooth conjugation between $\Sigma_{f}$ and $\mathbb{T}^{2}$.
\end{proof}

\subsubsection{Reeb component}

Next, we define the notion of a Reeb component. 
\begin{defn}
We say that ${\cal F}_{W}$ \textit{contains a Reeb component} if
there exists an embedded annulus $A\cong S^{1}\times[0,1]$ inside
$\mathbb{T}^{2}$ such that the following two properties are satisfied:

\vspace{2mm}\noindent (i) $\partial A$ is the union of two compact
leaves that are in opposite homology classes (under the $W$-orientation);

\vspace{2mm}\noindent (ii) there are no compact leaves of ${\cal F}_{W}$
in the interior of $A$.
\end{defn}
\begin{figure}[H] 
\begin{center}  
\includegraphics[scale=0.14]{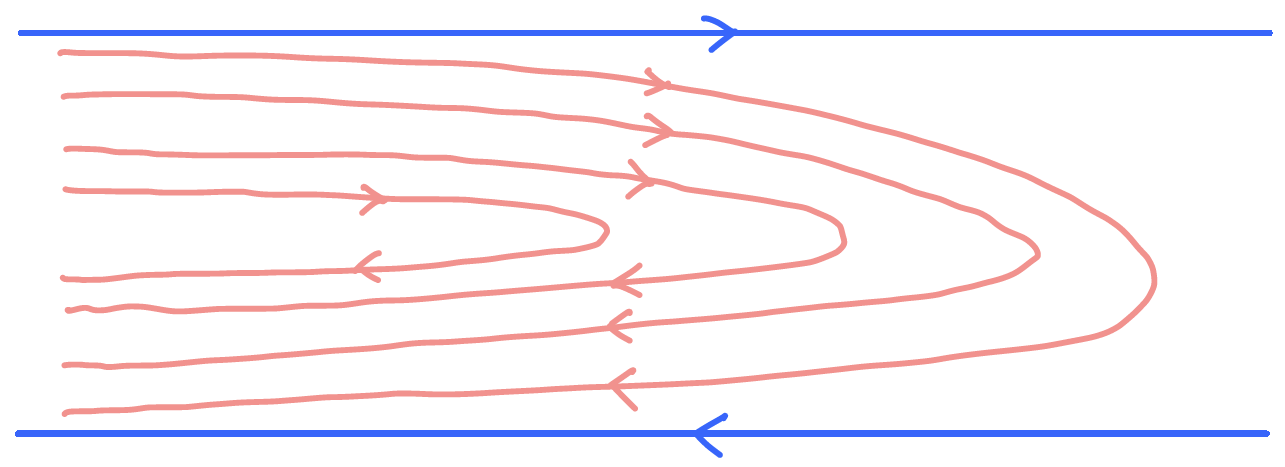} 
\end{center}
\end{figure}

The following proposition provides a simple criterion for the existence
of a Reeb component.
\begin{prop}
\label{prop:ReebCrit}The foliation ${\cal F}_{W}$ has a Reeb component
if and only if there exist two periodic orbits with opposite homology
classes. 
\end{prop}
\begin{proof}
Necessity follows immediate from the definition and Proposition \ref{prop:POTop}.
We now prove sufficiency. Let $C_{+},C_{-}$ be two periodic orbits
with opposite homology classes. Let $A$ be a compact annulus bounded
by $C_{\pm}.$ Define ${\cal P}_{A}^{\pm}$ to be the set of points
in $A$ whose orbit is homologous to $C_{\pm}.$ According to Proposition
\ref{prop:OrbitSetClose}, it is easily seen that both ${\cal P}_{A}^{\pm}$
are closed subsets of $A$ (a sequence of periodic orbits with the
same homology class cannot converge to a periodic orbit with an opposite
direction). Therefore, 
\[
\inf_{x\in{\cal P}_{A}^{+},y\in{\cal P}_{A}^{-}}d(x,y)>0.
\]
Let $x_{*}\in{\cal P}_{A}^{+}$ and $y_{*}\in{\cal P}_{A}^{-}$ be
two points achieving the above infimum. It follows that there are
no periodic orbits between ${\cal O}_{x_{*}}$ and ${\cal O}_{y_{*}}$
inside $A$. As a consequence, they bound a Reeb component. 
\end{proof}

\subsubsection{Mutual exclusiveness of the two regimes}

We now show that suspension and Reeb component are mutually exclusive. 
\begin{prop}
Suppose ${\cal F}_{W}$ is conjugate to a suspension. Then it cannot
contain a Reeb component.
\end{prop}
\begin{proof}
Let $\Psi:{\cal F}_{W}\rightarrow\Sigma_{f}$ be a conjugation to
a suspension 
\[
\Sigma_{f}=\big\{[x,t]:x\in S^{1},t\in\mathbb{R},(x,t)\sim(f^{k}(x),t-k)\big\}.
\]
Let $\eta:=\Psi^{*}dt,$ where $dt$ is the vertical form on $\Sigma_{f}$
defined by
\[
\langle dt,\partial_{x}\rangle=0,\quad\langle dt,\partial_{t}\rangle=1.
\]
The image of a periodic orbit $\gamma$ of $W$ under $\Psi$ is a
reparametrisation of a periodic orbit $\beta$ of the suspension field
$\partial_{t}$. The period of $\beta$ is a positive integer $q$,
whose value is uniquely determined by $f$ and is dependent of the
choice of $\gamma$ (cf. Remark \ref{rem:PeriodicPt}). It follows
that 
\[
\int_{\gamma}\eta=\int_{\beta}dt=q.
\]
As a consequence, having two periodic orbits with opposite homology
classes is not possible in this case. 
\end{proof}

\subsection{The main classification theorem}

In this section, we state and prove the main classification theorem
for nondegenerate torus flows. 
\begin{thm}
\label{thm:Classification}Let ${\cal F}_{W}$ be the 1D foliation
associated with a smooth,  nonvanishing vector field $W$ on $\mathbb{T}^{2}.$
Then precisely one of the following two mutually exclusive possibilities
holds:

\vspace{2mm}\noindent (i) ${\cal F}_{W}$ is conjugate to a suspension;

\vspace{2mm}\noindent (ii) ${\cal F}_{W}$ contains a Reeb component. 
\end{thm}
In what follows, we develop the proof of Theorem \ref{thm:Classification}.
We will consider two cases separately according to whether $W$ has
periodic orbits. 

\subsubsection{\label{subsec:NoPO}Case I: No periodic orbits}

We first assume that $W$ does not have periodic orbits (in particular,
there are no Reeb components). Our aim is to prove that ${\cal F}_{W}$
is conjugate to a suspension in this case. According to Proposition
\ref{prop:SuspChara}, it is equivalent to showing that there exists
a complete transversal. 
\begin{lem}
\label{lem:ExistTran}Suppose that there exists a nonperiodic orbit.
Then there exists a smooth, simple closed curve $\Sigma$ that is
everywhere transversal to $W$. 
\end{lem}
\begin{proof}
Let $\gamma$ be a nonperiodic orbit. There exists $t_{n}\rightarrow\infty$
such that $\gamma_{t_{n}}$ converges to some $y$. By drawing a flow
box $V$ around $y$ and adjusting $t_{n}$ if necessary (cf. Proof
of Lemma \ref{lem:PRPerOrb}), one can find $s_{0}<s_{1}$ such that
$z_{i}=\gamma_{s_{i}}\in\Sigma$ ($i=0,1$) and $\gamma_{t}\notin\Sigma$
for all $t\in(s_{0},s_{1})$. Here $\Sigma$ denotes the vertical
transversal in $V$ passing through $y$. Note that $z_{0}\neq z_{1}.$

Let us fix a parametrisation $\sigma:(-c,c)\rightarrow\Sigma$ so
that $z_{0}=\sigma(0)$ and $z_{1}=\sigma(d)$ for some $d\in(0,c)$.
Define the first return map $f:(-\delta,\delta)\rightarrow(-c,c)$
by the relation 
\[
\phi_{\tau(r)}(\sigma(r))=\sigma(f(r)),\quad r\in(-\delta,\delta),
\]
where $\delta$ is a sufficiently small number and $\tau(r)$ is the
first return time of $\phi_{t}(\sigma(r))$ to $\Sigma$ (cf. (\ref{eq:FirstRetMapQ})).
Note that $f(0)=d$. Consider the function
\[
F(s,r):=\phi_{s\tau(r)}(\sigma(r))
\]
where $s\in[0,1]$ and $r\in(-\delta,\delta)$ (cf. (\ref{eq:Fsr})).
It is clear that 
\[
F(0,r)=\sigma(r),\quad F(1,r)=\sigma(f(r)).
\]
By shrinking $\delta$ if necessary, the map 
\[
F:[0,1]\times(-\delta,\delta)\rightarrow\mathbb{T}^{2}
\]
is injective and 
\[
F\big((0,1)\times(-\delta,\delta)\big)\cap\Sigma=\emptyset.
\]
Similar to the proof of Proposition \ref{prop:SuspChara}, one can
show that $f'(r)>0$ for all $r\in(-\delta,\delta)$. As a result,
there exists $\eta\in(-\delta,0)$ such that $f(\eta)>0$. 

Fix a function $\rho:[0,1]\rightarrow[\eta,0]$ such that $\rho(0)=0$,
$\rho(1)=\eta$ and $\rho'(s)<0$ for all $s\in[0,1]$. Define the
curve 
\[
\alpha_{s}:=F(s,\rho(s)),\quad0\leqslant s\leqslant1.
\]
Note that $\alpha$ is a simple curve that joins $z_{0}$ to $\sigma(f(\eta))$
outside $\Sigma$. By definition, one has 
\[
\dot{\alpha}_{s}=\partial_{s}F(s,\rho(s))+\rho'(s)\partial_{r}F(s,\rho(s)).
\]
The vector $\partial_{s}F(s,\rho(s))\in{\rm Span}\{W(\alpha_{s})\}$.
In addition, since $\rho'(s)<0,$ the vector $\partial_{r}F(s,\rho(s))$
is linearly independent of $\partial_{s}F(s,\rho(s))$ (cf. (\ref{eq:DrF})).
As a consequence, $\alpha$ is everywhere transversal to $W$. Let
$\beta:[1,2]\rightarrow\Sigma$ be the arc joining $\sigma(f(\eta))$
back to $z_{0}$ along $\Sigma$. Apparently, $\beta$ is also everywhere
transversal to $W$. The concatenated curve 
\[
\zeta_{s}:=\begin{cases}
\alpha_{s}, & 0\leqslant s\leqslant1;\\
\beta_{s}, & 1\leqslant s\leqslant2
\end{cases}
\]
is now a simple closed curve, which is smooth and everywhere transversal
to $W$ except at the two corners $z_{0}$ and $\sigma(f(\eta))$.
\begin{figure}[H] 
\begin{center}  
\includegraphics[scale=0.14]{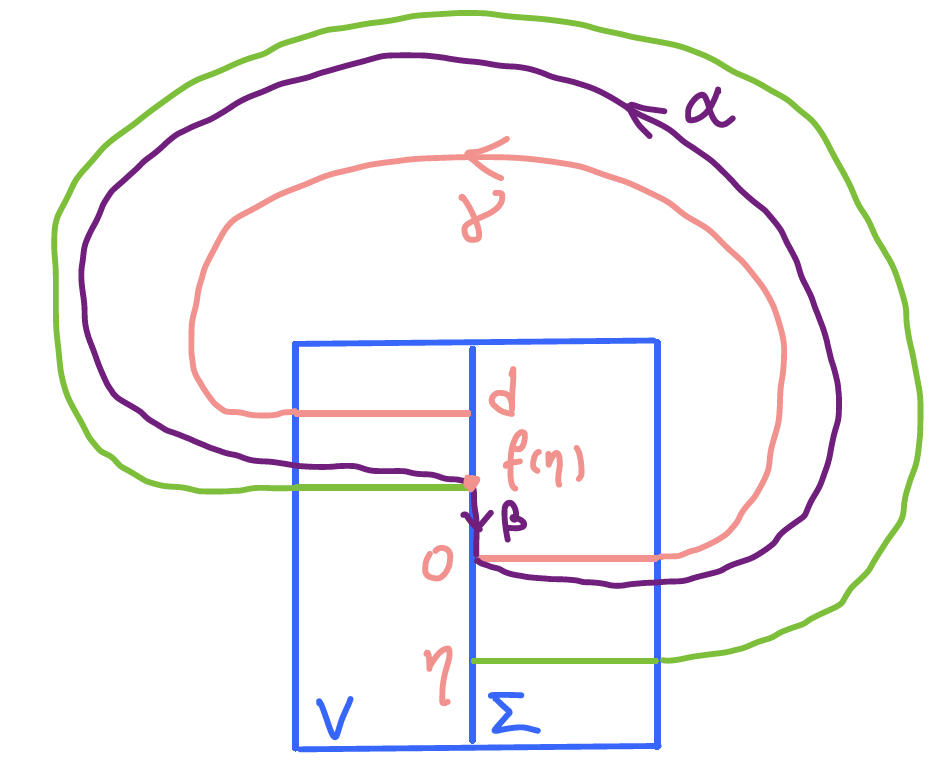} 
\end{center}
\end{figure}

It remains to see that the two corners can be smoothed locally without
losing the transversal property. Let $\Omega$ be a volume form on
$\mathbb{T}^{2}$ such that $\Omega(W(z_{0}),\dot{\sigma}(0))>0$.
The function 
\[
\psi(s):=\Omega\big(W(\zeta_{s}),\dot{\zeta}_{s}\big)
\]
has constant sign on $(0,1)$ and $(1,2)$. We claim that 
\begin{equation}
{\rm sgn}\psi(1-)={\rm sgn}\psi(1+)={\rm sgn}\psi(2-)={\rm sgn}\psi(0+).\label{eq:SignConsist}
\end{equation}
Indeed, since $\rho'(0)<0$ one has 
\begin{align*}
{\rm sgn}\psi(1-) & ={\rm sgn}\psi(0+)={\rm sgn}\Omega\big(W(z_{0}),\rho'(0)\dot{\sigma}(0)\big)\\
 & =-{\rm sgn}\Omega\big(W(z_{0}),\dot{\sigma}(0)\big).
\end{align*}
In addition, by the definition of $\beta$ one also has 
\[
{\rm sgn}\psi(1+)={\rm sgn}(2-)=-{\rm sgn}\Omega\big(W(z_{0}),\dot{\sigma}(0)\big)
\]
As a consequence, the relation (\ref{eq:SignConsist}) holds. This
relation shows that at any of the two corners $s_{*}=0,1$, both vectors
$\dot{\zeta}_{s_{*}-}$ and $\dot{\zeta}_{s_{*}+}$ ($0-$ is the
same as $2-$) lie in the cone 
\[
C_{s_{*}}:=\big\{ v:\Omega\big(W(\zeta_{s_{*}}),v\big)<0\big\}.
\]
It is clearly possible to smooth the corner $\zeta_{s_{*}}$ locally
in a way such that the resulting new smooth arc $\tilde{\zeta}_{s}$
still satisfies $\Omega(W(\tilde{\zeta}_{s}),v)<0$ when $s$ is near
$s_{*}$. One thus obtains a smooth, simple closed curve which is
everywhere transversal to $W$. 
\end{proof}
\begin{proof}[Proof of Theorem \ref{thm:Classification}: No periodic orbits case]

Suppose that $W$ has no periodic orbits. According to Lemma \ref{lem:ExistTran},
there exists a smooth, simple closed transversal curve $\Sigma$.
If $\Sigma$ does not intersect some orbit $\gamma,$ the tails of
$\gamma$ will be trapped outside a tubular neighbourhood of $\Sigma.$
The Poincar\'e-Bendixson theorem then produces a periodic orbit,
which leads to a contradiction (cf. Part (ii) in the proof of Proposition
\ref{prop:SuspChara}). Therefore, $\Sigma$ is complete. It follows
from Proposition \ref{prop:SuspChara} that ${\cal F}_{W}$ is conjugate
to a suspension. 

\end{proof}

\subsubsection{Case II: At least one periodic orbit}

Now we assume that $W$ has at least one periodic orbit. In this case,
we aim at showing that 
\[
{\cal F}_{W}\text{ has no Reeb component}\implies{\cal F}_{W}\text{ is conjugate to a suspension.}
\]
The key lemma for this part is stated as follows. 
\begin{lem}
\label{lem:NoReebTran}Let $A\cong S^{1}\times[0,1]$ be an embedded
annulus whose boundary consists of two compact leaves $C_{0}$ and
$C_{1}$. Suppose that $A$ does not contain any Reeb component. Then
there exists a smooth, simple curve $c:[0,1]\rightarrow A$ such that
$c_{0}\in C_{0},$ $c_{1}\in C_{1}$ and $c$ is everywhere transversal
to $W$ on $A$.
\end{lem}
\begin{proof}
First of all, the periodic orbits $C_{0},C_{1}$ must have the same
homology class; otherwise, there would be a Reeb component inside
$A$ according to Proposition \ref{prop:ReebCrit}, which contradicts
the assumption. Let $\Omega$ be a volume form on $A$ such that $\Omega(W(x),v)>0$
when $x\in C_{0}$ and $v\in T_{x}\mathbb{T}^{2}$ is a transversal
vector pointing into $A$. A vector $v\in T_{y}\mathbb{T}^{2}$ is
said to be \textit{positively transversal} to $W$ at $y\in A$ if
$\Omega(W(y),v)>0$. A smooth curve $c_{t}$ in $A$ is \textit{positively
transversal} to $W$ if $\Omega(W(c_{t}),\dot{c}_{t})>0$ for all
$t.$ Let us define 
\begin{align*}
{\cal R}:= & \big\{ y\in A:\exists\text{ a smooth, simple, positively transversal }\\
 & \quad\quad\quad c:[0,1]\rightarrow A\text{ s.t. }c_{0}\in C_{0},c_{1}=y\big\}.
\end{align*}
Our aim is to show that ${\cal R}\cap C_{1}\neq\emptyset$. We divide
the argument into several steps.

\vspace{2mm}\noindent (i) Let $N_{1}$ be an open half-tubular neighbourhood
of $C_{1}$ in $A$ (a lower collar). We claim that ${\cal R}\cap N_{1}\neq\emptyset$.
Suppose this is true and let $y\in{\cal R}\cap N_{1}$. By definition,
there is a simple, positively transversal curve $c$ from $C_{0}$
to $y$. By travelling along the fiber in $N_{1}$, one can easily
continue $c$ to the boundary $C_{1}$. The resulting curve is smooth
and positively transversal to $W$ except at the junction. Similar
to the proof of Lemma \ref{lem:ExistTran}, one can smooth the corner
to obtain a smooth, simple, positively transversal curve from $C_{0}$
to $C_{1}$. Now it remains to show that ${\cal R}\cap N_{1}\neq\emptyset$.
In what follows, we assume on the contrary that ${\cal R}\subseteq A\backslash N_{1}.$

\vspace{2mm}\noindent (ii) We first show that ${\cal R}\cap{\rm int}A$
is an open and nonempty subset of $A$. Indeed, any point inside a
half-tubular neighbourhood of $C_{0}$ in $A$ (an upper collar) is
clearly reachable from $C_{0}$. As a result, ${\cal R}\cap{\rm int}A\neq\emptyset$.
Next, let $y\in{\cal R}\cap{\rm int}A$. Draw a flow box $(V;s,r)\subseteq{\rm int}A$
centered at $y=(0,0).$ We assume that the transversal coordinate
$r$ is chosen such that a curve $r_{t}$ being positively transversal
means $\dot{r}_{t}>0$. Let $c$ be a smooth, simple and positively
transversal arc from $C_{0}$ to $y$. Let $z\in V\cap c$. By positive
transversality, one must have $r(z)<0.$ It follows that all points
in 
\[
V':=\{x\in V:r(x)>r(z)\}
\]
is reachable from $z$ by a simple positive transversal. By smoothing
the corner if necessary, one concludes that $V'\subseteq{\cal R}$.
Therefore, ${\cal R}\cap{\rm int}A$ is an open subset of $A$. 

\vspace{2mm}\noindent (iii) Define $\partial_{A}{\cal R}:=\bar{{\cal R}}\backslash{\cal R}$.
Note that $\partial_{A}{\cal R}\neq\emptyset$ under the current assumption
that ${\cal R}\subseteq A\backslash N_{1}$, since ${\cal R}\cap{\rm int}A$
is open. We claim that $\partial_{A}{\cal R}$ is a single compact
leaf of ${\cal F}_{W}$. 

To this end, let $y\in\partial_{A}{\cal R}$. Draw a flow box $(V;s,r)$
centered at $y=(0,0),$ whose transversal coordinate $r$ is oriented
in the same way as in Part (ii). The key observation is that 
\begin{equation}
V\cap{\cal R}=V\cap\{r>0\}.\label{eq:ReaSetFB}
\end{equation}
Let $z\in V\cap\{r>0\}$. Since $y\in\partial_{A}{\cal R}$ and ${\cal R}$
is open in ${\rm int}A$, there exists $y'\in V\cap{\cal R}$ such
that $0<r(y')<r(z)$. It follows that $z$ is reachable by a positive
transversal from $y'$. Upon smoothing the corner at $y'$, one concludes
that $z\in{\cal R}$. To prove the other direction, one observes that
$V\cap\{r<0\}$ is not reachable; otherwise, $y$ would be reachable
which is a contradiction. This also implies that $V\cap\{r=0\}$ is
not reachable, since the only way to reach $V\cap\{r=0\}$ is through
$V\cap\{r<0\}$. As a consequence, $V\cap\{r\leqslant0\}\subseteq{\cal R}^{c}$.
The relation (\ref{eq:ReaSetFB}) thus follows. An immediate corollary
of (\ref{eq:ReaSetFB}) is that $V\cap\partial_{A}{\cal R}=V\cap\{r=0\}.$
In particular, $\partial_{A}{\cal R}$ is a 1D submanifold of $A$. 

The next observation is that $\partial_{A}{\cal R}$ is a disjoint
union of leaves. Indeed, let $x\in\partial_{A}{\cal R}$ and define
\[
I_{x}:=\{t\in\mathbb{R}:\phi_{t}(x)\in\partial_{A}{\cal R}\}.
\]
Since $\partial_{A}{\cal R}$ is closed in $A$, it is obvious that
$I_{x}$ is closed. To prove that $I_{x}$ is open, let $y=\phi_{t}(x)\in\partial_{A}{\cal R}$.
By applying the local relation (\ref{eq:ReaSetFB}) to $y$, one sees
that $\phi_{s}(x)\in\partial_{A}{\cal R}$ for $s$ near $t$ (because
$V\cap\{r=0\}$ is a flow line). Therefore, $I_{x}$ is both open
and closed, hence being equal to $\mathbb{R}$. 

Since $\partial_{A}{\cal R}$ is a closed, embedded 1D submanifold
of the compact annulus $A$, every connected component of $\partial_{A}{\cal R}$
has to be compact. As a consequence, $\partial_{A}{\cal R}$ is a
disjoint union of compact leaves. Suppose that there are at least
two such leaves, say $C',C''$ (we assume that $C'$ is closer to
$C_{0}$). Since a small lower collar of $C''$ is reachable, any
such positive transversal would intersect $C'$. But $C'\subseteq\partial_{A}{\cal R}\subseteq{\cal R}^{c}$.
This leads to a contradiction. Therefore, $\partial_{A}{\cal R}$
has only one connected component, i.e. a compact leaf. 

\vspace{2mm}\noindent (iv) Let $\partial_{A}{\cal R}=C$ for some
compact leaf $C$ according to Part (iii). We claim that $C$ and
$C_{0}$ are in opposite homology classes. Suppose this is not true.
Due to our orientation convention, positively transversal vectors
at $y\in C$ are pointing towards the $C_{1}$-side. Let $N$ be a
tubular neighboorhood of $C.$ Note that $C$ divides $N$ into two
halves: $N_{+}$ (the component on the $C_{1}$-side) and $N_{-}$
(the component on the $C_{0}$-side). Let $z\in N_{+}$ and $y\in C=\partial_{A}{\cal R}$.
Then a sufficiently close-to-$y$ element $y'\in N_{-}$ is reachable
by a positive transveral from $C_{0}$. By continuing this transversal
towards $z$ inside $N$ in the obvious way (along fiber directions
towards $C_{1}$),one obtains a positive transversal that will intersect
$C$. This contradicts the fact that $C\subseteq{\cal R}^{c}$. Therefore,
$C$ and $C_{0}$ must have opposite directions. It follows from Proposition
\ref{prop:ReebCrit} that there is a Reeb component inside the annulus
bounded by $C_{0}$ and $C$. This contradicts the assumption of the
lemma. \begin{figure}[H] 
\begin{center}
\includegraphics[scale=0.11]{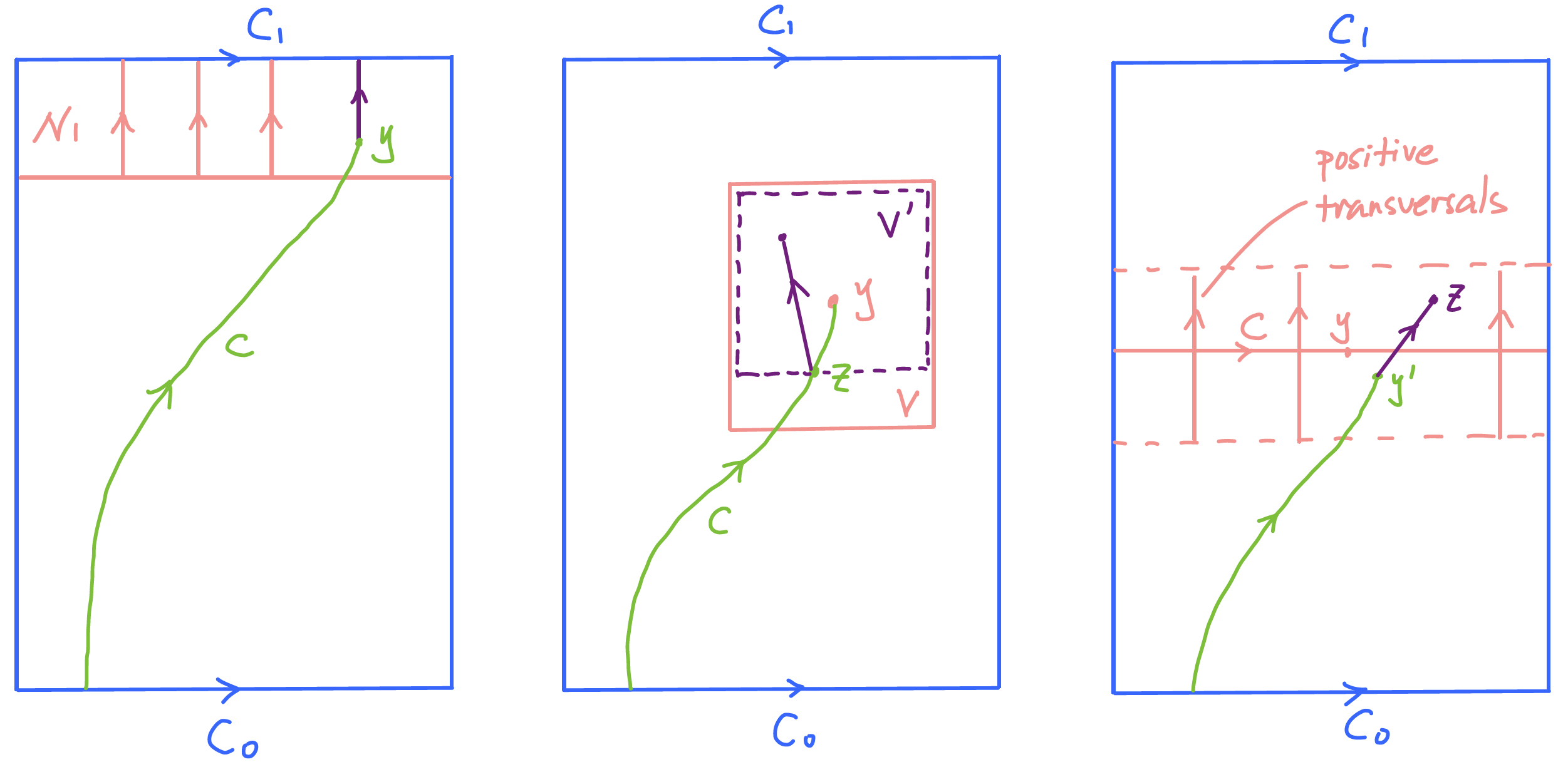} 
\protect\caption{The left, middle and right figures correspond to Steps (i), (ii) and (iv) respectively.} 
\end{center}
\end{figure}

To summarise, one arrives at the conclusion that the assumption ${\cal R}\subseteq A\backslash N_{1}$
cannot hold. In other words, one must have ${\cal R}\cap N_{1}\neq\emptyset.$
The proof of the lemma is now complete. 
\end{proof}
\begin{proof}[Proof of Theorem \ref{thm:Classification}: At least one periodic orbit case]

Suppose that ${\cal F}_{W}$ has no Reeb component and there is a
periodic orbit $C$ of $W$. Cutting the torus $\mathbb{T}^{2}$ along
$C$ produces a compact annulus $A$ whose boundary consists of two
compact leaves $C_{0}$ and $C_{1}$ (they are copies of $C$). Since
$A$ does not contain any Reeb component, according to Lemma \ref{lem:NoReebTran},
there exists a smooth, simple transversal $c:[0,1]\rightarrow A$
from $c_{0}\in C_{0}$ to $c_{1}\in C_{1}$. By gluing back the torus
along $C$, the curve $c$ becomes a smooth, simple transversal such
that $c_{0},c_{1}\in C$ and $c_{t}\notin C$ for all $t\in(0,1).$
Let $\beta$ be the arc joining $c_{1}$ to $c_{0}$ along $C$.

Similar to the proof of Lemma \ref{lem:ExistTran}, by suitably deforming
$\beta$ along transversal directions (e.g. using the fiber coordinate
in a tubular neighbourhood of $C$) and corner smoothing, one obtains
a smooth, simple closed transversal $\Sigma$ that intersects $C$
precisely once. In particular, $I(\Sigma,C)=\pm1$.

\begin{figure}[H] 
\begin{center}  
\includegraphics[scale=0.11]{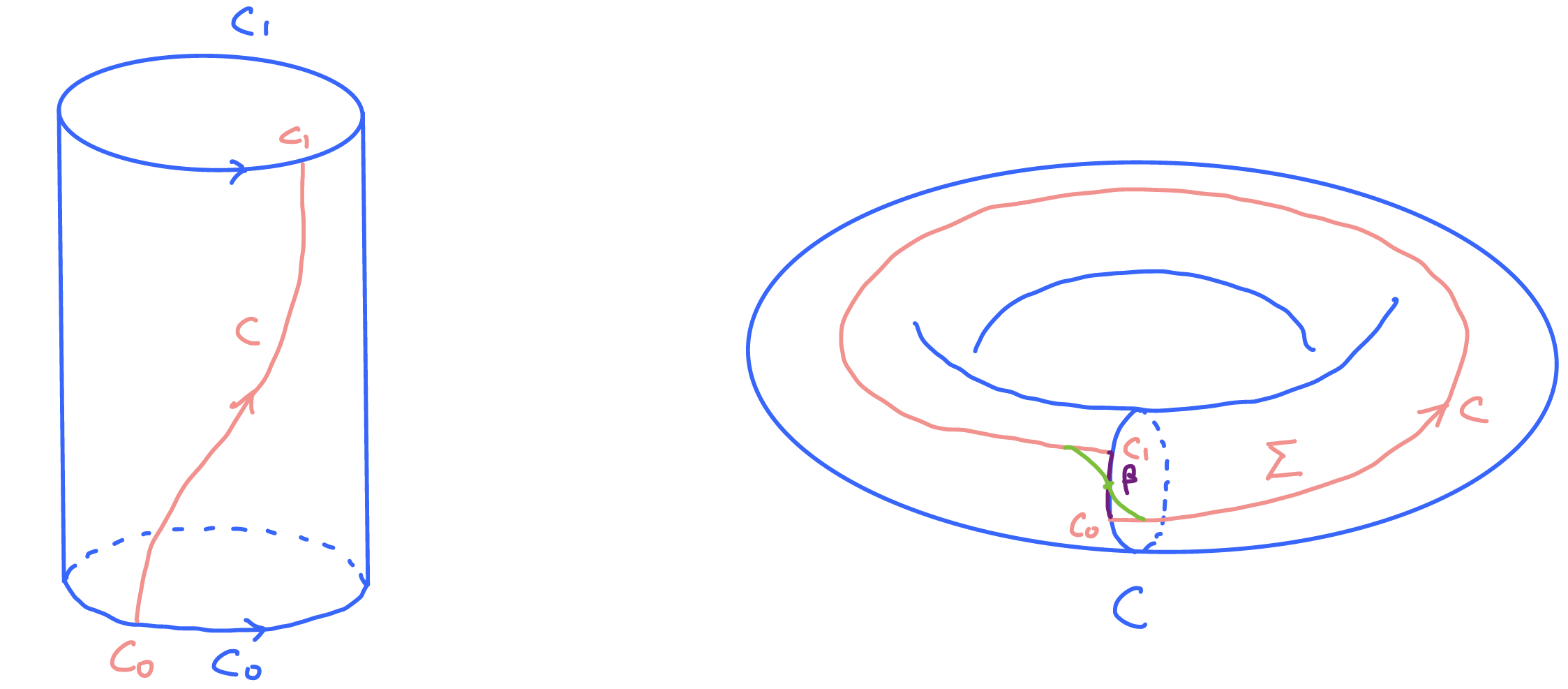} 
\end{center}
\end{figure}

Suppose that $\Sigma$ is not complete. The Poincar\'e-Bendixson
theorem then produces a periodic orbit $C'$ which does not intersect
$\Sigma.$ In particular, $I(\Sigma,C')=0$. But one knows from Proposition
\ref{prop:POTop} that $[C']=\pm[C].$ This leads to a contradiction.
Therefore, $\Sigma$ is a complete transversal. It follows from Proposition
\ref{prop:SuspChara} that ${\cal F}_{W}$ is conjugate to a suspension.

\end{proof}

\subsection{\label{sec:TorDecomp}The dynamical picture of nondegenerate torus
flow}

Let us summarise the dynamical picture of a nondegenerate torus flow.
Let $W$ be a smooth, nonvanishing vector field on $\mathbb{T}^{2}$.
Recall from Proposition \ref{prop:OrbitSetClose} that the set 
\[
{\cal P}:=\{x:{\cal O}_{x}\text{ is periodic}\}
\]
is closed. As a result, one can decompose $\mathbb{T}^{2}\backslash{\cal P}$
into an at most countable union of disjoint, connected, open subsets
$U_{\alpha}$:
\[
\mathbb{T}^{2}\backslash{\cal P}=\bigsqcup_{\alpha}U_{\alpha}.
\]

\begin{lem}
Suppose that ${\cal P}\neq\emptyset,\mathbb{T}^{2}.$ Then each $U_{\alpha}$
is an open annulus whose boundary consists of two periodic orbits
(which may coincide if ${\cal P}$ is a single periodic orbit). All
orbits inside $U_{\alpha}$ are nonperiodic.
\end{lem}
\begin{proof}
Fix a periodic orbit $C_{*}$ and cut the torus along $C_{*}$. This
produces a compact annulus $A\cong S^{1}\times[0,1]$ whose boundary
consists of two copies $C_{0},C_{1}$ of $C_{*}$ and whose interior
is diffeomorphic to $\mathbb{T}^{2}\backslash C$ (we view $C_{0},C_{1}$
as distinct periodic orbits). Let ${\cal C}$ be the collection of
periodic orbits on $A$. A given periodic orbit $C\in{\cal C}\backslash\{C_{0},C_{1}\}$
divides $A$ into two connected components: $A_{-}(C)$ (the one containing
$C_{0}$) and $A_{+}(C)$ (the one containing $C_{1}$). As a convention,
$A_{-}(C_{0})=\emptyset$ and $A_{+}(C_{0})=A\backslash C_{0}$ (respectively,
$A_{-}(C_{1})=A\backslash C_{1}$ and $A_{+}(C_{1})=\emptyset$).

Now let $U_{\alpha}$ be a connected component of $A\backslash{\cal P}.$
Given $x\in U_{\alpha}$, we define 
\begin{align*}
{\cal P}_{-}(x) & :=\big\{ C\in{\cal C}:x\in A_{+}(C)\big\},\\
{\cal P}_{+}(x) & :=\big\{ C\in{\cal C}:x\in A_{-}(C)\big\}
\end{align*}
respectively. Let 
\[
S_{-}:=\bigcup_{C\in{\cal P}_{-}(x)}C
\]
and define 
\[
d_{-}:=\inf_{y\in S_{-}}d(x,y).
\]
Note that $d_{-}\in(0,\infty)$ because $S_{-}$ is nonempty ($C_{0}\in S_{-}$)
and $U_{\alpha}\subseteq{\cal P}^{c}$ is open. Let $y_{n}\in S_{-}$
be a sequence such that $d(x,y_{n})\rightarrow d_{-}.$ Since $A$
is compact, by passing to a subsequence one can find some $z$ with
$y_{n}\rightarrow z$ and thus $d(x,z)=d_{-}.$ It follows from Proposition
\ref{prop:OrbitSetClose} that $C_{-}(x):={\cal O}_{z}$ is a periodic
orbit and one has $C_{-}(x)\in{\cal P}_{-}(x)$. One can identify
another extremal periodic orbit $C_{+}(x)\in{\cal P}_{+}(x)$ in a
similar way. The two orbits $C_{-}(x)$ and $C_{+}(x)$ bound an annulus
whose interior is precisely $U_{\alpha}$. By construction, there
are no periodic orbits inside $U_{\alpha}$.
\end{proof}
\begin{figure} 
\begin{center}  
\includegraphics[scale=0.14]{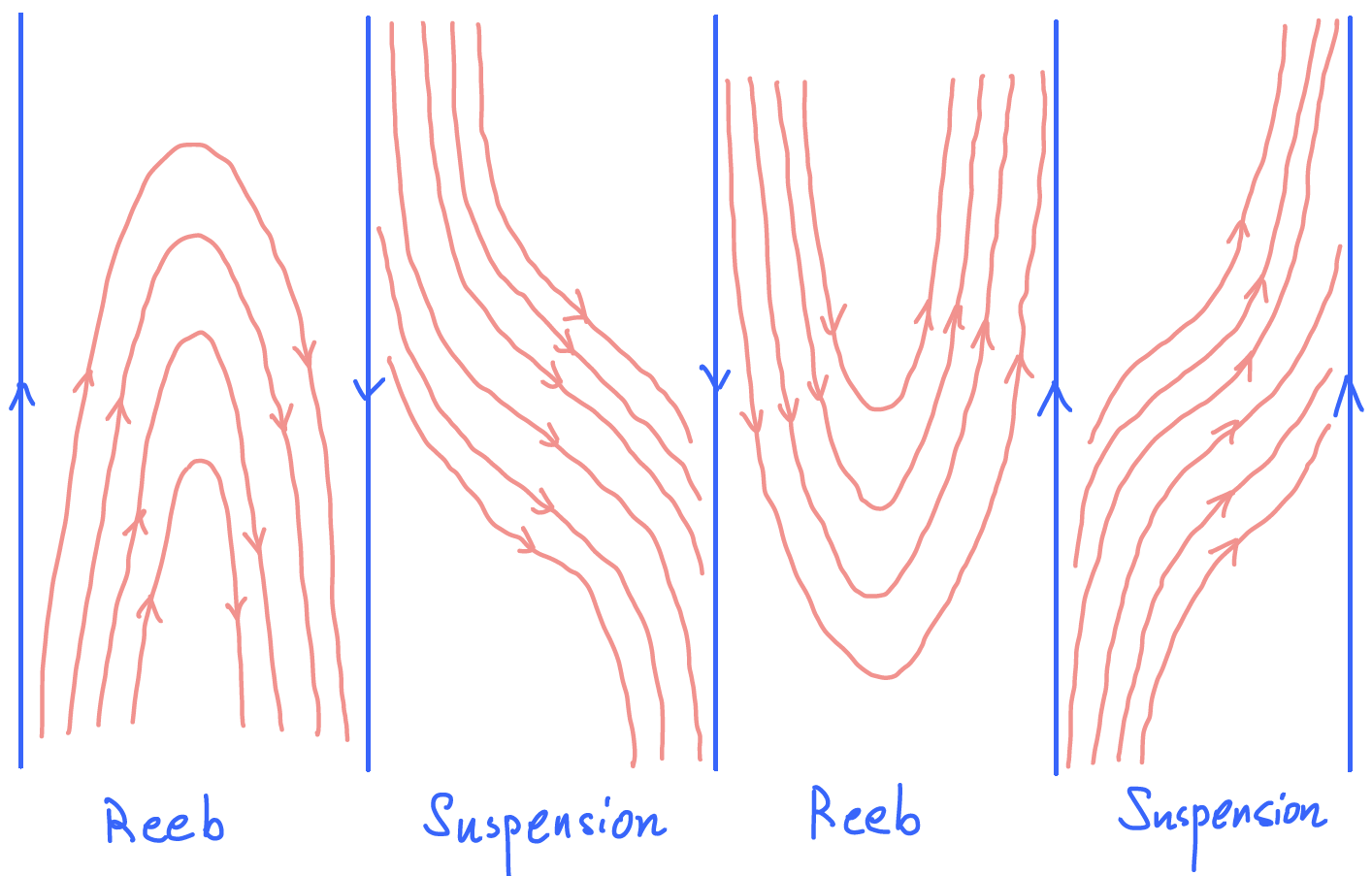} 
\protect\caption{Dynamical picture of torus flow.}\label{fig:torus_dyn}
\end{center}
\end{figure}
Suppose that ${\cal P}$ is not a single periodic orbit nor the whole
torus. Let $\partial U_{\alpha}=C_{1}\sqcup C_{2}$ and let $x\in U_{\alpha}$.
According to Lemma \ref{lem:DistEndPO}, one has either $\alpha(x)=C_{1},$
$\omega(x)=C_{2}$ or $\alpha(x)=C_{2},$ $\omega(x)=C_{1}$. In other
words, every orbit inside $U_{\alpha}$ approaches the two distinct
boundary leaves in both ends. If $C_{1},C_{2}$ have opposite homology
classes, $\overline{U_{\alpha}}$ is a Reeb component. In this case,
it is clear from the proof of Proposition \ref{prop:ReebCrit} that
there is at least another Reeb component inside $\overline{U_{\alpha}}^{c}$. 

On the other hand, if $C_{1},C_{2}$ have the same orientation, $\overline{U_{\alpha}}$
is conjugate to a \textit{suspension annulus}. The notion of a suspension
annulus is similar to a suspension, except that one does not close
the circle. More precisely, let $f:[0,1]\rightarrow[0,1]$ be a diffeomorphism
such that $f(0)=0$ and $f(1)=1$. Define the quotient space
\[
\Xi_{f}:=([0,1]\times\mathbb{R})/_{\sim}
\]
where the equivalence relation $\sim$ is defined by (\ref{eq:SuspEquiv}).
The foliation of vertical leafs $t\mapsto[(x,t)]$ is the suspension
annulus with respect to $f$. Topologically, $\Xi_{f}$ is a compact
annulus whose boundary curves are the two periodic leaves 
\[
t\mapsto[(0,t)],\quad t\mapsto[(1,t)].
\]
Suppose that $\partial U_{\alpha}$ consists of two periodic orbits
with the same homology class. By using the same argument as in Section
\ref{subsec:NoPO}, one can show that $\overline{U_{\alpha}}$ is
conjugate to a suspension annulus. 

If ${\cal P}=C$ is a single periodic orbit, the torus $\mathbb{T}^{2}$
is conjugate to a suspension of some $f\in{\rm Diff}^{+}(S^{1})$
with exactly one periodic point. Every nonperiodic orbit approaches
$C$ in both tails. If ${\cal P}=\emptyset,$ every orbit is nonperiodic
and none of their tails will be trapped in a compact subset of a planar
region. In this case, ${\cal F}_{W}$ is also conjugate to a suspension
whose suspension map $f\in{\rm Diff}^{+}(S^{1})$ does not have periodic
points. 

A typical dynamical picture of the foliation ${\cal F}_{W}$ (viewed on the universal cover $\mathbb{R}^2$) is sketched
in Figure \ref{fig:torus_dyn}. 

\begin{rem}
It is possible to have a continuum family of periodic orbits, e.g.
${\cal P}$ can be a Cantor-like set. 
\end{rem}
\end{appendices}

\section*{Acknowledgement}


S. Wang gratefully acknowledges the support by the
European Research Council (ERC) through the project
\emph{Stochastic Quantum Gauge Theories} (SQGT), Grant Agreement
No.~101116964.

\newpage
\bibliographystyle{plainurl}
\bibliography{bibliography}

@inproceedings{Adl90,
  title={An introduction to continuity, extrema, and related topics for general Gaussian processes},
  author={Adler, Robert J},
  year={1990},
  organization={IMS}
}

@article{BNOT16,
  title={On probability laws of solutions to differential systems driven by a fractional Brownian motion},
  author={Baudoin, F. and Nualart, E. and Ouyang, C. and Tindel, S.},
  journal = {Ann. Probab.},
  pages={2554--2590},
  year={2016},
}

@article{BOT14,
  author = {Baudoin, F. and Ouyang, C. and Tindel, S.},
  title = {Upper bounds for the density of solutions to stochastic differential equations driven by fractional Brownian motions},
  journal = {Ann. Inst. Henri Poincar\'e Probab. Stat.},
  volume = {50},
  number = {1},
  pages = {111--135},
  year = {2014},
}

@article{BKT16,
  title={Gaussian-type lower bounds for the density of solutions of SDEs driven by fractional Brownian motions},
  author={Besal\'u, M. and Kohatsu-Higa, A. and Tindel, S.},
  journal={Ann. Probab.},
  volume={44},
  number={1},
  pages={399--443},
  year={2016},
}

@book{BHOZ08,
  title={Stochastic calculus for fractional Brownian motion and applications},
  author={Biagini, F. and Hu, Y. and {\O}ksendal, B. and Zhang, T.},
  year={2008},
  publisher={Springer},
  edition={}
}

@article{BG23,
  title={$\mathrm{SL}_2(\mathbb{R})$-developments and signature asymptotics for planar paths with bounded variation},
  author={Boedihardjo, H. and Geng, X.},
  journal={Rev. Mat.
Iberoamericana},
  volume={293},
  pages={720--737},
  year={2016},
}

@article{BG24,
  title={On the lack of Gaussian tail for rough line integrals along fractional Brownian paths},
  author={Boedihardjo, H. and Geng, X.},
  journal={Probab. Theory Relat. Fields.},
  volume={188},
  number={3},
  pages={1287--1313},
  year={2024},
}

@article{BGW26,
  title={Optimal tail estimates for differential equations driven by fractional Brownian motion},
  author={Boedihardjo, H. and Geng, X. and Wang, S.},
  journal={ArXiv preprint},
  year={2026},
}

@article{BZ17,
  title={Transverse foliations on the torus $\mathbb{T}^2$ and partially hyperbolic diffeomorphisms on $3$-manifolds},
  author={Bonatti, C. and Zhang, J.},
  journal={Comment. Math. Helv.},
  volume={92},
  pages={513--550},
  year={2017},
}

@book{Bre13,
  title={Topology and geometry},
  author={Bredon, G. E.},
  year={2013},
  publisher={Springer-Verlag}
}

@article {CLL13,
    author = {Cass, T. and Litterer, C. and Lyons, T.},
     title = {Integrability and tail estimates for Gaussian rough differential equations},
     journal = {Ann. Probab.},
     pages = {3026--3050},
     year = {2013},
}

@article{CQ02,
  title={Stochastic analysis, rough path analysis and fractional Brownian motions},
  author={Coutin, L. and Qian, Z.},
  journal={Probab. Theory Relat. Fields.},
  volume={122},
  number={1},
  pages={108--140},
  year={2002},
}

@article{DU99,
  title={Stochastic analysis of the fractional Brownian motion},
  author={Decreusefond, L. and \"Ust\"unel},
  journal={Potential Anal.},
  volume={10},
  pages={177--214},
  year={1999},
}

@article{DPV12,
  title={Hitchhiker's guide to the fractional Sobolev spaces},
  author={Di Nezza, E. and Palatucci, G. and Valdinoci, E.},
  journal={Bull. Sci. Math.},
  volume={136},
  number={5},
  pages={521--573},
  year={2012},
}

@article{ER06,
  title={Variational formulation for the stationary fractional advection dispersion equation},
  author={Ervin, V.J. and Roop, J.P.},
  journal={Numerical Methods for Partial Differential Equations},
  volume={22},
  number={3},
  pages={558--576},
  year={2006},
}

@book{FH14,
  title={A course on rough paths},
  author={Friz, P.K. and Hairer, M.},
  year={2014},
  publisher={Springer},
}

@book{FV10,
  title={Multidimensional stochastic processes as
rough paths},
  author={Friz, P.K. and Victoir, N.},
  year={2010},
  publisher={Cambridge University Press},
  edition={}
}

@book{GP25,
  title={Differential topology},
  author={Guillemin, V. and Pollack, A.},
  volume={370},
  year={2025},
  publisher={American Mathematical Society}
}

@book{HH86,
  title={Introduction to the geometry of foliations. Part A: Foliations on compact surfaces, fundamentals for arbitrary codimension, and holonomy},
  author={Hector, G. and Hirsch, U.},
  year={2013},
  publisher={Springer-Verlag}
}

@book{KMS93,
  title={Fractional integrals and derivatives: theory and applications},
  author={Kilbas, A.A. and Marichev, O.I. and Samko, S.G.},
  year={1993},
  publisher={Gordon and Breach, Switzerland}
}

@article{KS85,
  title={Applications of the Malliavin calculus, Part II},
  author={Kusuoka, S. and Stroock, D.},
  journal={J. Fac. Sci. Univ. Tokyo},
  volume={32},
  pages={1--76},
  year={1985},
}

@article{KS87,
  title={Applications of the Malliavin calculus, Part III},
  author={Kusuoka, S. and Stroock, D.},
  journal={J. Fac. Sci. Univ. Tokyo},
  volume={34},
  pages={391--442},
  year={1987},
}

@incollection{Led94,
  title={Isoperimetry and Gaussian analysis},
  author={Ledoux, M.},
  booktitle={\'Ecole d\'Et\'e de Probabilit\'es de Saint-Flour},
  year={1994},
  publisher={Springer},
}

@book{Lee13,
  title={Introduction to smooth manifolds},
  author={Lee, J.M.},
  year={2013},
  publisher={Springer}
}

@article{Lyo94,
  title={Differential equations driven by rough signals (I): An extension of an inequality of LC Young},
  author={Lyons, T.J.},
  journal={Math. Res. Lett.},
  volume={1},
  number={4},
  year={1994},
}

@article{Lyo98,
  title={Differential equations driven by rough signals},
  author={Lyons, T.J.},
  journal={Rev. Mat. Iberoamericana},
  volume={14},
  numer={2},
  pages={215--310},
  year={1998},
}

@book{LQ02,
  title={System control and rough paths},
  author={Lyons, T.J. and Qian, Z.},
  year={2002},
  publisher={Oxford University Press},
}

@article{NO02,
  title={Regularization of Differential Equations by Fractional Noise},
  author={Nualart, D. and Ouknine, Y.},
  journal={Stoch. Proc. Appl.},
  volume={102},
  pages={103--116},
  year={2002},
}

@book{Per13,
  title={Differential equations and dynamical systems},
  author={Perko, L.},
  year={2013},
  publisher={Springer Science \& Business Media}
}

@incollection{Pic10,
  title={Representation formulae for the fractional Brownian motion},
  author={Picard, J.},
  booktitle={S{\'e}minaire de probabilit{\'e}s XLIII},
  pages={3--70},
  year={2010},
  publisher={Springer}
}

\end{document}